\documentclass[10pt]{amsart}
\usepackage{geometry}
\usepackage[dvips]{graphicx}

\usepackage{amscd,amsfonts,amssymb,amsmath,amsthm,latexsym}

\usepackage{graphics}
\usepackage[all]{xy}
\xyoption{curve}
\xyoption{import}
\xyoption{arc}
\xyoption{ps}

\numberwithin{equation}{section}

\theoremstyle{plain}
    \newtheorem{thm}{Theorem}[section]
    \newtheorem{lemma}[thm]{Lemma}
    \newtheorem{coro}[thm]{Corollary}
    \newtheorem{prop}[thm]{Proposition}
    \newtheorem{problem}[thm]{Problem}
    
\theoremstyle{definition}
    \newtheorem{defi}[thm]{Definition}
    \newtheorem{ex}[thm]{Example}
    \newtheorem{remark}[thm]{Remark}
\theoremstyle{remark}
    \newtheorem{case}{Case}
    \newtheorem{casea}{Case}
    \newtheorem{subcase}{Subcase}
    \newtheorem{subcasea}{Subcase}

\newcommand{\suchthat}{\ | \ }
\newcommand{\field}{K}
\newcommand{\A}{\mathbb{A}}
\newcommand{\Z}{\mathbb{Z}}
\newcommand{\e}{\mathbf{e}}
\newcommand{\g}{\mathbf{g}}
\newcommand{\s}{\operatorname{s}}
\newcommand{\rr}{\operatorname{r}}
\newcommand{\ttt}{\operatorname{t}}
\newcommand{\z}{\operatorname{z}}

\newcommand{\vv}{\operatorname{v}}
\newcommand{\ra}{R\langle\langle Q\rangle\rangle}
\newcommand{\rap}{R\langle\langle Q'\rangle\rangle}
\newcommand{\usualra}{R\langle Q\rangle}
\newcommand{\idealM}{\mathfrak{m}}

\newcommand{\jacobas}{\mathcal{P}(Q,S)}
\newcommand{\jacobapsp}{\mathcal{P}(Q',S')}
\newcommand{\jacobqstau}{\mathcal{P}(Q(\tau),S(\tau))}
\newcommand{\astrivial}{(Q_{\operatorname{triv}},S_{\operatorname{triv}})}
\newcommand{\asreduced}{(Q_{\operatorname{red}},S_{\operatorname{red}})}

\newcommand{\surf}{(\Sigma, M)}
\newcommand{\punct}{P}
\newcommand{\marked}{M}
\newcommand{\arcsinsurf}{\mathbf{A}^\circ\surf}

\newcommand{\qtau}{Q(\tau)}
\newcommand{\hatqtau}{\widehat{Q}(\tau)}
\newcommand{\qsigma}{Q(\sigma)}
\newcommand{\atau}{Q(\tau)}
\newcommand{\asigma}{Q(\sigma)}
\newcommand{\ratau}{R\langle\langle Q(\tau)\rangle\rangle}
\newcommand{\stau}{S(\tau)}

\newcommand{\qstau}{(Q(\tau),S(\tau))}
\newcommand{\qssigma}{(Q(\sigma),S(\sigma))}

\newcommand{\ssigma}{S(\sigma)}
\newcommand{\sptau}{S'(\tau)}

\newcommand{\surfnoM}{\Sigma}
\newcommand{\premutj}{\widetilde{\mu}_j}
\newcommand{\mutj}{\mu_j}
\newcommand{\premuti}{\widetilde{\mu}_j}
\newcommand{\muti}{\mu_j}

\newcommand{\tildestau}{\widetilde{S(\tau)}}

\newcommand{\tildeqtau}{\widetilde{Q(\tau)}}

\newcommand{\unredqtau}{\widehat{Q}(\tau)}
\newcommand{\unredstau}{\widehat{S}(\tau)}
\newcommand{\unredqsigma}{\widehat{Q}(\sigma)}

\newcommand{\runredatau}{R\langle\langle\widehat{Q}(\tau)\rangle\rangle}
\newcommand{\unredastau}{(\widehat{Q}(\tau),\widehat{S}(\tau))}

\newcommand{\qpmod}{\mathcal{M}=(Q,S,M,V)}
\newcommand{\qpmodp}{\mathcal{M}'=(Q',S',M',V')}
\newcommand{\qprep}{(Q,S,(M_i)_{i\in Q_0},(a_M)_{a\in Q_1},(V_i)_{i\in Q_0})}
\newcommand{\ored}{\operatorname{red}}
\newcommand{\oin}{\operatorname{in}}
\newcommand{\oout}{\operatorname{out}}
\newcommand{\al}{\mathfrak{a}}
\newcommand{\be}{\mathfrak{b}}
\newcommand{\ga}{\mathfrak{c}}
\newcommand{\image}{\operatorname{im}}

\newcommand{\id}{\mathbf{1}}
\newcommand{\zero}{\mathbf{0}}
\newcommand{\rh}{\mathfrak{r}}
\newcommand{\si}{\mathfrak{s}}
\newcommand{\io}{\mathfrak{i}}
\newcommand{\pipi}{\mathfrak{p}}
\newcommand{\rtildea}{R\langle\langle\widetilde{Q}\rangle\rangle}

\newcommand{\idM}{\mathbf{1}_M}
\newcommand{\idV}{\mathbf{1}_V}
\newcommand{\idK}{\mathbf{1}_K}
\newcommand{\badintersection}{\mathfrak{B}}

\newcommand{\arc}{i}
\newcommand{\calMsigmaarc}{\mathcal{M}(\sigma,\arc)}
\newcommand{\arcone}{j}
\newcommand{\arctwo}{k}
\newcommand{\calMtauarc}{\mathcal{M}(\tau,\arc)}
\newcommand{\Mtauarc}{M(\tau,\arc)}
\newcommand{\Msigmaarc}{M(\sigma,\arc)}
\newcommand{\mtauarc}{m(\tau,\arc)}
\newcommand{\msigmaarc}{m(\sigma,\arc)}
\newcommand{\Vtauarc}{V(\tau,\arc)}
\newcommand{\triv}{\operatorname{triv}}

\begin{document}

\title[Arc representations associated to triangulated surfaces]{Quivers with potentials associated to triangulated surfaces, Part II: Arc representations}
\author{Daniel Labardini-Fragoso}
\address{Department of Mathematics, Northeastern University, Boston, MA 02115}
\email{labardini-fra.d@neu.edu}
\date{\today}
\thanks{This work was partially supported by Prof. Andrei Zelevinsky's NSF grant and Prof. Martha Takane's PAPIIT-UNAM grant IN103508.}
\dedicatory{To the memory of Jos\'{e} Guadalupe Ram\'irez-Rocha.}
\maketitle

\begin{abstract} This paper is a representation-theoretic extension of Part I. It has
been inspired by three recent developments: surface cluster algebras
studied by Fomin-Shapiro-Thurston, the mutation theory of quivers with
potentials initiated by Derksen-Weyman-Zelevinsky, and string modules
associated to arcs on unpunctured surfaces by
Assem-Br\"ustle-Charbonneau-Plamondon.
Modifying the latter construction, to each arc and each ideal
triangulation of a bordered marked surface we associate in an explicit
way a representation of the quiver with potential constructed in Part
I, so  that whenever two ideal triangulations are related by a flip,
the associated representations are related by the corresponding
mutation.
\end{abstract}

\tableofcontents

\section{Introduction}

This paper is a representation-theoretic extension of its predecessor \cite{Lqps}. The aim remains the same: to study the relation between the representation-theoretic approach to cluster algebras developed by H. Derksen, J. Weyman and A. Zelevinsky using mutations of quivers with potentials, and the framework for surface cluster algebras established by S. Fomin, M. Shapiro and D. Thurston. In \cite{DWZ}, \cite{DWZ2}, the authors defined certain QP-representations whose geometric/topological data determines the Laurent expansions of cluster variables in terms of a given cluster. In this work we explicitly compute many of these representations when the cluster algebra is associated to a bordered surface with marked points.

In \cite{Lqps}, the author defined a quiver with potential for each ideal triangulation of a bordered surface with marked points, and proved the compatibility between flips on ideal triangulations and mutations of QPs, thus partially extending the compatibility flip $\leftrightarrow$ mutation shown in \cite{FG}, \cite{FST} and \cite{GSV}. Moreover, it was proved that, as long as the boundary of the surface is non-empty, the associated QPs are rigid and hence non-degenerate, and that the corresponding Jacobian algebra is finite-dimensional (thus making possible to apply C. Amiot's categorification \cite{Amiot-gldim2}).

Here we associate to arc $\arc$ and each ideal triangulation $\tau$ on the surface, a representation $\Mtauarc$ of the QP $\qstau$ defined in \cite{Lqps}, in such a way that, if the arc $\arc$ is kept fixed, ideal triangulations related by a flip give rise to representations related by the corresponding QP-mutation (cf. Theorem \ref{thm:flip<->mut} below). Our construction generalizes that of \emph{string} modules for curves in unpunctured surfaces that has been given in \cite{ABCP} by I. Assem, T. Br\"{u}stle, G. Charbonneau and P-G. Plamondon. It is worth mentioning that Assem-Br\"{u}stle-Charbonneau-Plamondon's string modules are already a generalization of a construction given by P. Caldero, F. Chapoton and R. Schiffler in \cite{CCS04} for unpunctured polygons. When the bordered marked surface has no punctures, ABCP/CCS's construction could be (very roughly) described as ``traversing curves with the identity map of the field $\field$". In the presence of punctures the situation becomes more complicated, but \emph{strings} still function as strong combinatorial parameters for representations that are mutation-equivalent to negative simples.

The representations $\Mtauarc$ turn out to be those representations whose geometric data gives Laurent expansions of cluster variables. Let us be more precise: On the one hand, in \cite{DWZ2} it is shown that given a quiver $Q$ and a non-degenerate potential $S$ on it, the $(Q,S)$-representations mutation equivalent to negative simple ones are naturally associated to the cluster variables of any cluster algebra having the quiver $Q$ at one of its seeds. Furthermore, it is proved that the Euler characteristics of the quiver Grassmannians of these $(Q,S)$-representations are the coefficients of the $F$-polynomials associated in \cite{FZ4} to the corresponding cluster variables. On the other hand, in \cite{FST}, \cite{FT}, a geometric-combinatorial model has been given for the cluster algebras having a quiver of the form $\qtau$ at one of its seeds, where $\tau$ is an ideal (or tagged) triangulation of a bordered surface with marked point. In this model, cluster variables are (tagged) arcs on the surface, and mutation of seeds corresponds to flips of (tagged) arcs. By results of \cite{Lqps}, $S(\tau)$ is a non-degenerate potential on $Q(\tau)$ whenever the underlying surface has non-empty boundary or exactly one puncture, and by Theorem \ref{thm:flip<->mut} below, the representations $\Mtauarc$ are mutation-equivalent to negative simple ones. Therefore, these representations $\Mtauarc$ are \emph{the} representations that can be used to calculate the $F$-polynomials of the cluster variables (in the \emph{positive stratum}) of (any) cluster algebra associated to $\qtau$. (By results of \cite{FZ4}, cluster dynamics is governed to a great extent by $\g$-vectors and $F$-polynomials.)

Let us proceed to describe the contents of the paper in more detail. Section \ref{Section:background} is divided into three parts, the first two of which are devoted to establish some notation and terminology, taken from \cite{DWZ}, about QPs and their representations. In Subsection \ref{subsec:restriction} we recall the operation of restriction of QPs and its behavior with respect to mutation, then we define the operation of restriction of a QP-representation to a subset of the vertex set of the underlying quiver in the obvious way, and prove that it commutes with the operations of reduction, premutation and mutation of representations as long as the vertex subset $I\subseteq Q_0$ satisfies certain vanishing condition. This will help us to reduce the proof of our main result to the situation where the surface has empty boundary.

In Section \ref{Section:QPoftriangulation} we recall the basic properties of flips of triangulations, the definition of the potential $\stau$ associated to an ideal triangulation $\tau$, and the compatibility between flips of triangulations and mutations of QPs. In Section \ref{repsforarcsspecial} we present the main constructions of this paper: to start, we draw certain short oriented curves on the surface, which we call the \emph{detours} of $(\tau,\arc)$, and arrange some information extracted from these curves into a family of \emph{detour matrices} $D^{\triangle}_{\arc,\arcone}$, $\arcone\in\tau$; then we define the \emph{segment representations} $\mtauarc$ following Assem-Br\"ustle-Charbonneau-Plamondon/Caldero-Chapoton-Schiffler, and modify them using detour matrices, thus obtaining the \emph{arc representations} $\Mtauarc$, which are the main objects of study of the current note. Section \ref{repsforarcsspecial} is divided into two subsections since there is a specific situation where the arc $\arc$ needs to be cut in order to reach the flip $\leftrightarrow$ mutation compatibility of its associated representations. In Subsection \ref{case1.1} we present the case where $\arc$ does not need to be cut;
and the case where $\arc$ needs to be cut, namely, when it is a loop cutting out a once-punctured monogon, is dealt with in Subsection \ref{subsec:cutsout}.

Section \ref{Section:Jacobianidealsatisfied} consists of two parts, the first of which, Subsection \ref{subsec:localdecompositions}, is devoted to show that mutations of representations preserve not only direct sums, but also \emph{local} direct sums, and to \emph{locally decompose} the representations $\Mtauarc$ into \emph{simpler} representations, where it is easier to carry out the several checks of the main results. In Subsection \ref{subsec:mtauarcsatisfies} we prove that the arc representations satisfy the relations imposed on $\qtau$ by the Jacobian ideal $J(S(\tau))$.

In Section \ref{Section:Flip-mutationcompatibility} we present Theorem \ref{thm:flip<->mut}, which is the main result of this paper: If the arc $\arc$ is fixed and we have two ideal triangulations (without self-folded triangles) $\tau$ and $\sigma$ related by a flip, then the arc representations $\Mtauarc$ and $\Msigmaarc$ are related by the corresponding mutation of representations. The section starts with Subsection \ref{subsec:effectonmatrices}, where we verify that the linear maps attached by an arc representation $\Mtauarc$ to an arrow of $\qtau$ not incident at the arc $\arcone$ to be flipped do not change when we perform the flip $f_\arcone$ (we know that they should not change by definition of the mutation $\mu_{\arcone}$). In Subsection \ref{subsec:proof}, we analyze the behavior of the linear maps attached to the arrows that are incident to the arc $\arcone$ in the configurations obtained in Subsection \ref{subsec:localdecompositions}.
We close Section \ref{Section:Flip-mutationcompatibility} with a quite informal overview of the constructions and results of the Section.

An application of arc representations is given in Section \ref{Section:gvectors}, where we give a very simple formula to calculate the $\g$-vector of an (ordinary) arc with respect to an ideal triangulation $\surf$.

In section \ref{Section:problems} we mention some problems that remain open and whose solution the author believes would help to have a complete explicit exemplification of Derksen-Weyman-Zelevinsky's QP-mutation theory in the context of surface cluster algebras.

The context in which this paper takes place deserves some comments. Triangulations and flips have been present in models of cluster algebras since the beginning of the theory (see, e.g., Subsection 12.1 of \cite{FZ2}), and also in the effort of categorifying these algebras (cf. \cite{CCS04}, \cite{S}). This had been done in a somewhat restricted set-up, until signed adjacency quivers for ideal triangulations of arbitrary bordered surfaces with marked points, and the compatibility between the operations of flip on ideal triangulations and ordinary quiver mutation, appeared in works by V. Fock and A Goncharov \cite{FG}; S. Fomin, M. Shapiro and D. Thurston \cite{FST}; and M. Gekhtman, M. Shapiro and A. Vainshtein \cite{GSV}. This yielded a general realization of ideal triangulations as clusters in the cluster algebras whose exchange matrices are determined by the signed adjacency quivers of ideal triangulations. However, not all clusters in such a cluster algebra could be interpreted as ideal triangulations. In \cite{FST}, S. Fomin, M. Shapiro and D. Thurston introduced the notions of tagged triangulations and their signed adjacency quivers, proving that all arcs in a tagged triangulation are flippable and the corresponding compatibility between flips and ordinary mutations, thus realizing all clusters as tagged triangulations, and seed mutations as flips on tagged triangulations. In \cite{FT}, S. Fomin and D. Thurston have deepened this realization further to interpret cluster variables in terms of R. Penner's lambda-lengths and coefficients in terms of W. Thurston's unbounded measured integral laminations.

A similar story can be said about the mutation properties of cluster-tilted algebras. In \cite{DWZ}, H. Derksen, J. Weyman and A. Zelevinsky defined mutations of quivers with potentials and their representations, thus providing a new representation-theoretic interpretation for quiver mutations originated in the theory of cluster algebras, in a way that generalizes the classical Bernstein-Gelfand-Ponomarev reflection functors. The depth and importance of Derksen-Weyman-Zelevinsky's QP-mutation theory in both Representation theory and Cluster algebra theory has manifested in several recent works, one of which is \cite{DWZ2}, where the same group of authors makes a heavy use of non-degenerate potentials to prove several conjectures from \cite{FZ4}.

After the foundational papers \cite{FG}, \cite{FST} and \cite{GSV} on the one hand, and \cite{DWZ} on the other, both surface cluster algebras and quivers with potentials have attracted the attention of several authors (e.g., \cite{Amiot-gldim2}, \cite{ABCP}, \cite{FT}, \cite{Lqps}, \cite{MSmatchings}, \cite{MSW}, \cite{PMaster}, \cite{Sunpunct2}, \cite{STunpunct1}) for many different reasons. Some of their works are directly related to \cite{Lqps} and the present note. For instance, in \cite{ABCP}, Assem-Br\"{u}stle-Charbonneau-Plamondon define, for each triangulation $\tau$ of an unpunctured surface, a gentle quotient of $\qtau$ which is precisely the Jacobian algebra of the potential $S(\tau)$ defined independently in \cite{Lqps}; they also show that there is a one-to-one correspondence between the string modules, and the isotopy classes of arcs on the surface. These modules turn out to be our arc representations in the unpunctured set-up.

Another related paper is \cite{MSW}, where, working in the full generality of \emph{tagged triangulations} of punctured surfaces, G. Musiker, R. Schiffler and L. Williams give explicit formulas for the expansion of an arbitrary cluster variable (that is, \emph{tagged arc} on the surface) in terms of an arbitrary initial cluster (that is, tagged triangulation), thus establishing, for example, the \emph{positivity conjecture} of S. Fomin and A. Zelevinsky \cite{FZ4}, independently of the much more general approach of \cite{DWZ2}.

Finally, let us mention that in C. Amiot's categorification context \cite{Amiot-gldim2}, each arc on $\surf$ represents an object of the \emph{cluster category} $\mathcal{C}$, and each triangulation $\tau$ represents a \emph{cluster-tilting object} whose endomorphism algebra is precisely the Jacobian algebra $\mathcal{P}(Q(\tau),S(\tau))$; moreover, for each fixed triangulation there is a functor from $\mathcal{C}$ to the module category of the Jacobian algebra of the triangulation. As a consequence of Theorem \ref{thm:flip<->mut} below, the arc representation $\Mtauarc$ gives an explicit calculation of the image of $\arc$ under the functor $\mathcal{C} \rightarrow \operatorname{mod}\mathcal{P}(Q(\tau),S(\tau))$.

\section{Background on QP-representations and their mutations}\label{Section:background}

\subsection{Quivers with potentials and their mutations}

With the purpose of recalling some notation and terminology, in this subsection we briefly review the basics of the mutation theory of quivers with potentials initiated in \cite{DWZ}. For a way more detailed and elegant treatment of the subject, we refer the reader to that paper. A short survey of these topics can be found in \cite{Z}.

A \emph{quiver} is a finite directed graph, that is, a quadruple $Q=(Q_0,Q_1,h,t)$, where $Q_0$ is the (finite) set of
\emph{vertices} of $Q$, $Q_1$ is the (finite) set of \emph{arrows}, and $h:Q_1\rightarrow Q_0$ and $t:Q_1\rightarrow Q_0$ are the \emph{head} and \emph{tail} functions. A common notation to indicate that $a$ is an arrow of $Q$ with $t(a)=i$, $h(a)=j$, is $a:i\rightarrow j$ .

A \emph{path of length} $d>0$, or \emph{$d$-path}, in $Q$ is a sequence $a_1a_2\ldots a_d$ of arrows with $t(a_j)=h(a_{j+1})$ for $j=1,\ldots,d-1$. A path
$a_1a_2\ldots a_d$ of length $d>0$ is a $d$\emph{-cycle} if $h(a_1)=t(a_d)$. A quiver is \emph{2-acyclic} if it has no 2-cycles.

Paths are composed as functions, that is, if $a=a_1\ldots a_d$ and $b=b_1\ldots b_{d'}$ are paths with $h(b)=t(a)$, then the product $ab$ is defined as the path $a_1,\ldots a_db_1\ldots b_{d'}$, which starts at $t(b_{d'})$ and ends at $h(a_1)$. See Figure \ref{prodofpaths}.

 \begin{figure}[!h]
                \caption{Paths are composed as functions}\label{prodofpaths}
                \centering
$$
\bullet\overset{b_{d'}}{\longrightarrow}\ldots\overset{b_1}{\longrightarrow}\bullet\overset{a_d}{\longrightarrow}\ldots\overset{a_1} {\longrightarrow}\bullet
$$
 \end{figure}

\begin{defi}\label{threesteps} Given a quiver $Q$ and a vertex $j\in Q_0$ such that $Q$ has no $2$-cycles incident at $j$, we define the
\emph{mutation} of $Q$ in direction $j$ as the quiver $\muti(Q)$ with vertex set $Q_0$ that results after applying the following three-step
procedure:
\begin{itemize}
\item[(Step 1)] For each $j$-hook $ab$ introduce an arrow $[ab]:t(b)\rightarrow h(a)$ (a \emph{$j$-hook} is a 2-path whose middle vertex is $j$).
\item[(Step 2)] Replace each arrow $a:j\rightarrow h(a)$ of $Q$ with an arrow $a^*:h(a)\rightarrow j$, and each arrow $b:t(b)\rightarrow j$
of $Q$ with an arrow $b^*:j\rightarrow t(b)$.
\item[(Step 3)] Choose a maximal collection of disjoint 2-cycles and remove them.
\end{itemize}
We call the quiver obtained after the $1^{\operatorname{st}}$ and $2^{\operatorname{nd}}$ steps the \emph{premutation} $\premuti(Q)$.
\end{defi}

Given a quiver $Q$, the \emph{complete path algebra} $R\langle\langle Q\rangle\rangle$ is the $\field$-vector space consisting of all possibly infinite linear combinations of paths in $Q$, with multiplication induced by concatenation of paths (cf. \cite{DWZ}, Definition 2.2). Note that $\ra$ is a $\field$-algebra and has the usual \emph{path
algebra} $\usualra$ as a dense subalgebra under the $\idealM$-adic topology, whose fundamental system
of open neighborhoods around $0$ is given by the powers of the ideal $\idealM$ of $\ra$ generated by the arrows.

A \emph{potential} on $Q$ is any element of $\ra$ all of whose terms are cyclic paths of positive length (cf. \cite{DWZ}, Definition 3.1). Two potentials $S,S'\in\ra$ are \emph{cyclically equivalent} if $S-S'$
lies in the topological closure of the vector subspace of $\ra$ spanned by all the elements of the form $a_1\ldots a_d-a_2\ldots a_da_1$ with $a_1\ldots a_d$ a cyclic path of positive length (cf. \cite{DWZ}, Definition 3.2).

A \emph{quiver with potential} is a pair $(Q,S)$, where $S$ is a potential on $Q$ such that no two different cyclic paths appearing in the expression of $S$ are cyclically equivalent (cf. \cite{DWZ}, Definition 4.1). Following \cite{DWZ}, we use the shorthand \emph{QP} to abbreviate ``quiver with potential".

If $(Q,S)$ and $(Q',S')$ are QPs on the same set of vertices, we say that $(Q,S)$ is \emph{right-equivalent} to $(Q',S')$ if there exists a \emph{right-equivalence} between them, that is, an $K$-algebra isomorphism $\varphi:\ra\rightarrow\rap$ that fixes the idempotents associated to the vertices and such that $\varphi(S)$ is cyclically
equivalent to $S'$ (cf. \cite{DWZ}, Definition 4.2).

For each arrow $a\in Q_1$ and each cyclic path $a_1\ldots a_d$ in $Q$ we define the \emph{cyclic derivative}
\begin{equation}
\partial_a(a_1\ldots a_d)=\underset{i=1}{\overset{d}{\sum}}\delta_{a,a_i}a_{i+1}\ldots a_da_1\ldots a_{i-1}
\end{equation}
(where $\delta_{a,a_i}$ is the \emph{Kronecker delta}) and extend $\partial_a$ by linearity and continuity to the space of all potentials (cf. \cite{DWZ}, Definition 3.1). Note that
$\partial_a(S)=\partial_a(S')$ whenever the potentials $S$ and $S'$ are cyclically equivalent.

For a potential $S$, the \emph{Jacobian ideal} $J(S)$ is the closure of the two-sided ideal of $\ra$ generated by $\{\partial_a(S)\suchthat a\in Q_1\}$, and the \emph{Jacobian algebra} $\jacobas$ is the quotient algebra $\ra/J(S)$ (cf. \cite{DWZ}, Definition 3.1). Jacobian ideals and Jacobian algebras are invariant under right-equivalences, in the sense that if $\varphi:\ra\rightarrow\rap$ is a right-equivalence between $(Q,S)$ and $(Q',S')$, then $\varphi$
sends $J(S)$ onto $J(S')$ and therefore induces an isomorphism $\jacobas\rightarrow\jacobapsp$ (cf. \cite{DWZ}, Proposition 3.7).

A QP $(Q,S)$ is \emph{trivial} if all the summands of $S$ are 2-cycles and $\{\partial_a(S)\suchthat a\in Q_1\}$ spans the vector subspace of $\ra$ generated by the arrows of $Q$ (cf. \cite{DWZ}, Definition 4.3, see also Proposition 4.4 therein). We say that a QP $(Q,S)$ is \emph{reduced} if the degree-$2$ component of $S$ is $0$, that is,
if the expression of $S$ involves no $2$-cycles (cf. \cite{DWZ}, the paragraph preceding Theorem 4.6). Note that the underlying quiver of a reduced QP may have 2-cycles. We say that a quiver $Q$ (or any QP on it) is \emph{2-acyclic} if it has no $2$-cycles.

The \emph{direct sum} of two QPs $(Q,S)$ and $(Q',S')$ on the same set of vertices is the QP $(Q,S)\oplus(Q',S')=(Q\oplus Q',S+S')$, where $Q\oplus Q'$ is the quiver on the vertex set $Q_0=Q'_0$ whose set of arrows is the disjoint union of $Q_1$ and $Q'_1$ (cf. \cite{DWZ}, Section 4).

\begin{thm}[Splitting Theorem, \cite{DWZ}, Theorem 4.6]\label{splittingthm} For every QP $(Q,S)$ there exist a trivial QP
$\astrivial$ and a reduced QP $\asreduced$ such that $(Q,S)$ is right-equivalent to the direct sum $\astrivial\oplus\asreduced$. Furthermore,
the right-equivalence class of each of the QPs $\astrivial$ and $\asreduced$ is determined by the right-equivalence class of $(Q,S)$.
\end{thm}

In the situation of Theorem \ref{splittingthm}, the QP $\asreduced$ (resp. $\astrivial$) is called the \emph{reduced part} (resp. \emph{trivial
part}) of $(Q,S)$ (cf. \cite{DWZ}, Definition 4.13); this terminology is well defined up to right-equivalence.

We now turn to the definition of mutation of a QP. Let $(Q,S)$ be a QP on the vertex set $Q_0$ and let $j\in Q_0$. Assume that $Q$ has no 2-cycles incident to $j$. If necessary, we replace $S$ with a cyclically equivalent potential so that we can assume that every cyclic path appearing in the expression of $S$ does not begin at $j$. This allows us to define $[S]$ as the potential on $\widetilde{\mu}_j(Q)$ obtained from $S$ by replacing each $j$-hook $ab$ with the arrow $[ab]$ (a \emph{$j$-hook} is a 2-path whose middle vertex is $j$). Also, we define $\Delta_j(Q)=\sum b^*a^*[ab]$, where the sum runs over all $j$-hooks $ab$ of $Q$.

\begin{defi}[\cite{DWZ}, equations (5.3) and (5.8) and Definition 5.5] Under the assumptions and notation just stated, we define the \emph{premutation} of $(Q,S)$ in direction $j$ as the QP
$\premuti(Q,S)=(\widetilde{\mu}_j(Q),\widetilde{S})$, where
$\widetilde{S}=[S]+\Delta_j(Q)$. The \emph{mutation} $\muti(Q,S)$ of $(Q,S)$ in direction $j$ is then defined as the reduced part of $\premuti(Q,S)$.
\end{defi}

\begin{thm}[\cite{DWZ}, Theorem 5.2 and Corollary 5.4] Premutations and mutations are well defined up to right-equivalence. That is, if $(Q,S)$ and $(Q',S')$ are right-equivalent QPs with no 2-cycles incident to the vertex $j$, then $\premuti(Q,S)$ is right-equivalent to $\premuti(Q',S')$ and the $\muti(Q,S)$ is right-equivalent to $\muti(Q',S')$.
\end{thm}

\begin{thm}[\cite{DWZ}, Theorem 5.7] Mutations are involutive up to right-equivalence. More specifically, if $(Q,S)$ a 2-acyclic QP, then $\muti^2(Q,S)$ is right-equivalent to $(Q,S)$.
\end{thm}

\begin{defi}[\cite{DWZ}, Definition 7.2] A QP $(Q,S)$ is \emph{non-degenerate} if it is 2-acyclic and the quiver of the QP obtained after any possible sequence of QP-mutations is 2-acyclic.
\end{defi}

\begin{thm}[\cite{DWZ}, Proposition 7.3 and Corollary 7.4] If the base field $\field$ is uncountable, then every 2-acyclic quiver admits a non-degenerate QP.
\end{thm}

A QP $(Q,S)$ is \emph{rigid} if every cycle in $Q$ is cyclically equivalent to an element of the Jacobian ideal $J(S)$ (cf. \cite{DWZ}, Definition 6.10 and equation 8.1). Rigidity is invariant under QP-mutation.

\begin{thm}[\cite{DWZ}, Corollary 6.11, Proposition 8.1 and Corollary 8.2] Every reduced rigid QP is 2-acyclic. The class of reduced rigid QPs is closed under QP-mutation. Consequently, every rigid reduced QP is non-degenerate.
\end{thm}

\begin{prop}[\cite{DWZ}, Corollary 6.6]\label{findimisinvariant} Let $(Q,S)$ be a non-degenerate QP and $j\in Q_0$ any vertex, then the Jacobian algebra $\mathcal{P}(Q,S)$ is finite-dimensional if and only if so is $\mathcal{P}(\muti(Q,S))$. In other words, finite-dimensionality of Jacobian algebras is invariant under QP-mutations.
\end{prop}

\subsection{Decorated representations and their mutations}\label{backgroundrepresentations}

In this subsection we describe how the notions of right-equivalence and QP-mutation extend to the level of representations. As in the previous subsection, our reference is \cite{DWZ}.

Recall that the \emph{vertex span} of a quiver $Q$ is the $\field$-vector space $R$ with basis $\{e_i\suchthat i\in Q_0\}$. This vector space is actually a commutative ring if we define $e_ie_j=\delta_{ij}e_i$.

\begin{defi}[\cite{DWZ}, Definition 10.1] Let $(Q,S)$ be any QP. A \emph{decorated} $(Q,S)$\emph{-representation}, or \emph{QP-representation}, is a quadruple $\mathcal{M}=(Q,S,M,V)$, where $M$ is a finite-dimensional left $\jacobas$-module and $V$ is a finite-dimensional left $R$-module.
\end{defi}

By setting $M_i=e_iM$ for each $i\in Q_0$, and $a_M:M_{t(a)}\rightarrow M_{h(a)}$ as the multiplication by $a\in Q_1$ given by the $\ra$-module structure of $M$, we easily see that each $\mathcal{P}(Q,S)$-module induces a representation of the quiver $Q$. The following lemma, whose proof can be found in \cite{DWZ}, allows us to deduce the relations this representation satisfies.

\begin{lemma} Every finite-dimensional $\ra$-module is nilpotent. That is, if $M$ is a finite-dimensional $\ra$-module, then there exists a positive integer $r$ such that $\idealM^rM=0$. (Remember that $\idealM$ is the ideal of $R\langle\langle Q\rangle\rangle$ generated by the arrows.)
\end{lemma}

Because of this lemma, we see that any QP-representation is prescribed by the following data:
\begin{enumerate}\item A tuple $(M_i)_{i\in Q_0}$ of finite-dimensional $K$-vector spaces;
\item a family $(a_M:M_{t(a)}\rightarrow M_{h(a)})_{a\in Q_0}$ of $K$-linear transformations annihilated by $\{\partial_a(S)\suchthat a\in Q_0\}$, for which there exists an integer $r\geq 1$ with the property that the composition ${a_1}_{M}\ldots {a_r}_{M}$ is identically zero for every $r$-path $a_1\ldots a_r$ in $Q$.
\item a tuple $(V_i)_{i\in Q_0}$ of finite-dimensional $\field$-vector spaces (without any specification of linear maps between them).
\end{enumerate}

\begin{remark} In the literature, the linear map $a_M:M_{t(a)}\rightarrow M_{h(a)}$ induced by left multiplication by $a$ is more commonly denoted by $M_a$. We will use both of these notations indistinctly.
\end{remark}

\begin{defi}[\cite{DWZ}, Definition 10.2] Let $(Q,S)$ and $(Q',S')$ be QPs on the same set of vertices, and let $\mathcal{M}=(Q,S,M,V)$ and $\mathcal{M}'=(Q',S',M',V')$ be decorated representations. A triple $\Phi=(\varphi,\psi,\eta)$ is called a \emph{right-equivalence} between $\mathcal{M}$ and $\mathcal{M}'$ if the following three conditions are satisfied:
\begin{itemize}\item $\varphi:\ra\rightarrow\rap$ is a right-equivalence of QPs between $(Q,S)$ and $(Q',S')$;
\item $\psi:M\rightarrow M'$ is a vector space isomorphism such that $\psi\circ u_M=\varphi(u)_{M'}\circ\psi$ for all $u\in\ra$;
\item $\eta:V\rightarrow V'$ is an $R$-module isomorphism.
\end{itemize}
\end{defi}

\begin{ex} Consider the QP $(Q,0)$, where $Q$ is the quiver
\begin{displaymath}
\xymatrix{ & 2 \ar[dr]^b & \\
1 \ar[ur]^c \ar[rr]_a & & 3 }
\end{displaymath}
For any $\lambda\in\field$, the QP-representation
\begin{displaymath}
\xymatrix{ & \field \ar[dr]^{\idK} & & & V_2 &  \\
\field \ar[ur]^{\idK} \ar[rr]_{\lambda\idK} & & \field & V_1 & & V_3 }
\end{displaymath}
is right-equivalent to the QP-representation
\begin{displaymath}
\xymatrix{ & \field \ar[dr]^{\idK} & & & V_2 &  \\
\field \ar[ur]^{\idK} \ar[rr]_{0} & & \field & V_1 & & V_3 }
\end{displaymath}
by means of the triple $\Phi=(\varphi,\psi,\idV)$, where $\varphi:\ra\rightarrow\ra$ is the $R$-algebra isomorphism whose action on the arrows is given by $a\mapsto \lambda bc-a$, $b\mapsto b$, $c\mapsto c$, and $\psi$ is the identity on each copy of $\field$. This example shows in particular that there are right-equivalent representations that are not isomorphic.
\end{ex}

Recall that every QP is right-equivalent to the direct sum of its reduced and trivial parts, which are determined up to right-equivalence (Theorem \ref{splittingthm}). These facts have representation-theoretic extensions, which we now describe. Let $(Q,S)$ be
any QP, and let $\varphi:R\langle\langle Q_{\ored}\oplus C\rangle\rangle\rightarrow\ra$ be a right equivalence between
$(Q_{\ored},S_{\ored})\oplus(C,T)$ and $(Q,S)$, where $(Q_{\ored},S_{\ored})$ is a reduced QP and $(C,T)$ is a trivial QP. Let
$\mathcal{M}=(Q,S,M,V)$ be a decorated representation, and set $M^\varphi=M$ as $K$-vector space. Define an action of $R\langle\langle
Q_{\ored}\rangle\rangle$ on $M^\varphi$ by setting $u_{M^\varphi}=\varphi(u)_M$ for $u\in R\langle\langle
Q_{\ored}\rangle\rangle$.

\begin{prop}[\cite{DWZ}, Propositions 4.5 and 10.5]\label{reddetermineduptorequiv} With the action of $R\langle\langle Q_{\ored}\rangle\rangle$ on $M^\varphi$ just defined, the quadruple $(Q_{\ored},S_{\ored},M^\varphi,V)$ becomes a QP-representation. Moreover, the right-equivalence class of $(Q_{\ored},S_{\ored},M^\varphi,V)$ is determined by the right-equivalence class of $\mathcal{M}$.
\end{prop}

\begin{defi}[\cite{DWZ}, Definition 10.4]\label{reducedpartofrep} The (right-equivalence class of the) QP-representation $\mathcal{M}_{\ored}=(Q_{\ored},S_{\ored},M^\varphi,V)$ is the \emph{reduced part} of $\mathcal{M}$.
\end{defi}

\begin{remark}\label{rem:restrictingaction} The construction of a right-equivalence between a QP $(Q,S)$ and the direct sum of a reduced QP with a trivial one is not given by a canonical procedure in any obvious way; that is, there is no canonical way to construct $(Q_{\ored},S_{\ored})$. However, as we saw above, once such a right-equivalence $\varphi:R\langle\langle Q_{\ored}\oplus C\rangle\rangle\rightarrow\ra$ is known, passing from $\mathcal{P}(Q,S)$ to $\mathcal{P}(Q_{\ored},S_{\ored})$ can be done functorially in terms of $\varphi$. In \cite{DWZ} a specific right-equivalence $\varphi$ is defined, satisfying the condition of acting as the identity on all the arrows of $Q$ that do not appear in the degree-2 component $S^{(2)}$ of $S$ as long as no arrow appearing in $S^{(2)}$ appears in different summands of $S^{(2)}$. Using this property of $\varphi$, we see that given a decorated representation $\mathcal{M}=(Q,S,M,V)$, the corresponding action of $u\in R\langle\langle
Q_{\ored}\rangle\rangle$ on $M^\varphi$ coincides with its action on $M$ when $u$ is seen as an element of $\ra$. That is, restricting the action of $\ra$ on $M$ to its subalgebra $R\langle\langle Q_{\ored}\rangle\rangle$ gives us the reduced part $\mathcal{M}_{\ored}$.
\end{remark}

We now turn to the representation-theoretic analogue of the notion of QP-mutation (cf. \cite{DWZ}, Section 10). Let $(Q,S)$ be a QP. Fix a vertex $j\in Q_0$, and suppose that $Q$ has no 2-cycles incident to $j$. Denote by $a_1,\ldots,a_s$ (resp. $b_1,\ldots,b_t$) the arrows ending at $j$ (resp. starting at $j$). Take a QP-representation $\mathcal{M}=(Q,S,M,V)$ and set
$$
M_{\oin}=\underset{k=1}{\overset{s}{\bigoplus}}M_{t(a_k)}, \ \ M_{\oout}=\underset{l=1}{\overset{t}{\bigoplus}}M_{h(b_l)}.
$$
Multiplication by the arrows $a_1,\ldots,a_s$ and $b_1,\ldots,b_t$ induces $K$-linear maps
$$
\al=\al_M=[a_1 \ \ldots \ a_s]:M_{\oin}\rightarrow M_j, \ \ \be=\be_M=\left[\begin{array}{c}b_1 \\ \vdots \\ b_t\end{array} \right]:M_j\rightarrow M_{\oout}.
$$

For each $k$ and each $l$ let $\ga_{k,l}:M_{h(b_l)}\rightarrow M_{t(a_k)}$ be the linear map given by multiplication by the the element $\partial_{[b_la_k]}([S])$, and arrange these maps into a matrix to obtain a linear map $\ga=\ga_M:M_{\oout}\rightarrow M_{\oin}$ (remember that $[S]$ is obtained from $S$ by replacing each $j$-hook $ab$ with the arrow $[ab]$). Since $M$ is a $\jacobas$-module, we have $\al\ga=0$ and $\ga\be=0$ (cf. \cite{DWZ}, Lemma 10.6).

Define vector spaces $\overline{M}_i=M_i$ and $\overline{V}_i=V_i$ for $i\in Q_0$, $i\neq j$, and
$$
\overline{M}_j=\frac{\ker\ga}{\image\be}\oplus\image\ga\oplus\frac{\ker\al}{\image\ga}\oplus V_j, \ \ \ \ \overline{V}_j=\frac{\ker\be}{\ker\be\cap\image\al}.
$$

We define an action of the arrows of $\widetilde{\mu}_j(Q)$ on $\overline{M}=\underset{i\in Q_0}{\bigoplus}\overline{M}_i$ as follows. If $c$ is an arrow of $Q$ not incident to $j$, we define $c_{\overline{M}}=c_M$, and for each $k$ and each $l$ we set $[b_la_k]_{\overline{M}}=(b_la_k)_M={b_l}_M{a_k}_M$. To define the action of the remaining arrows, choose a linear map $\rh:M_{\oout}\rightarrow\ker\ga$ such that $\rh\io=\id_{\ker\ga}$ (where $\io:\ker\ga\hookrightarrow M_{\oout}$ is the inclusion) and a linear map $\si:\frac{\ker\al}{\image\ga}\rightarrow\ker\al$ such that $\pipi\si=\id_{\ker\al/\image\ga}$ (where $\pipi:\ker\al\twoheadrightarrow\frac{\ker\al}{\image\ga}$ is the canonical projection). Then set
$$
[b_1^* \ \ldots \ b_t^*]=\overline{\al}=\left[\begin{array}{c}-\pipi\rh \\ -\ga \\ 0 \\ 0\end{array}\right]:M_{\oout}\rightarrow\overline{M}_j, \ \ \ \  \left[\begin{array}{c}a_1^* \\ \vdots \\ a_s^*\end{array} \right]=\overline{\be}=[0 \ \io \ \io\si \ 0]:\overline{M}_j\rightarrow M_{\oin}.
$$

This action of the arrows of $\widetilde{\mu}_j(Q)$ on $\overline{M}$ extends uniquely to an action of $R\langle\widetilde{\mu}_j(Q)\rangle$ under which $\overline{M}$ is an $R\langle\widetilde{\mu}_j(Q)\rangle$-module. And since $\idealM^{r}M=0$ for some sufficiently large $r$, this action of $R\langle\widetilde{\mu}_j(Q)\rangle$ on $\overline{M}$ extends uniquely to an action of $R\langle\langle\widetilde{\mu}_j(Q)\rangle\rangle$ under which $\overline{M}$ is an $R\langle\langle\widetilde{\mu}_j(Q)\rangle\rangle$-module.

\begin{remark} Note that the choice of the linear maps $\rh$ and $\si$ is not canonical. However, different choices lead to isomorphic $R\langle\langle\widetilde{\mu}_j(Q)\rangle\rangle$-module structures on $\overline{M}$, see \cite{DWZ} Proposition 10.9.
\end{remark}

\begin{defi}[\cite{DWZ}, Section 10] With the above action of $\rtildea$ on $\overline{M}$ and the obvious action of $R$ on $\overline{V}=\underset{i\in Q_0}{\bigoplus}\overline{V}_i$, the quadruple $(\widetilde{\mu}_j(Q),\widetilde{S},\overline{M},\overline{V})$ is called the \emph{premutation}  of $\mathcal{M}=(Q,S,M,V)$ in direction $j$, and denoted $\premutj(\mathcal{M})$. The \emph{mutation} of $\mathcal{M}$ in direction $j$, denoted by $\mutj(\mathcal{M})$, is the reduced part of $\premutj(\mathcal{M})$.
\end{defi}

\begin{thm}[\cite{DWZ}, Proposition 10.10 and Corollary 10.12] Premutations and mutations are well defined up to right-equivalence. That is, if $\qpmod$ and $\qpmodp$ are right-equivalent QP-representations and $Q$ and $Q'$ have no 2-cycles incident to the vertex $j$, then $\premutj(\mathcal{M})$ is right-equivalent to $\premutj(\mathcal{M}')$, and $\mutj(\mathcal{M})$ is right-equivalent to $\mutj(\mathcal{M}')$.
\end{thm}

\begin{thm}[\cite{DWZ}, Theorem 10.13] Mutations of QP-representations are involutive up to right-equivalence. More precisely, if $(Q,S)$ is a 2-acyclic QP and $\mathcal{M}$ is a decorated $(Q,S)$-representation, then $\mutj^2(\mathcal{M})$ and $\mathcal{M}$ are right-equivalent.
\end{thm}

\subsection{Restriction}\label{subsec:restriction}

The operation of restriction of QPs is a simple, yet useful, operation: In \cite{Lqps} it helped to relatively simplify the proof of Theorem 30 by focusing on surfaces with empty boundary. In this subsection we review some properties of this operation and study its representation-theoretic analogue with the same aim in mind: reducing the proof of our main result to the situation of surfaces with empty-boundary.

\begin{defi}[\cite{DWZ}, Definition 8.8] Let $(Q,S)$ be a QP and $I$ be a subset of the vertex set $Q_0$. The \emph{restriction} of $(Q,S)$ to $I$ is the QP $(Q,S)|_I=(Q|_I,S|_I)$ on the vertex set $Q_0$, with arrow span $A|_I=\underset{i,j\in I}{\bigoplus}A_{ij}$ and potential $S|_I=\rho_I(S)$, where $A=\underset{i,j\in Q_0}{\bigoplus}A_{ij}$ is the arrow span of $(Q,S)$ and $\rho_I:R\langle\langle Q\rangle\rangle\rightarrow R\langle\langle Q|_I\rangle\rangle$ is the $R$-algebra homomorphism  such that $\rho_I(a)=a$ for $a\in (Q|_I)_1$ and $\rho_I(b)=0$ for each arrow $b\notin (Q|_I)_1$.
Notice that $Q|_I$ is the quiver whose vertex set is $Q_0$ and whose arrow set $(Q|_I)_1=Q_1|_I$ consists of the arrows of $Q$ having both head and tail in $I$.
We will use the notation $u|_I=\rho_I(u)$ for the restriction to $I$ of an element $u$ of a complete path algebra.
\end{defi}

\begin{remark}\label{remarkisolated} Notice that if $I$ is a proper subset of $Q_0$, then the elements of $Q_0\setminus I$ are isolated vertices of the restriction to $I$, that is, there are no arrows of $A|_I$ whose head or tail belongs to $Q_0\setminus I$.
\end{remark}

\begin{lemma}[\cite{Lqps}, Lemma 20]\label{redres=resred} Let $(Q,S)$ be a QP, and let $I$ be any subset of the vertex set $Q_0$. There exist a reduced and a trivial QP, $\asreduced$ and $\astrivial$, respectively, such that $(Q,S)$ is right-equivalent to $\asreduced\oplus\astrivial$, and with the property that the restriction $(Q|_I,S|_I)$ is right-equivalent to $(Q_{\operatorname{red}}|_I,S_{\operatorname{red}}|_I)\oplus(Q_{\operatorname{triv}}|_I,S_{\operatorname{triv}}|_I)$.
\end{lemma}

\begin{prop}[\cite{Lqps}, Proposition 21]\label{resmut=mutres} Let $(Q,S)$ be a QP and $I$ a subset of $Q_0$. For $i\in I$, the mutation $\muti(Q|_I,S|_I)$ is right-equivalent to the restriction of $\muti(Q,S)$ to $I$.
\end{prop}

\begin{coro}[\cite{Lqps}, Corollary 22] If $(Q,S)$ is a non-degenerate QP, then for every subset $I$ of $Q_0$ the restriction $(Q|_I,S|_I)$ is non-degenerate as well. In other words, restriction preserves non-degeneracy.
\end{coro}

We now turn to restriction of representations.  For a representation $M$, given a subset $I$ of the vertex set $Q_0$, one can define the restriction $M|_I$ in the obvious way, namely, by setting $(M|_I)_j=0$ for $j\notin I$, $(M|_I)_i=M_i$ for $i\in I$, and $a_{M|_I}=a_M:M_{t(a)}\rightarrow M_{h(a)}$ for $a\in Q_1|_I$. However, if the representation $M$ satisfies the cyclic derivatives of a potential $S$, the restriction $M|_I$ does not necessarily satisfy the relations obtained by restricting the cyclic derivatives of $S$.

\begin{defi} Let $I$ be a subset of the vertex set $Q_0$, and $M$ be a representation of $Q$. We will say that $M$ is \emph{$I$-path-restrictable} if for each pair $i,j\in I$, every path from $i$ to $j$ in $Q$ passing through a vertex $k\notin I$ acts as zero on $M_i$. Accordingly, a QP-representation $\qpmod$ will be called $I$-path-restrictable if $M$ is $I$-path-restrictable.
\end{defi}

Clearly, if the representation $M$ of $Q$ is $I$-path-restrictable and satisfies the cyclic derivatives of the potential $S$, then $M|_I$ satisfies the cyclic derivatives of $S|_I$ (which are also the restrictions of the cyclic derivatives of $S$). This implies the following:

\begin{lemma}\label{lemmarestI} Let $I$ be a subset of the vertex set $Q_0$ and $\qpmod$ be a QP-representation such that $M$ is $I$-path-restrictable. Set $(M|_I)_{i}=M_i$ for $i\in I$ and $(M|_I)_{j}=0$ for $j\notin I$. Then $(Q|_I,S|_I,((M|_I)_i)_{i\in Q_0},(a_M)_{a\in Q_1|_I},(V_i)_{i\in Q_0})$ is a QP-representation.
\end{lemma}

\begin{defi} In the situation of Lemma \ref{lemmarestI}, the QP-representation $(Q|_I,S|_I,((M|_I)_i)_{i\in Q_0},$ $(a_M)_{a\in Q_1|_I},$
$(V_i)_{i\in Q_0})$ will be called the \emph{restriction} of $\qpmod$ to $I$ and denoted $\mathcal{M}|_I$.
\end{defi}

The following lemma and its corollary are the representation-theoretic analogue of Lemma \ref{redres=resred} above.

\begin{lemma}\label{restrofrightequiv} If the decorated representations $\qpmod$ and $\qpmodp$ are right-equivalent and $\qpmod$ is $I$-path-restrictable, then $\qpmodp$ is $I$-path-restrictable as well, and the restrictions $\mathcal{M}|_I$ and $\mathcal{M}'|_I$ are right-equivalent.
\end{lemma}

\begin{proof} Let $\Phi=(\varphi,\psi,\eta)$ be a right-equivalence between $\qpmod$ and $\qpmodp$. Consider the $R$-algebra homomorphism $\varphi|_I:R\langle\langle Q|_I\rangle\rangle\rightarrow R\langle\langle Q'|_I\rangle\rangle$, defined by the rule $u\mapsto\varphi(u)|_I$. We claim that $\varphi|_I(S|_I)=\varphi(S)|_I$. To see this, write $S=S|_I+S'$, where $S'\in\ra$ is a potential each of whose terms has at least one arrow $b\notin Q_1|_I$. Then each term of $\varphi(S')$ has at least one arrow not from $Q'|_I$, which means that $\varphi(S')|_I=0$, and hence $\varphi(S)|_I=\varphi(S|_I)|_I=\varphi|_I(S|_I)$. Now, the $R$-algebra homomorphism $\rho_I:R\langle\langle Q'\rangle\rangle\rightarrow R\langle\langle Q'|_I\rangle\rangle$, being continuous, sends cyclically equivalent potentials to cyclically equivalent ones, from which it follows that $\varphi(S)|_I$ is cyclically equivalent to $S'|_I$. Therefore, $\varphi|_I$ is a right-equivalence between $(Q|_I,S|_I)$ and $(Q'|_I,S'|_I)$.

To see that $M'$ is $I$-path-restrictable if $M$ is, take any pair of vertices $i,j\in I$, and let $u$ be any path from $i$ to $j$ in $Q'$ that passes through some vertex $k\notin I$. Then for $m\in M_i'$ we have $u_{M'}m=\varphi(\varphi^{-1}(u))_{M'}\psi(\psi^{-1}(m))=\psi\circ(\varphi^{-1}(u)_M)(\psi^{-1}(m))=0$ since $\varphi^{-1}(u)$ is a (possibly infinite) linear combination of paths that pass through $k\notin I$. This shows that $M'$ is $I$-path-restrictable.

Note that $M|_I$ (resp. $M'|_I$) is an $R$-submodule of $M$ (resp. $M'$), and $\psi(M|_I)=M'|_I$. This allows us to define $\psi|_I:M|_I\rightarrow M'|_I$ as the restriction of the map $\psi$ to $M|_I$. Clearly $\psi|_I:M|_I\rightarrow M'|_I$ is a $K$-vector space isomorphism, and for $a\in Q_1|_I$ we have $\psi|_I\circ a_{M|_I}=(\psi\circ a_{M})|_I=(\varphi(a)_{M'}\circ\psi)|_I=(\varphi|_I(a)_{M'|_I})\circ(\psi|_I)$ (the last equality follows from the fact that $M'$ is $I$-path-restrictable).

It follows that the triple $\Phi|_I=(\varphi|_I,\psi|_I,\eta)$ is a right-equivalence between $\mathcal{M}|_I$ and $\mathcal{M}'|_I$.
\end{proof}

\begin{remark} The first paragraph of the above proof is an alternative proof of Lemma 20 of \cite{Lqps} (stated as Lemma \ref{redres=resred} above).
\end{remark}

\begin{coro}\label{redcommuteswithrestr} Let $I$ be a subset of the vertex set $Q_0$ and $\mathcal{M}=\qprep$ be an $I$-path-restrictable QP-representation. Then $\mathcal{M}_{\ored}$ is $I$-path-restrictable and $\mathcal{M}_{\ored}|_I$ is right-equivalent to the reduced part of $\mathcal{M}|_I$.
\end{coro}

\begin{proof} Let $\varphi:R\langle\langle Q_{\ored}\oplus C\rangle\rangle\rightarrow\ra$ be a right-equivalence between
$(Q_{\ored},S_{\ored})\oplus(C,T)$ and $(Q,S)$, where $(Q_{\ored},S_{\ored})$ is a reduced QP and $(C,T)$ is a trivial QP. Then the triple $\Phi=(\varphi,\idM,\idV)$ is a right-equivalence between $(Q_{\ored}\oplus C,S_{\ored}+T,M^\varphi,V)$ and $(Q,S,M,V)$ (where $M^\varphi=M$ as $K$-vector spaces, see Proposition \ref{reddetermineduptorequiv} and Definition \ref{reducedpartofrep}). By Lemma \ref{restrofrightequiv} $M^\varphi$ is $I$-path-restrictable. By the proof of Lemma \ref{restrofrightequiv}, the triple $\Phi|_I=(\varphi|_I,\idM|_I,\idV)$ is a right-equivalence between $(Q_{\ored}|_I\oplus C|_I,S_{\ored}|_I+T|_I,M'|_I,V)$ and $(Q|_I,S|_I,M|_I,V)$. By Proposition \ref{reddetermineduptorequiv}, this implies that the reduced part of $\mathcal{M}|_I$ is right-equivalent to $(Q_{\ored}|_I,S_{\ored}|_I,M'|_I,V)$.
\end{proof}

\begin{thm}\label{thm:resmut=mutres1.1} Let $(Q,S)$ be a QP, $I$ a subset of the vertex set $Q_0$, and $j\in I$. If $\qpmod$ is an $I$-path-restrictable QP-representation, then the mutation $\mutj(\mathcal{M})$ is $I$-path-restrictable as well, and  the restriction $\mutj(\mathcal{M})|_I$ is right-equivalent to the mutation $\mutj(\mathcal{M}|_I)$.
\end{thm}

\begin{proof} An easy check shows that $(\widetilde{Q|_I},\widetilde{S|_I},\overline{M|_I},\overline{V})=(\widetilde{Q}|_I,\widetilde{S}|_I,\overline{M}|_I,\overline{V})$, where $(\widetilde{Q|_I},\widetilde{S|_I})=\widetilde{\mu}_j(Q|_I,S|_I)$ and $(\widetilde{Q}|_I,\widetilde{S}|_I)=\widetilde{\mu}_j(Q,S)|_I$. The theorem follows then from Corollary \ref{redcommuteswithrestr}.
\end{proof}

\begin{remark} It is very easy to give examples where $I$-path-restrictability fails and the conclusion of Theorem \ref{thm:resmut=mutres1.1} does not hold. Consider, for instance the QP-representation $\calMtauarc$ of Example \ref{hexagonnice} below, with $I=\{\arcone_1\}$. On the other hand, there are conditions weaker than path-restrictability that still ensure the conclusion of Theorem \ref{thm:resmut=mutres1.1}; we do not state these conditions here since we will not need them.
\end{remark}

\section{The QP of an ideal triangulation}\label{Section:QPoftriangulation}

In this section we briefly review the basic material on triangulations of surfaces and their signed-adjacency quivers and potentials. For far-reaching discussions on the triangulations' cluster behavior, we refer the reader to \cite{FST} and \cite{FT}.

\begin{defi}[\cite{FST}, Definition 2.1] A \emph{bordered surface with marked points} is a pair $\surf$, where $\surfnoM$ is a compact connected oriented Riemann surface
with (possibly empty) boundary, and $\marked$ is a finite set of points on $\surfnoM$, called \emph{marked points}, such that $\marked$ is
non-empty and has at least one point from each connected component of the boundary of $\surfnoM$. The marked points that lie in the interior of
$\surfnoM$ will be called \emph{punctures}, and the set of punctures of $\surf$ will be denoted $\punct$. We will always assume that $\surf$ is
none of the following:
\begin{itemize}
\item a sphere with less than five punctures;
\item an unpunctured monogon, digon or triangle;
\item a once-punctured monogon.
\end{itemize}
Here, by a monogon (resp. digon, triangle) we mean a disk with exactly one (resp. two, three) marked point(s) on the boundary.
\end{defi}

An (\emph{ordinary}) \emph{arc} in $\surf$ (cf. \cite{FST}, Definition 2.2) is a curve $\arc$ in $\surfnoM$ such that:
\begin{itemize}
\item the endpoints of $\arc$ are marked points in $\marked$;
\item $\arc$ does not intersect itself, except that its endpoints may coincide;
\item the relative interior of $\arc$ is disjoint from $\marked$ and from the boundary of $\surfnoM$;
\item $\arc$ does not cut out an unpunctured monogon or an unpunctured digon.
\end{itemize}
We consider two arcs $\arc_1$ and $\arc_2$ to be the same whenever they are isotopic in $\surfnoM$ rel $\marked$, that is whenever there exists an isotopy $H:I\times\surfnoM\rightarrow\surfnoM$ such that $H(0,x)=x$ for all $x\in\surfnoM$, $H(1,\arc_1)=\arc_2$, and $H(t,m)=m$ for all $t\in I$ and all $m\in\marked$. An arc whose endpoints coincide will be called
a \emph{loop}. We denote the set of (isotopy classes of) arcs in
$\surf$ by $\arcsinsurf$.

Two arcs are \emph{compatible} if there are arcs in their respective isotopy classes whose relative interiors do not intersect (cf. \cite{FST}, Definition 2.4). An \emph{ideal triangulation} of $\surf$ is any maximal collection of pairwise compatible arcs whose relative interiors do not intersect each other (cf. \cite{FST}, Definition 2.6).
All ideal triangulations of $\surf$ have the same number $n$ of arcs, the \emph{rank} of $\surf$ (because it coincides with the rank of the cluster algebra associated to $\surf$, see \cite{FST}).

If $\tau$ is an ideal triangulation of $\surf$ and we take a connected component of the complement in $\surfnoM$ of the union of the arcs in $\tau$, the closure $\triangle$ of this component will be called an \emph{ideal triangle} of $\tau$.
An ideal triangle $\triangle$ is \emph{self-folded} if it contains exactly two arcs of $\tau$ and shares at most one point with the boundary of $\surf$ (see Figure \ref{selffoldedtriang}).
        \begin{figure}[!h]
                \caption{Self-folded triangle}\label{selffoldedtriang}
                \centering
                \includegraphics[scale=.4]{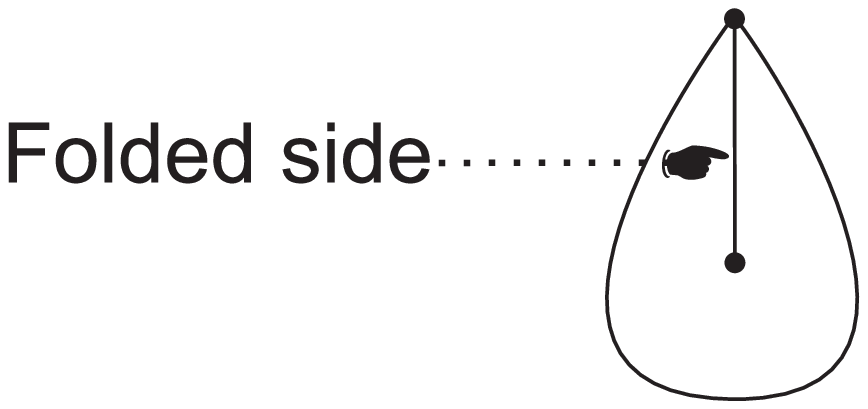}
        \end{figure}

Let $\tau$ be an ideal triangulation of $\surf$ and let $\arcone\in\tau$ be an arc. If $\arcone$ is not the folded side of a self-folded triangle,
then there exists exactly one arc $\arctwo$, different from $\arcone$, such that $\sigma=(\tau\setminus\{\arcone\})\cup\{\arctwo\}$ is an ideal
triangulation of $\surf$. We say that $\sigma$ is obtained by applying a \emph{flip} to $\tau$, or by \emph{flipping} the arc $\arcone$ (cf. \cite{FST}, Definition 3.5), and write
$\sigma=f_\arcone(\tau)$. Any two ideal triangulations are related by a sequence of flips (cf. \cite{FST}, Proposition 3.8). In the literature, flips are sometimes called \emph{Whitehead moves} or \emph{elementary moves}.

In order to be able to flip the folded sides of self-folded triangles, one has to enlarge the set of arcs with which
triangulations are formed. This is done by introducing the notion of \emph{tagged arc}. Since we will deal only with ordinary arcs in this paper, we refer the reader to \cite{FST} and \cite{FT} for the definition and properties of tagged arcs and \emph{tagged triangulations}. If $\surf$ is not a surface with empty boundary and exactly one puncture, then any two tagged triangulations are related by a sequence of flips (cf. \cite{FST}, Proposition 7.10).

\begin{defi}\label{def:adjquiver} Each ideal triangulation $\tau$ without self-folded triangles is the vertex set of a quiver whose arrows are defined by means of the following two-step procedure:
\begin{enumerate}
\item For each triangle $\triangle$ with sides $\arc,\arcone,\arctwo\in\tau$, ordered in the clockwise direction induced by the orientation of $\surfnoM$, introduce multiplicity-one arrows $\arc\rightarrow\arcone$, $\arcone\rightarrow\arctwo$, $\arctwo\rightarrow\arc$.
\item Delete 2-cycles one by one.
\end{enumerate}
The quiver $\hatqtau$ obtained after the $1^{\operatorname{st}}$ step will be called the \emph{unreduced signed-adjacency quiver of $\tau$}, whereas the quiver $\qtau$ obtained after applying both steps will receive the name of \emph{(reduced) signed-adjacency quiver}.
\end{defi}

\begin{remark}\begin{enumerate}\item Signed-adjacency quivers and their unreduced versions can be defined for any ideal triangulation, and even for any tagged triangulation (see \cite{FST} or \cite{FT}). We do not include the general definitions since we will work only with ideal triangulations without self-folded triangles.
\item As we will see in Definition \ref{QPfortriangulation}, all the 2-cycles of $\widehat{Q}(\tau)$ will be summands of the corresponding potential, hence the reduction process of Theorem \ref{splittingthm} will get rid of all the 2-cycles of $\widehat{Q}(\tau)$. Therefore the underlying quiver of the reduced part of this QP will indeed be $\qtau$. The reason for not deleting the 2-cycles before defining the potential is that, if we did so, then for some triangulations we would get a potential for which Theorem \ref{flip<->mutation} does not hold.
\end{enumerate}
\end{remark}

\begin{thm}[\cite{FST}, Proposition 4.8] Let $\tau$ and $\sigma$ be ideal triangulations. If $\sigma$ is obtained from $\tau$
by flipping the arc $\arc$ of $\tau$, then $Q(\sigma)=\muti(\qtau)$.
\end{thm}

\begin{defi}[\cite{Lqps}, Definition 23]\label{QPfortriangulation} Let $\surf$ be a bordered surface with marked points and $P\subseteq M$ be the set of punctures of $\surf$. Fix a choice $(x_p)_{p\in P}$ of non-zero scalars (one scalar $x_p\in\field$ for each $p\in P$), which is going to remain fixed for all ideal triangulations of $\surf$. Let $\tau$ be an ideal triangulation of $\surf$ without self-folded triangles. Based on our choice $(x_p)_{p\in\punct}$ we associate to $\tau$ a potential $\stau\in\ratau$ as follows. Let $\widehat{Q}(\tau)$ be the unreduced signed adjacency quiver of $\tau$ (cf. \cite{Lqps}, Definition 8).
\begin{itemize} \item Each interior ideal triangle $\triangle$ of $\tau$ gives rise to an oriented triangle of $\widehat{Q}(\tau)$, let $\widehat{S}^\triangle$ be such oriented triangle up to cyclical equivalence.
\item For each puncture $p$, the arrows of $\widehat{Q}(\tau)$ between the arcs incident to $p$ form a unique cycle $a^p_1\ldots a^p_d$ that exhausts all such arcs and gives a complete round around $p$ in the counter-clockwise orientation defined by the orientation of $\surfnoM$. We define $\widehat{S}^p=x_pa^p_1\ldots a^p_d$ (up to cyclical equivalence).
\end{itemize}
The \emph{unreduced potential} $\unredstau\in\runredatau$ of $\tau$ is then defined by
\begin{equation}
\unredstau=\underset{\triangle}{\sum}\widehat{S}^\triangle+\underset{p\in\punct}{\sum}\widehat{S}^p,
\end{equation}
where the first sum runs over all interior triangles.

Finally, we define $\qstau$ to be the (right-equivalence class of the) reduced part of $\unredastau$.
\end{defi}

\begin{thm}[\cite{Lqps}, Theorems 30 and 31]\label{flip<->mutation} Let $\tau$ and $\sigma$ be ideal triangulations of $\surf$. If $\sigma=f_j(\tau)$, then $\mu_j(\atau,\stau)$ and $(\asigma,\ssigma)$ are right-equivalent QPs. Furthermore, if $\surfnoM$ has non-empty boundary, then all the potentials $\qtau$ are rigid, hence non-degenerate.
\end{thm}

For the definition of $\qstau$ in the presence of self-folded triangles we refer the reader to \cite{Lqps}, where some examples are treated as well.

\section{Definition of arc representations}\label{repsforarcsspecial}

Let $\surf$ be a bordered surface with marked points and $P\subseteq M$ be the set of punctures of $\surf$. Fix a choice $(x_p)_{p\in P}$ of non-zero elements of the field $\field$. Let $\tau$ be an ideal triangulation of $\surf$ without self-folded triangles and $\arc\in\arcsinsurf$ be an arc.

\subsection{First case: $\arc$ does not cut out a once-punctured monogon}\label{case1.1}

Throughout this subsection we will assume that
\begin{equation}\label{arcnotaloop}\text{the arc $\arc$ is not a loop that cuts out a once-punctured monogon.}
\end{equation}

We are going to define a representation $\Mtauarc$ in several stages. First we define the \emph{detours} of $\arc$ with respect to $\tau$ and encode them into \emph{detour matrices}. Then we define the \emph{segment representation} of $Q(\tau)$ with respect to $\arc$, and use the detour matrices to modify it and obtain the \emph{arc representation} $\Mtauarc$ of the unreduced QP $(\widehat{Q}(\tau),\widehat{S}(\tau))$. By Remark \ref{rem:restrictingaction}, $\Mtauarc$ will actually be a representation of $(Q(\tau),S(\tau))$, where the action of $Q(\tau)$ on $\Mtauarc$ is given by simply forgetting the action of the arrows appearing in the 2-cycles of $\widehat{Q}(\tau)$.

Let us begin assuming that $\arc\notin\tau$. Then, replacing $\arc$ by an isotopic arc if necessary, we can assume that
\begin{equation}\label{transversally} \text{$\arc$ intersects transversally each of the arcs of $\tau$ (if at all), and}
\end{equation}
\begin{equation}\label{minimalintersection}\text{the number of intersection points of $\arc$ with each of the arcs of $\tau$ is minimal.}
\end{equation}

In (\ref{minimalintersection}) we mean that, if $\arc'$ is isotopic to $\arc$, then $\arc'$ does not have a smaller number of intersection points with any of the arcs of $\tau$.

Fix an arc $\arcone\in\tau$; it is contained in two ideal triangles. Fix one such triangle $\triangle$, and let
$\badintersection^{\triangle,1}_{\arc,\arcone}$ be the set whose elements are the ordered quadruples $(q_{\arcone,t_1},q_{\arcone,t_2},r_1,p)$ for which we have the situation of Figure \ref{curvearoundcounterclock},
\begin{figure}[!h]
                \caption{}\label{curvearoundcounterclock}
                \centering
                \includegraphics[scale=.5]{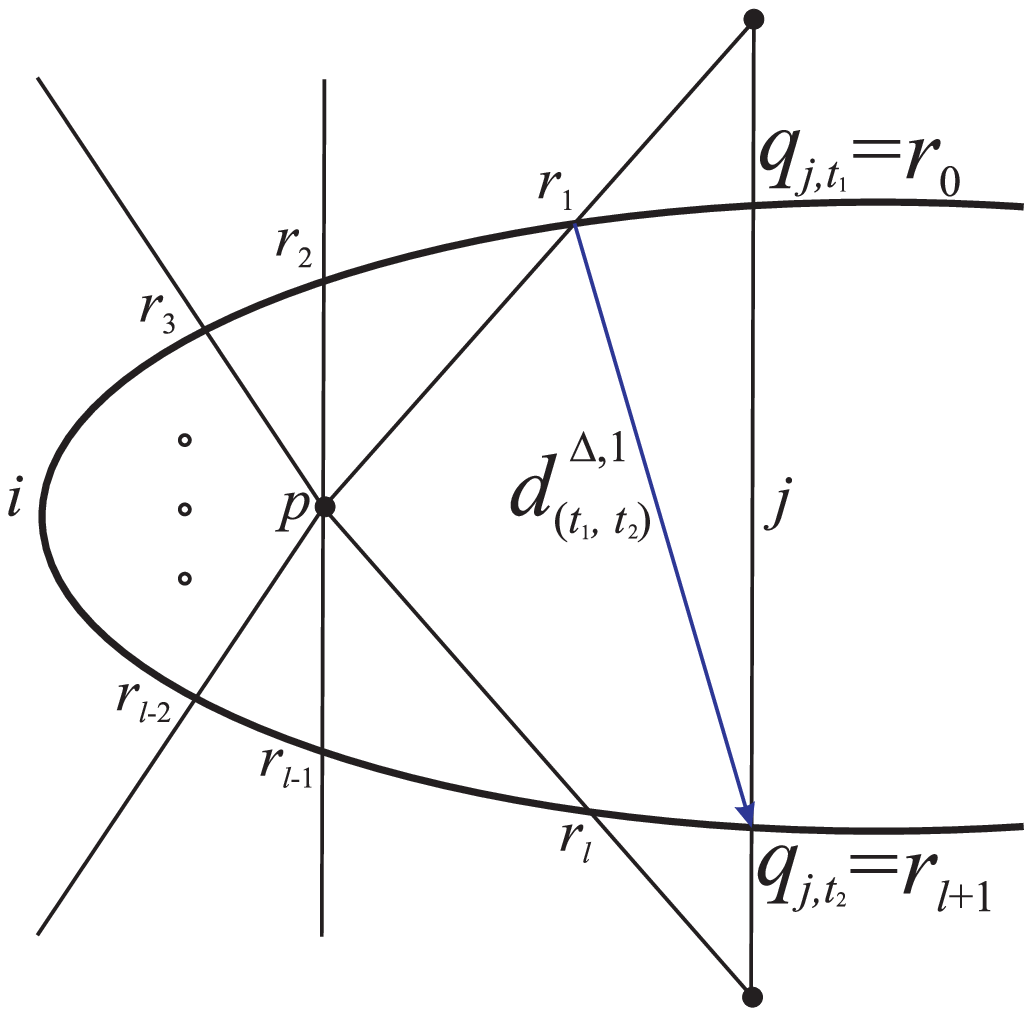}
        \end{figure}
where the segment of $\arc$ that goes from $q_{\arcone,t_1}$ to $q_{\arcone,t_2}$ can be divided into segments $[r_0,r_1]_\arc,\ldots,[r_{l},r_{l+1}]_\arc$, with the following properties:
\begin{itemize}
\item $r_0=q_{\arcone,t_1}\in\arcone$, $r_{l+1}=q_{\arcone,t_2}\in\arcone$;
\item $l$ is the number of arcs of $\tau$ incident to the puncture $p$ (counted with multiplicity);
\item the only points of $[r_k,r_{k+1}]_\arc$ that lie on an arc of $\tau$ are $r_k$ and $r_{k+1}$;
\item the segment $[r_k,r_{k+1}]_\arc$ is contractible to the puncture $p$ with a homotopy each of whose intermediate maps are segments with endpoints in the arcs of $\tau$ to which $r_k$ and $r_{k+1}$ belong;
\item the segments $[r_0,r_1]_\arc$ and $[r_{l},r_{l+1}]_\arc$ are contained in $\triangle$;
\item the union of the oriented segments $[q_{\arcone,t_1},q_{\arcone,t_2}]_\arc$ and $[q_{\arcone,t_2},q_{\arcone,t_1}]_\arcone$ is a closed simple curve contractible to the puncture $p$, and whose complement in $\surfnoM$ consists of two connected components, one of which contains exactly one puncture (namely $p$);
\item the oriented closed curve of the previous item surrounds $p$ in the counterclockwise direction.
\end{itemize}

\begin{defi} For each such quadruple $(q_{\arcone,t_1},q_{\arcone,t_2},r_1,p)\in\badintersection^{\triangle,1}_{\arc,\arcone}$ we draw an oriented simple curve $d^{\triangle,1}_{(t_1,t_2)}$ contained in $\triangle$ and going from $r_1$ to $q_{t_2}$, and say that $d^{\triangle,1}_{(t_1,t_2)}$ is a \emph{1-detour} of $(\tau,\arc)$. We will write $b(d^{\triangle,1}_{(t_1,t_2)})=r_1$ for the \emph{beginning point} of $d^{\triangle,1}_{(t_1,t_2)}$ and $e(d^{\triangle,1}_{(t_1,t_2)})=q_{t_2}$ for its \emph{ending point}. We shall also say that $p$ is the puncture \emph{detoured} by $d^{\triangle,1}_{(t_1,t_2)}$.
\end{defi}

For $n\geq 1$, after having drawn all $n$-detours of $(\tau,\arc)$, take an arc $\arcone$ and fix an ideal triangle $\triangle$ containing $\arcone$. Let $\badintersection^{\triangle,n+1}_{\arc,\arcone}$ be the set whose elements are the ordered quadruples $(q_{\arcone,t_1},q_{\arcone,t_2},b(d^n),p)$ for which we have the situation shown in Figure \ref{nplus1detour},
\begin{figure}[!h]
                \caption{Drawing $n+1$-detours after drawing the $n$-detours}\label{nplus1detour}
                \centering
                \includegraphics[scale=.5]{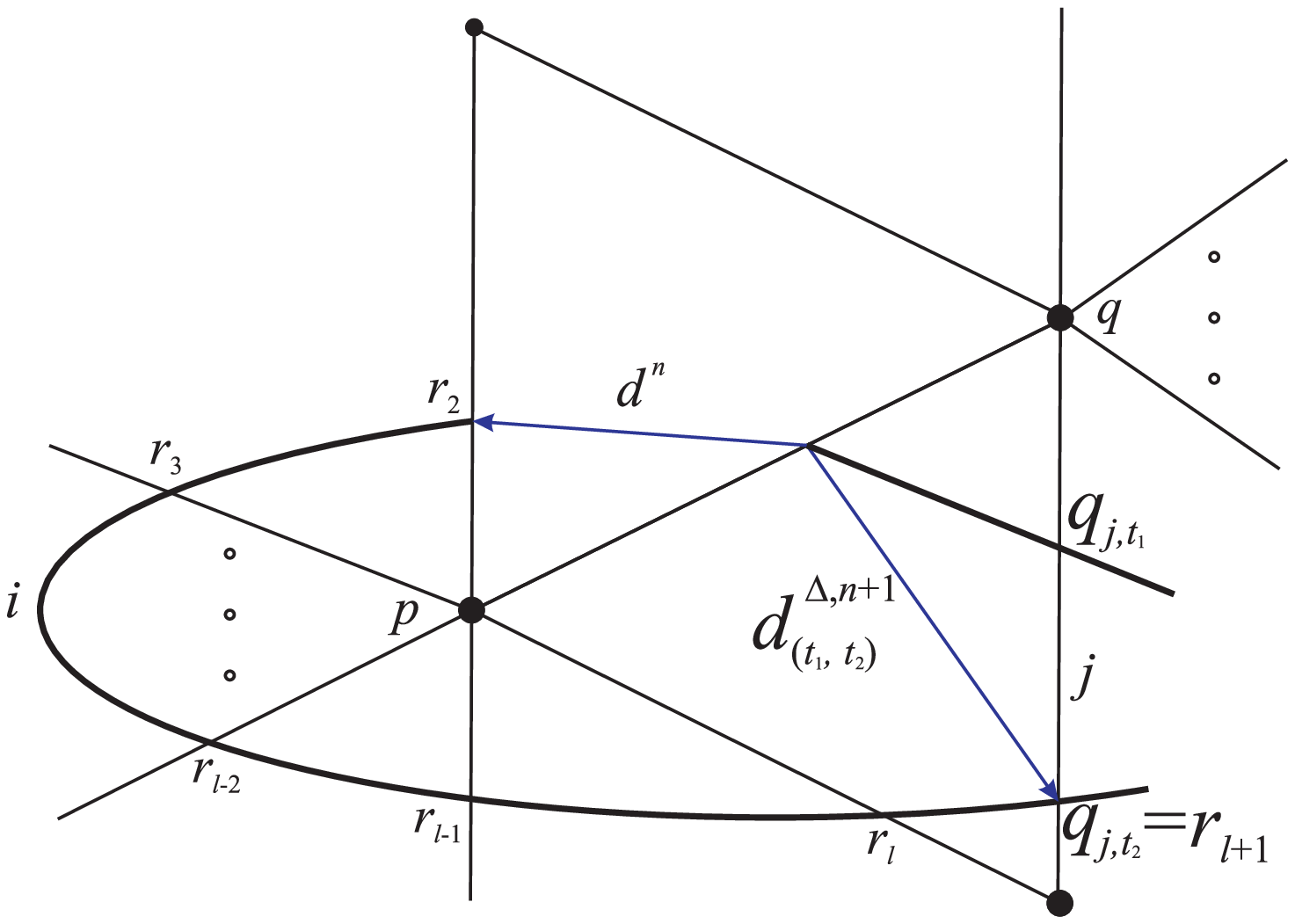}
        \end{figure}
where the segment of $\arc$ that goes from the endpoint of the $n$-detour $d^n$ to $q_{\arcone,t_2}$ can be divided into segments
$[r_2,r_3]_\arc,\ldots,[r_{l},r_{l+1}]_\arc$, with the following properties:
\begin{itemize}
\item $r_{l+1}=q_{\arcone,t_2}\in\arcone$;
\item $l$ is the number of arcs of $\tau$ incident to the puncture $p$ (counted with multiplicity);
\item the only points of $[r_k,r_{k+1}]_\arc$ that lie on an arc of $\tau$ are $r_k$ and $r_{k+1}$;
\item for $2\leq k\leq l-1$, the segment $[r_k,r_{k+1}]_\arc$ is contractible to the puncture $p$ with a homotopy each of whose intermediate maps are segments with endpoints in the arcs of $\tau$ to which $r_k$ and $r_{k+1}$ belong;
\item the segments $[q_{\arcone,t_1},b(d^n)]_{\arc}$ and $[r_l,q_{\arcone,t_2}]_\arc$ are contained in $\triangle$;
\item the union of the oriented segment $[q_{\arcone,t_1},b(d^n)]_\arc$, the $n$-detour $d^n$, and the oriented segments $[e(d^n),q_{\arcone,t_2}]_\arc$ and $[q_{\arcone,t_2},q_{\arcone,t_1}]_\arcone$ is a closed simple curve contractible to the puncture $p$, and whose complement in $\surfnoM$ consists of two connected components, one of which contains exactly one puncture (namely $p$);
\item the oriented closed curve of the previous item surrounds $p$ in the counterclockwise direction.
\end{itemize}

\begin{defi} For each quadruple $(q_{\arcone,t_1},q_{\arcone,t_2},b(d^n),p)\in\badintersection^{\triangle,n+1}_{\arc,\arcone}$ we draw an oriented simple curve $d^{\triangle,n+1}_{(t_1,t_2)}$ contained in $\triangle$ and going from $b(d^n)$ to $q_{t_2}$, and say that $d^{\triangle,n+1}_{(t_1,t_2)}$ is an $(n+1)$\emph{-detour} of $(\tau,\arc)$. We will write $b(d^{\triangle,n+1}_{(t_1,t_2)})$ and $e(d^{\triangle,n+1}_{(t_1,t_2)})$ for the \emph{beginning} and \emph{ending} points of $d^{\triangle,n+1}_{(t_1,t_2)}$. We shall also say that $p$ is the puncture \emph{detoured} by $d^{\triangle,n+1}_{(t_1,t_2)}$.
\end{defi}

\begin{remark}\label{remaboutdetours} \begin{enumerate}\item Since each detour connects points of intersection of $\arc$ with (arcs of) $\tau$, and since for a triangle $\triangle$ and intersection points $q_{1}$ and $q_2$ there is at most one detour contained in $\triangle$ and connecting $q_1$ with $q_2$, the arc $\arc$ has only finitely many detours with respect to $\tau$. In other words, the process of drawing detours stops after finitely many steps;
\item Given $n$, a triangle $\triangle$ may contain more than one $n$-detour;
\item Given an $(n+1)$-detour $d^{n+1}$ there exists exactly one $n$-detour used to define $d^{n+1}$. This $n$-detour $d^n$ satisfies $b(d^n)=b(d^{n+1})$, point that lies on an arc of $\tau$ that connects the punctures detoured by $d^n$ and $d^{n+1}$. This means that each $n$-detour $d^n$ determines a sequence $(d^1,\ldots,d^n)$ where $d^m$ is an $m$-detour with $1\leq m\leq n$ and $b(d^m)=b(d^{n+1})$; the sequence of punctures detoured by the members of the sequence alternates between two punctures of $\surf$.
\item If we think of the arrows of $\qtau$ as oriented curves on the surface, then each detour is \emph{parallel} to exactly one arrow of $\qtau$. Notice that if an arrow $a$ is parallel to a detour, then $a$ is parallel to a 1-detour.
\end{enumerate}
\end{remark}

\begin{defi} Using the detours of $(\tau,\arc)$ we define two \emph{detour matrices} for each arc $\arcone$ as follows. Take an ideal triangle $\triangle$ cointaining $\arcone$. The rows and columns of the \emph{detour matrix} $D^\triangle_{\arc,\arcone}$ are indexed by the intersection points of $\arc$ with the relative interior of $\arcone$. For each such point $q_{\arcone,t}$, the corresponding column of $D^\triangle_{\arc,\arcone}$ is defined according to the following rules:
\begin{itemize}\item the $q_{\arcone,t}^{\phantom{\arcone t}\operatorname{th}}$ entry is 1;
\item if an intersection point $q_{\arcone,s}$ is the ending point of an $n$-detour $d^{\triangle,n}_{(t,s)}$ and there is a quadruple $(q_{\arcone,t},q_{\arcone,s},b(d^{\triangle,n}_{(t,s)}),p)\in \badintersection^{\triangle,n}_{\arc,\arcone}$,  then the $q_{\arcone,s}^{\phantom{\arcone s}\operatorname{th}}$ entry is
\begin{equation}\label{eq:nontriventrydetmatrix}
(-1)^nx_p^{\lfloor\frac{n+1}{2}\rfloor}x_q^{\lfloor\frac{n}{2}\rfloor},
\end{equation}
     where $\{p,q\}$ is the set of punctures incident to the arc that contains the point $b(d^{\triangle,n}_{(s,t)})$;
\item all the remaining entries of the $q_{\arcone,t}^{\phantom{\arcone t}\operatorname{th}}$ column are zero.
\end{itemize}
\end{defi}

We now turn to the definition of the \emph{segment} and \emph{arc representations} for $\arc$. In both of them, the vector spaces attached to the vertices of $\qtau$ will be given by
\begin{equation}\label{spacesforrepresentation}
\Mtauarc_\arcone=\field^{\A(\arc,\arcone)},
\end{equation}
where $\A(\arc,\arcone)$ is the number of intersection points of $\arc$ with the relative interior of $\arcone$.
For $t=1,\ldots,\A(\arc,\arcone)$, we will write $K_{\arcone,t}$ to denote the copy of the field $\field$ that corresponds to $q_{\arcone,t}$ in equation (\ref{spacesforrepresentation}).

Now we define the linear maps $(\mtauarc_a)_{a\in Q_1(\tau)}$. Let $a:\arcone\rightarrow\arctwo$ be an arrow of $\widehat{Q}(\tau)$; the way in which $\widehat{Q}(\tau)$ is defined gives us a puncture $p(a)$ canonically associated to $a$, namely the puncture at which $\arcone$ and $\arctwo$ are adjacent. Assume that $\arc$ intersects the relative interior of $\arcone$ (resp. $\arctwo$) in the $\A(\arc,\arcone)$ (resp. $\A(\arc,\arctwo)$) different points $q_{\arcone,1},\ldots,q_{\arcone,\A(\arc,\arcone)}$ (resp. $q_{\arctwo,1},\ldots,q_{\arctwo,\A(\arc,\arctwo)}$). For $1\leq s\leq\A(\arc,\arcone)$ and $1\leq r\leq\A(\arc,\arctwo)$, let $(\mtauarc_a)_{r,s}:\field_{\arcone,s}\rightarrow\field_{\arctwo,r}$ be the identity if and only if the following conditions are satisfied:
\begin{itemize}
\item The relative interior of the segment $[q_{\arcone,s},q_{\arctwo,r}]_\arc$ of $\arc$ that connects $q_{\arcone,s}$ and $q_{\arctwo,r}$ does not intersect any arc of $\tau$;
\item the segments $[p(a),q_{\arcone,s}]_{\arcone}$, $[p(a),q_{\arctwo,r}]_{\arctwo}$, $[q_{\arcone,s},q_{\arctwo,r}]_\arc$ form a triangle contractible in $\surfnoM\setminus(M\cup\partial\surfnoM)$. More precisely, the segment $[q_{\arcone,s},q_{\arctwo,r}]_\arc$ can be contracted to the puncture $p(a)$ with a homotopy each of whose intermediate maps are segments with endpoints in the arcs $\arcone$ and $\arctwo$.
\end{itemize}
Otherwise, define $(\mtauarc_a)_{r,s}:\field\rightarrow\field$ to be the zero map.

\begin{defi} The representation $\mtauarc$ just constructed will be called the \emph{segment representation} of $\unredqtau$ induced by $\arc$.
\end{defi}

It is easy to see that, in the presence of punctures, the segment representation $\mtauarc$ does not necessarily satisfy the cyclic derivatives of $\widehat{S}(\tau)$. Let us illustrate with an example.

\begin{ex} Consider the arc $\arc$ and the ideal triangulation $\tau$ of the once-punctured hexagon shown in Figure \ref{curvepuncthexagonmotivating}, where
        \begin{figure}[!h]
                \caption{}\label{curvepuncthexagonmotivating}
                \centering
                \includegraphics[scale=1]{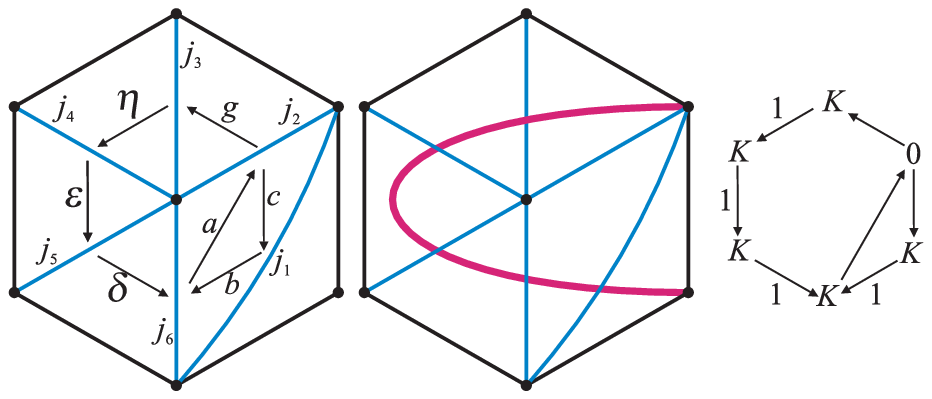}
        \end{figure}
the representation $\mtauarc$ is shown as well. This representation obviously satisfies the cyclic derivatives of the potential $\stau=abc+xa\delta\varepsilon\eta g$. Furthermore, after applying the sequence of mutations $\mu_{\arcone_1}$, $\mu_{\arcone_3}$, $\mu_{\arcone_4}$, $\mu_{\arcone_5}$, $\mu_{\arcone_6}$, we get the $\arcone_6^{\phantom{6}\operatorname{th}}$ negative simple representation $\mathcal{S}^-_6(\mathbb{D}_6,0)=\mu_{\arcone_6}\mu_{\arcone_5}\mu_{\arcone_4}\mu_{\arcone_3}\mu_{\arcone_1}(\mtauarc)$ of the QP $(\mathbb{D}_6,0)=\mu_{\arcone_6}\mu_{\arcone_5}\mu_{\arcone_4}\mu_{\arcone_3}\mu_{\arcone_1}\qstau$, where $\mathbb{D}_6$ has the following orientation and labeling of vertices:
\begin{displaymath}
\xymatrix{ & & & & \arcone_2\\
\arcone_3 \ar[r] & \arcone_4 \ar[r] & \arcone_5 \ar[r] & \arcone_6 \ar[ur] & \\
 & & & & \arcone_1 \ar[ul]}
\end{displaymath}

If we flip the arc $\arcone_2$ of $\tau$, we obtain the ideal triangulation $\sigma=f_{\arcone_2}(\tau)$ shown in Figure \ref{curvepuncthexagonmotivatingwithdetour},
        \begin{figure}[!h]
                \caption{}\label{curvepuncthexagonmotivatingwithdetour}
                \centering
                \includegraphics[scale=1]{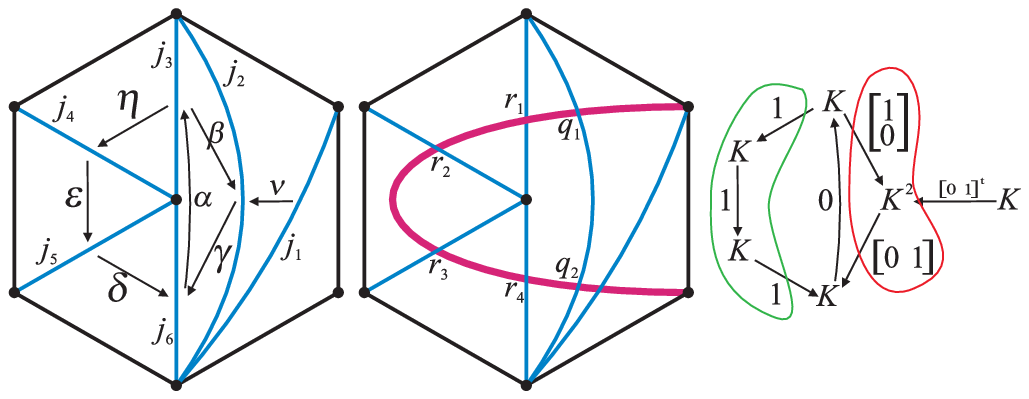}
        \end{figure}
where also the representation $\msigmaarc$ is shown. This representation does not satisfy all the cyclic derivatives of the potential $\ssigma=\gamma\beta\alpha+x\alpha\delta\varepsilon\eta$, namely, it is not annihilated by $\partial_\alpha(\ssigma)=\gamma\beta+x\delta\varepsilon\eta$. Therefore, $\msigmaarc$ cannot be obtained from $\mtauarc$ by applying the $\arcone_2^{\phantom{2}\operatorname{th}}$ mutation. Consequently, $\msigmaarc$ is not mutation-equivalent to the negative simple representation $\mathcal{S}^-_6(\mathbb{D}_6,0)=
\mu_{\arcone_6}\mu_{\arcone_5}\mu_{\arcone_4}\mu_{\arcone_3}\mu_{\arcone_1}(\mtauarc)$
of $\mathbb{D}_6$.
\end{ex}

We modify $\mtauarc$ using the detour matrices as follows. For each arrow $a:\arcone\rightarrow\arctwo$ of $\qtau$, let $\triangle^{a}$ be the unique ideal triangle that contains $a$. Define a linear map $\Mtauarc_a:\field^{\A(\arc,\arcone)}\rightarrow\field^{\A(\arc,\arctwo)}$ to be given by the matrix product
\begin{equation}
(D^{\triangle^a}_{\arc,\arctwo})(\mtauarc_a).
\end{equation}

\begin{defi} With the action of $R\langle\langle\qtau\rangle\rangle$ induced by the inclusion of quivers $\qtau\hookrightarrow\unredqtau$, the representation $\Mtauarc$ will be called the \emph{arc representation} of $\qtau$ induced by $\arc$.
\end{defi}

\begin{remark}\begin{enumerate}\item In many cases (even for punctured surfaces), the arc representation $\Mtauarc$ coincides with the segment representation $\mtauarc$. When the surface has no punctures, we always have $\Mtauarc=\mtauarc$.
\item For $\surf=$ unpunctured polygon, the arc representations were defined by P. Caldero, F. Chapoton an R. Schiffler (cf. \cite{CCS04}), and have been recently generalized to the situation $\surf=$ unpunctured surface by I. Assem, T. Br\"{u}stle, G. Charbonneau-Jodoin and P-G. Plamondon (cf. \cite{ABCP}).
\item We use the terms ``segment representation" and ``arc representation" instead of the more appealing ``string module" because, in the presence of punctures, the Jacobian algebras $\mathcal{P}(Q(\tau),S(\tau))$ are not necessarily string algebras.
\item If $\surfnoM$ has non-empty boundary, then all the QPs $\qstau$ are non-degenerate and have finite-dimensional Jacobian algebras, so these QPs admit C. Amiot's categorification \cite{Amiot-gldim2}. In this context, each arc on $\surf$ represents an object of the \emph{cluster category} $\mathcal{C}$, and each triangulation $\tau$ represents a \emph{cluster-tilting object} whose endomorphism algebra is precisely the Jacobian algebra $\mathcal{P}(Q(\tau),S(\tau))$; moreover, for each fixed triangulation there is a functor from $\mathcal{C}$ to the module category of the Jacobian algebra of the triangulation. As a consequence of Theorem \ref{thm:flip<->mut} below, the arc representation $\Mtauarc$ gives an explicit calculation of the image of $\arc$ under the functor $\mathcal{C} \rightarrow \operatorname{mod}\mathcal{P}(Q(\tau),S(\tau))$. For type $\mathbb{D}_n$, a complete geometric model of the cluster category was given by R. Schiffler in \cite{S}, and the representations $\Mtauarc$ can also be seen as an explicit calculation of the image of the objects under the corresponding functor.
\end{enumerate}
\end{remark}

To illustrate these constructions, let us have a look at an example.

\begin{ex}\label{hexagonnice} Consider the triangulation $\tau$ and the arc $\arc$ on the twice-punctured hexagon shown in Figure \ref{nicecurveonhexagon}.
\begin{figure}[!h]
                \caption{}\label{nicecurveonhexagon}
                \centering
                \includegraphics[scale=.35]{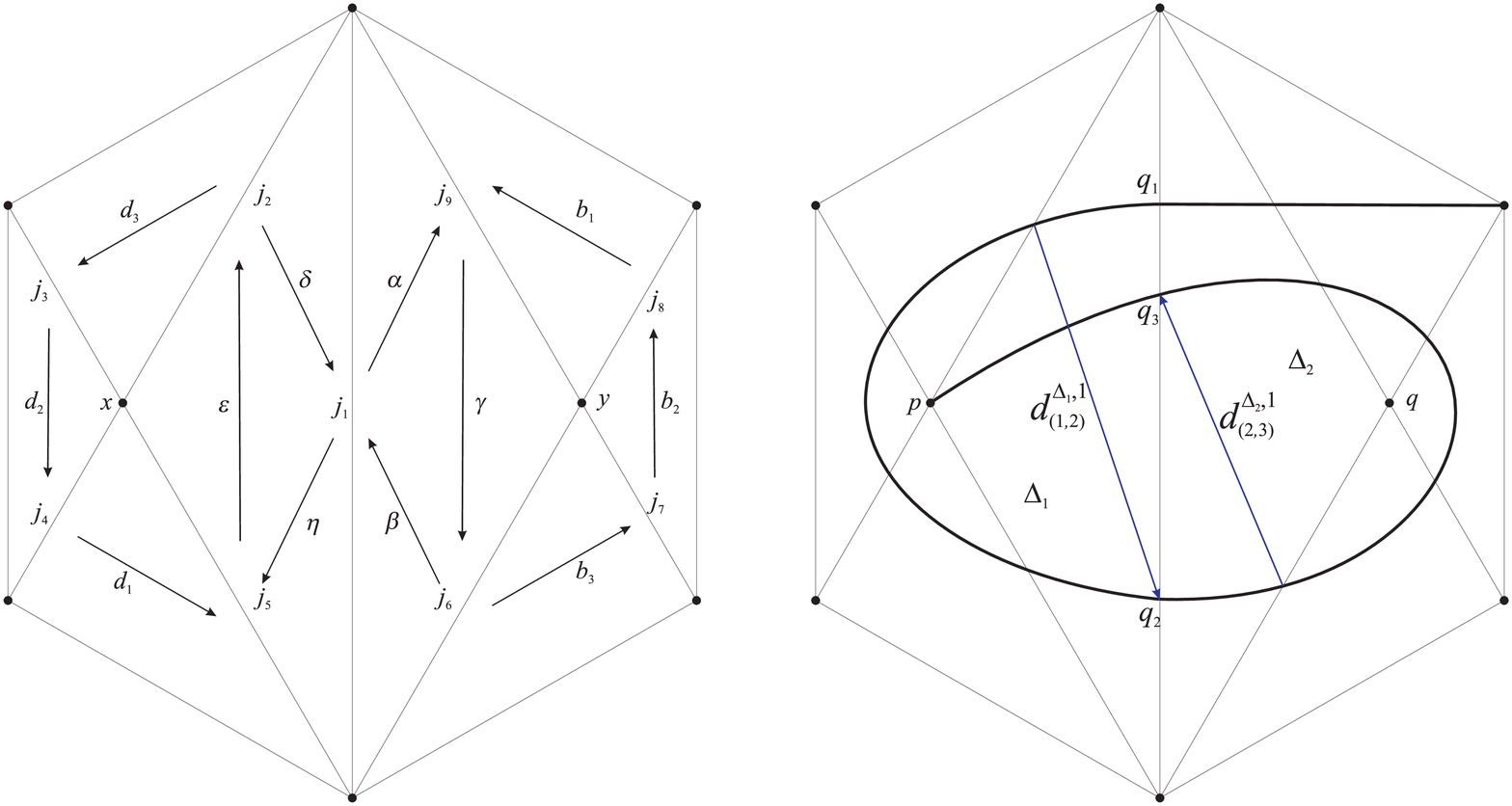}
        \end{figure}
It is straightforward to see that
$$
\badintersection^{\triangle_1,1}_{\arc,\arcone_1}=\{(q_1,q_2,b(d^{\triangle_1,1}_{(1,2)}),p)\}, \ \ \ \badintersection^{\triangle_2,1}_{\arc,\arcone_1}=\{(q_2,q_3,b(d^{\triangle_2,1}_{(2,3)}),q)\}
$$
$$
\text{and} \ \badintersection^{\triangle,n}_{\arc,\arcone}=\varnothing \ \text{for} \ n\geq 2 \ \text{and} \ \arcone\in Q_1(\tau)
$$
$$
\Mtauarc_{\arcone_1}=\field^3, \ \Mtauarc_{\arcone_l}=\field \ \text{for} \ l=2,\ldots,8, \ \text{\and} \ \Mtauarc_{\arcone_9}=\field^2;
$$
Hence $D^{\triangle}_{\arc,\arcone_l}=\mathbf{1}$ for $l=2,\ldots,9$ (and each ideal triangle $\triangle$ containing $\arcone_l$), whereas
$$
D^{\triangle_1}_{\arc,\arcone_1}=\left[\begin{array}{ccc}1 & 0 & 0 \\
                                                -x & 1 & 0 \\
                                                0 & 0 & 1\end{array}\right], \ \text{and} \
D^{\triangle_2}_{\arc,\arcone_1}=\left[\begin{array}{ccc}1 & 0 & 0 \\
                                                                      0 & 1 & 0 \\
                                                                      0 & -y & 1\end{array}\right],
$$
where $x=x_p$ and $y=x_q$. The linear maps $\mtauarc_a$ of the segment representation are given as follows
$$
\mtauarc_\alpha=\left[\begin{array}{ccc}
1 & 0 & 0 \\
0 & 0 & 1
\end{array}\right], \
\mtauarc_\beta=\left[\begin{array}{c}
0 \\
1 \\
0
\end{array}\right], \
\mtauarc_\gamma=0, \
\mtauarc_\delta=\left[\begin{array}{c}
1 \\
0 \\
0
\end{array}\right],
$$
$$
\mtauarc_\varepsilon=0, \
\mtauarc_\eta=\left[\begin{array}{ccc}
0 & 1 & 0
\end{array}\right], \
\mtauarc_{d_1}=\mathbf{1}, \
\mtauarc_{d_2}=\mathbf{1},
$$
$$
\mtauarc_{d_3}=\mathbf{1}, \
\mtauarc_{b_1}=\left[\begin{array}{c}
0 \\
1
\end{array}\right], \
\mtauarc_{b_2}=\mathbf{1} \
\mtauarc_{b_3}=\mathbf{1}.
$$
Therefore, the representation $\Mtauarc$ is
\begin{displaymath}
\xymatrix{
\field \ar[dd]_{1} & \field \ar[l]_{1} \ar[drr]^{{\tiny \left[\begin{array}{c}1 \\ -x \\ 0\end{array}\right] }} &   &  &  & \field^{2} \ar[dd]_{0} & \field \ar[l]_{{\tiny \left[\begin{array}{cc}0 & 1\end{array}\right] }} \\
 & & & \field^{3} \ar[urr]^{{\tiny \left[\begin{array}{ccc}1 & 0 & 0 \\ 0 & 0 & 1\end{array}\right] }} \ar[dll]^{{\tiny \left[\begin{array}{ccc}0 & 1 & 0\end{array}\right] }} & & & \\
\field \ar[r]_{1} & \field \ar[uu]_{0} & & & & \field \ar[ull]^{{\tiny \left[\begin{array}{c}0 \\ 1 \\ -y\end{array}\right] }} \ar[r]_1 & \field \ar[uu]_1}
\end{displaymath}

Note that this representation is actually a $\jacobqstau$-module, that is, it satisfies the relations imposed by the potential $\stau=\alpha\beta\gamma+yb_1b_2b_3\gamma+\delta\eta\varepsilon+xd_1d_2d_3\varepsilon$. Therefore, $\calMtauarc=(Q(\tau),S(\tau),\Mtauarc,\Vtauarc)$ is a QP-representation, where $\Vtauarc=0$.

If we flip the arc $\arcone_1$, we get ideal triangulation $\sigma=f_{\arcone_1}(\tau)$ shown in Figure \ref{nicecurveonhexagonflipped}
\begin{figure}[!h]
                \caption{}\label{nicecurveonhexagonflipped}
                \centering
                \includegraphics[scale=.35]{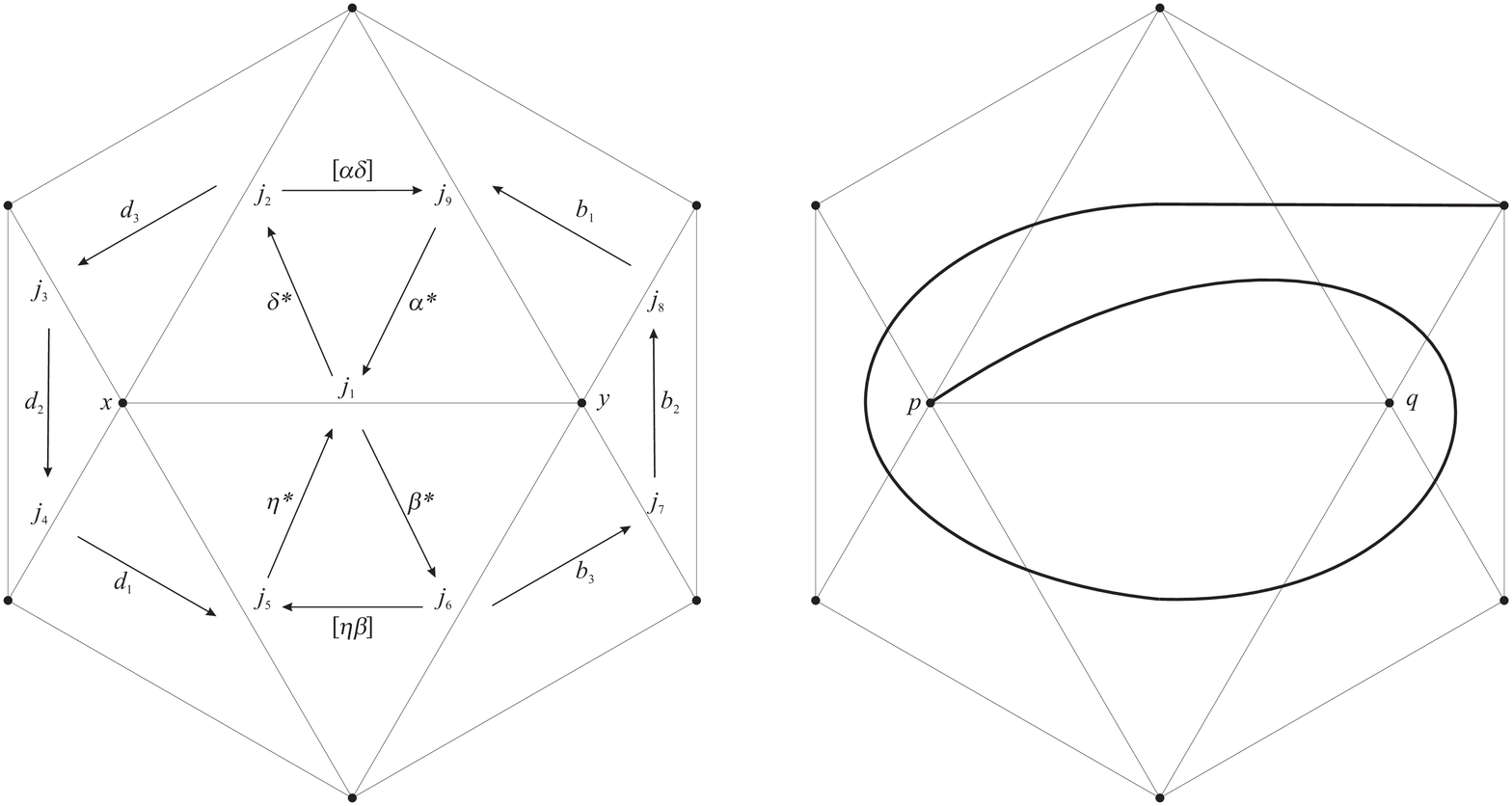}
        \end{figure}
(abusing notation, we use the same symbol $j_1$ in both $\tau$ and $\sigma$) and the following representation of its signed adjacency quiver
\begin{displaymath}
\xymatrix{
\field \ar[dd]_{1} & \field \ar[l]_{1} \ar[rr]^{\left[\begin{array}{c}1 \\ 0\end{array}\right]} & & \field^{2} \ar[dl] & \field \ar[l]_{\left[\begin{array}{c}0 \\ 1\end{array}\right]} \\
 &  & 0 \ar[ul] \ar[dr] & &  \\
\field \ar[r]_{1} & \field \ar[ur] & & \field \ar[ll]^{1} \ar[r]_{1} & \field \ar[uu]_{1}}
\end{displaymath}
which obviously satisfies the relations imposed by the potential $S(\sigma)=[\alpha\delta]\delta^*\alpha^*+[\eta\beta]\beta^*\eta^*+x\gamma^*\eta^*d_1d_2d_3+y\beta^*\alpha^*b_1b_2b_3$.

A straightforward calculation shows that this representation (with the zero decoration) can be obtained also by performing the mutation $\mu_{\arcone_1}$ to $\calMtauarc$. That is, the flip of $\arcone_1$ has the same effect on $\calMtauarc$ as the $\arcone_1^{\phantom{1}\operatorname{th}}$ QP-mutation.
\end{ex}

\subsection{Second case: $\arc$ cuts out a once-punctured monogon}\label{subsec:cutsout}

In this subsection we deal with the case where the arc $\arc$ is a loop that cuts out a once-punctured monogon from $\surf$. Specifically, throughout this section we will keep assuming that $\tau$ is an ideal triangulation of $\surf$ without self-folded triangles, that $\arc$ is an arc on $\surf$, $\arc\not\in\tau$, satisfying (\ref{transversally}) and (\ref{minimalintersection}), and that
\begin{equation}\label{arcloop} \text{the arc $\arc$ is a loop that cuts out a once-punctured monogon from $\surf$}.
\end{equation}

Let $\bigodot$ the monogon cut out by $\arc$ and $p$ be the puncture inside $\bigodot$. Consider the (unique) arc $\arctwo$ that connects $p$ with the marked point $m$ at which $\arc$ is based and is contained in $\bigodot$.
\begin{equation}
\text{If $\arctwo$ arc belongs to $\tau$, then the arc representation $\Mtauarc$ is defined following the}
\end{equation}
\begin{center}
the exact same rules of Subsection \ref{case1.1}.
\end{center}
So, assume $\arctwo$ does not belong to $\tau$. Consider all the segments of arcs of $\tau$ contained in $\bigodot$ that have one extreme on $\arc$ and the other at $p$. Let $\mathfrak{F}$ be those extreme points that lie on $\arc$  (see Figure \ref{aroundp}).
\begin{figure}[!h]
                \caption{$\mathfrak{F}$ consists of the circled intersection points}\label{aroundp}
                \centering
                \includegraphics[scale=.4]{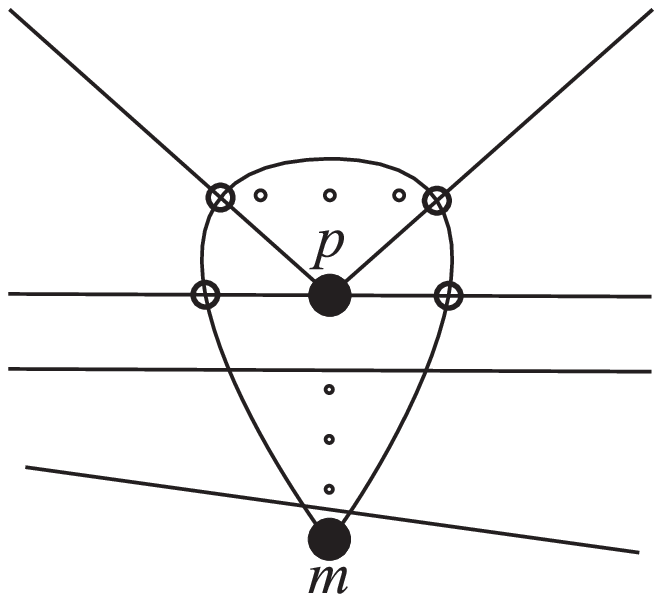}
        \end{figure}
If we traverse $\arc$ in the clockwise direction around $p$, at some moment we will begin passing through the elements of $\mathfrak{F}$. Right after exhausting these elements, before having finished traversing $\arc$, we must pass through a point of $\arc\cap\tau$ that does not belong to $\mathfrak{F}$ (see Figure \ref{aroundp}). Let $t$ be the first such point, and delete the segment of $\arc$ we have not traversed yet (see Figure \ref{deletingsegment}).

\begin{figure}[!h]
                \caption{}\label{deletingsegment}
                \centering
                \includegraphics[scale=.4]{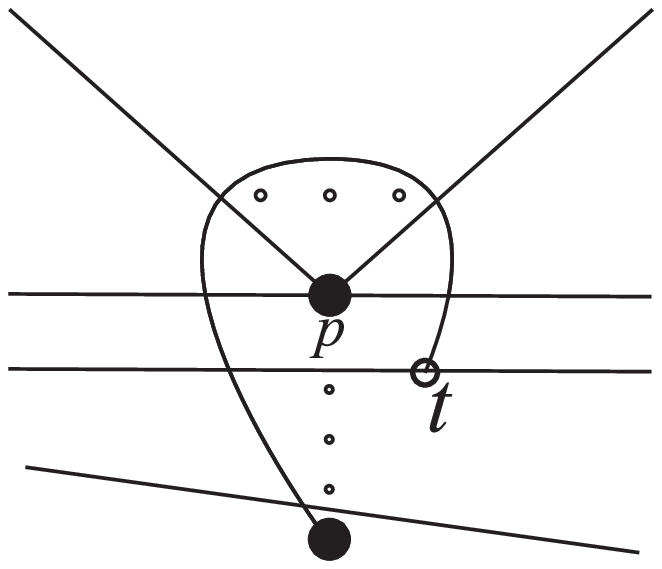}
        \end{figure}
The segment of $\arc$ we have not deleted is a curve $\iota=\iota_{\tau,\arc}$ on $\surf$ one of whose endpoints is $m$, a marked point, $t$ being the other endpoint. For $n\geq 1$ we define the $n$-detours of $(\tau,\iota)$ in the exact same way we did in the previous subsection, but with respect to $\iota$ instead of $\arc$. The detour matrices of $(\tau,\iota)$ are also defined in the exact same way.

Now we turn to the definition of the segment representation $\mtauarc$. For each arc $\arcone\in\tau$, assume that the points in which the
oriented segment $[m,t]_{\iota}$ intersects $\arcone$ are $q_{\arcone,1},\ldots,q_{\arcone,\A(\arc,\arcone)}$ (in this order along the orientation chosen for $[m,f]_{\iota}$). The vector spaces attached to the vertices of $\qtau$ by the segment representation will be given by
\begin{equation}\label{spacesforrepsMONOGONCASE}
\mtauarc_\arcone=\field^{\A(\arc,\arcone)}.
\end{equation}

The linear maps $\mtauarc_a$ are defined as follows. Let $a:\arcone\rightarrow\arctwo$ be an arrow of $\hatqtau$, and assume that the segment $[m,t]_{\iota}$ intersects the relative interior of $\arcone$ (resp. $\arctwo$) in the $\A(\arc,\arcone)$ (resp. $\A(\arc,\arctwo)$) different points $q_{\arcone,1},\ldots,q_{\arcone,\A(\arc,\arcone)}$ (resp. $q_{\arctwo,1},\ldots,q_{\arctwo,\A(\arc,\arctwo)}$). Let $(\mtauarc_a)_{r,s}:\field_{q_{\arcone,s}}\rightarrow\field_{q_{\arctwo,r}}$ be the identity if and only if the following conditions are satisfied:
\begin{itemize}
\item The relative interior of the segment $[q_{\arctwo,r},q_{\arcone,s}]_\arc$ of $\arc$ that connects $q_{\arctwo,r}$ and $q_{\arcone,s}$ does not intersect any arc of $\tau$;
\item the segments $[p(a),q_{\arcone,s}]_{\arcone}$, $[p(a),q_{\arctwo,r}]_{\arctwo}$, $[q_r,p_s]_\arc$ form a triangle contractible in $\surfnoM\setminus(M\cup\partial\surfnoM)$.
\end{itemize}
Otherwise, define $(\mtauarc_a)_{r,s}:\field\rightarrow\field$ to be the zero map.

\begin{defi} The representation $\mtauarc$ just constructed will be called the \emph{segment representation} of $\qtau$ induced by $\arc$.
\end{defi}

Just as in Subsection \ref{case1.1}, it is easy to see that the segment representation $\mtauarc$ does not satisfy the cyclic derivatives of $S(\tau)$. So we modify it using the detour matrices to produce the arc representation $\Mtauarc$. The dimension of this representation will be one less than that of $\mtauarc$. Let $\arc'$ be the arc of $\tau$ containing $t$; for $\arcone\neq\arc'$, we set
$$
\Mtauarc_{\arcone}=\mtauarc_{\arcone}.
$$
As for $\arc'$, the space $\Mtauarc_{\arc'}$ is defined to be the quotient of $\mtauarc_{\arc'}$ by the copy of $\field$ that corresponds to the intersection point $t$. That is, $\Mtauarc_{\arc'}$ takes into account only the intersection points of $[m,f]_{\iota}$ with $\arc'$.

Now let us define the linear maps of the arc representation. Let $a:\arcone\rightarrow\arctwo$ be an arrow of $\qtau$, and $\triangle^a$ be the unique ideal triangle of $\tau$ that contains $a$. If $\arcone\neq\arc'\neq\arctwo$, then $\Mtauarc_a:\field:\Mtauarc_{\arcone}\rightarrow\Mtauarc_{\arctwo}$ is defined to be $(D^{\triangle^a}_{\arc,\arctwo})(\mtauarc_a)$. If $\arc'=\arcone$, then $\Mtauarc_a=(D^{\triangle^a}_{\arc,\arctwo})(\mtauarc_a)\ell$, where $\ell:\Mtauarc_{\arc'}\hookrightarrow\mtauarc_{\arc'}$ is the canonical vector space inclusion. And if $\arc'=\arctwo$, then $\Mtauarc_a=\pi(D^{\triangle^a}_{\arc,\arctwo})(\mtauarc_a)$, where $\pi:\mtauarc_{\arc'}\twoheadrightarrow\Mtauarc$ is the canonical vector space projection.

\begin{defi} With the action of $R\langle\langle\qtau\rangle\rangle$ induced by the inclusion of quivers $\qtau\hookrightarrow\unredqtau$, the representation $\Mtauarc$ will be called the \emph{arc representation} of $\qtau$ induced by $\arc$.
\end{defi}

\begin{remark} The arc representation $\Mtauarc$ never coincides with the segment representation $\mtauarc$.
\end{remark}

To illustrate these definition, let us give an example.

\begin{ex}\label{monogonhexagon} Consider the triangulation $\tau$ and the arc $\arc$ on the twice-punctured hexagon shown in Figure \ref{punctmonogononhexagon}.
\begin{figure}[!h]
                \caption{}\label{punctmonogononhexagon}
                \centering
                \includegraphics[scale=.35]{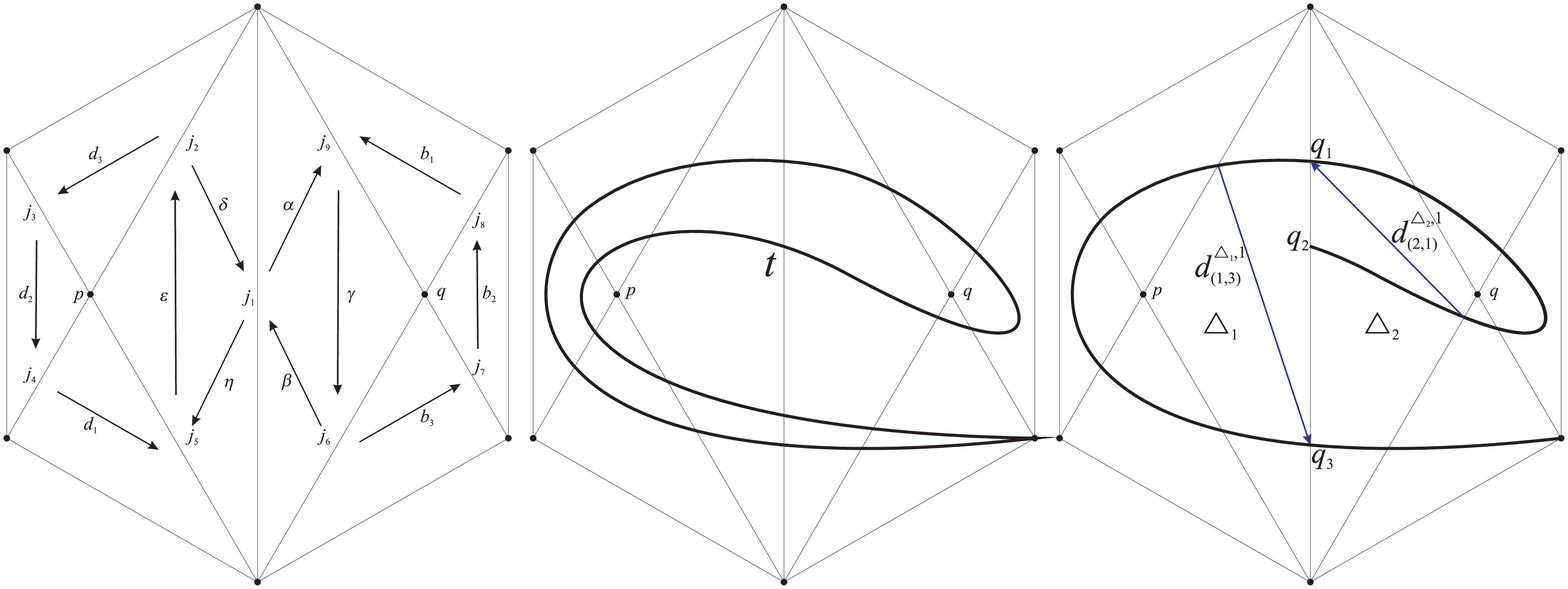}
        \end{figure}
The point $t$ is indicated there, and farthest right we can see the segment $\iota=\iota_{\arc,\tau}$ and its detours with respect to $\tau$. It is straightforward to see that
$$
\badintersection^{\triangle_1,1}_{\arc,\arcone_1}=\{(q_1,q_3,b(d^{\triangle_1,1}_{(1,3)}),p)\}, \ \ \ \badintersection^{\triangle_2,1}_{\arc,\arcone_1}=\{(q_2,q_1,b(d^{\triangle_2,1}_{(2,1)}),q)\}
$$
$$
\text{and} \ \badintersection^{\triangle,n}_{\arc,\arcone}=\varnothing \ \text{for} \ n\geq 2 \ \text{and} \ \arcone\in Q_1(\tau)
$$
$$
\mtauarc_{\arcone_1}=\field^3, \ \mtauarc_{\arcone_6}=\field^2, \ \text{and} \ \Mtauarc_{\arcone_l}=\field \ \text{for} \ l=2,3,4,5,7,8,9;
$$
Hence $D^{\triangle}_{\arc,\arcone_l}=\mathbf{1}$ for $l=2,\ldots,9$ (and each ideal triangle $\triangle$ containing $\arcone_l$), whereas
$$
D^{\triangle_1}_{\arc,\arcone_1}=\left[\begin{array}{ccc}1 & 0 & 0 \\
                                                0 & 1 & 0 \\
                                                -x & 0 & 1\end{array}\right], \ \text{and} \
D^{\triangle_2}_{\arc,\arcone_1}=\left[\begin{array}{ccc}1 & -y & 0 \\
                                                                      0 & 1 & 0 \\
                                                                      0 & 0 & 1\end{array}\right],
$$
where $x=x_p$ and $y=x_q$. The linear maps $\mtauarc_a$ of the segment representation are given as follows
$$
\mtauarc_\alpha=\left[\begin{array}{ccc}
1 & 0 & 0 \\
\end{array}\right], \
\mtauarc_\beta=\left[\begin{array}{cc}
0 & 0\\
1 & 0\\
0 & 1
\end{array}\right], \
\mtauarc_\gamma=0, \
\mtauarc_\delta=\left[\begin{array}{c}
1 \\
0 \\
0
\end{array}\right],
$$
$$
\mtauarc_\varepsilon=0, \
\mtauarc_\eta=\left[\begin{array}{ccc}
0 & 0 & 1
\end{array}\right], \
\mtauarc_{d_1}=\mathbf{1}, \
\mtauarc_{d_2}=\mathbf{1},
$$
$$
\mtauarc_{d_3}=\mathbf{1}, \
\mtauarc_{b_1}=\mathbf{1}, \
\mtauarc_{b_2}=\mathbf{1} \
\mtauarc_{b_3}=\left[\begin{array}{cc}1 & 0\end{array}\right].
$$
Therefore, the representation $\Mtauarc$ is
\begin{displaymath}
\xymatrix{
\field \ar[dd]_{1} & \field \ar[l]_{1} \ar[drr]^{{\tiny \left[\begin{array}{c}1 \\ -x\end{array}\right] }} &   &  &  & \field \ar[dd]_{0} & \field \ar[l]_{1} \\
 & & & \field^{2} \ar[urr]^{{\tiny \left[\begin{array}{cc}1 & 0\end{array}\right] }} \ar[dll]^{{\tiny \left[\begin{array}{cc}0 & 1\end{array}\right] }} & & & \\
\field \ar[r]_{1} & \field \ar[uu]_{0} & & & & \field^{2} \ar[ull]^{{\tiny \left[\begin{array}{cc}-y & 0 \\ 0 & 1\end{array}\right] }} \ar[r]_{{\tiny\left[\begin{array}{cc}1 & 0\end{array}\right]}} & \field \ar[uu]_1}
\end{displaymath}

Note that this representation is actually a $\jacobqstau$-module, that it, it satisfies the relations imposed by the potential $\stau=\alpha\beta\gamma+yb_1b_2b_3\gamma+\delta\varepsilon\eta+xd_1d_2d_3\varepsilon$. Therefore, $\calMtauarc=(Q(\tau),S(\tau),\Mtauarc,\Vtauarc)$ is a QP-representation, where $\Vtauarc=0$.

If we flip the arc $\arcone_1$, we get ideal triangulation $\sigma=f_{\arcone_1}(\tau)$ shown in Figure \ref{punctmonogononhexagonflipped}
\begin{figure}[!h]
                \caption{}\label{punctmonogononhexagonflipped}
                \centering
                \includegraphics[scale=.35]{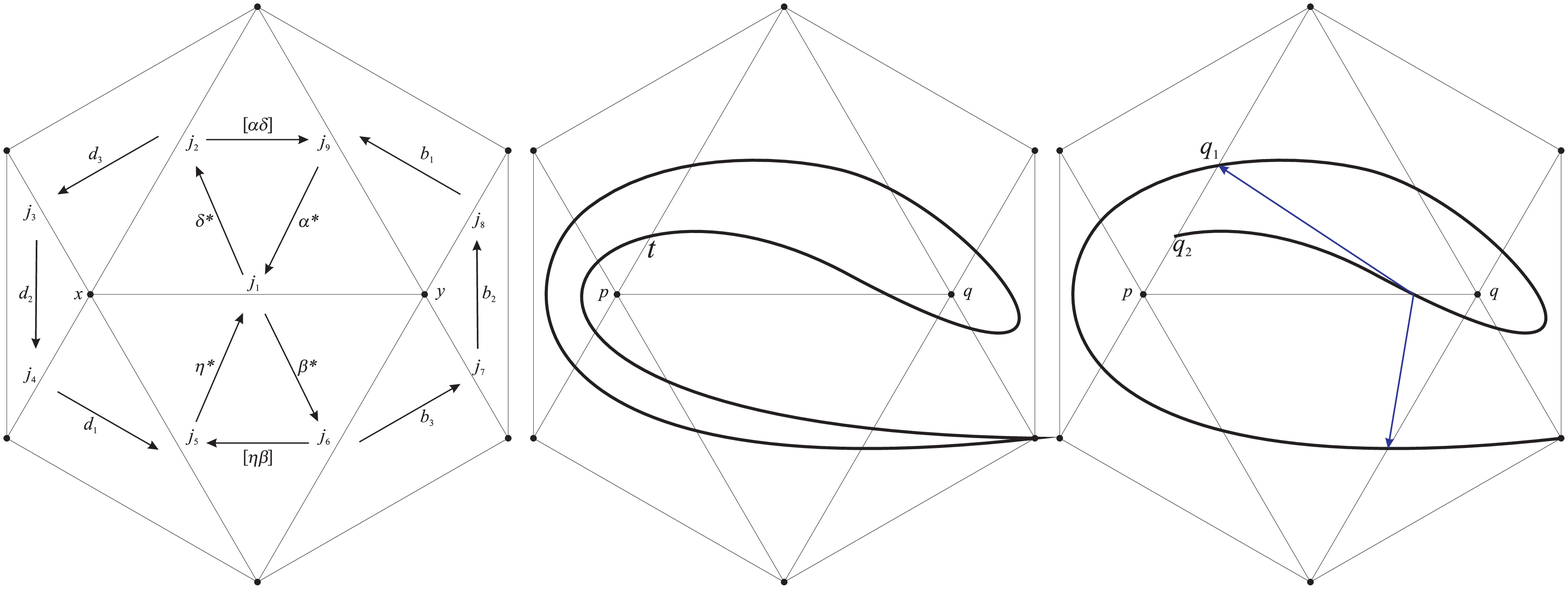}
        \end{figure}
(abusing notation, we use the same symbol $j_1$ in both $\tau$ and $\sigma$) and the following representation of its signed adjacency quiver
\begin{displaymath}
\xymatrix{
\field \ar[dd]_{1} & \field \ar[l]_{1} \ar[rr]^{1} & & \field \ar[dl]^0 & \field \ar[l]_{1} \\
 &  & \field \ar[ul]^{[-y]} \ar[dr]^{{\tiny\left[\begin{array}{c}1 \\ xy\end{array}\right]}} & &  \\
\field \ar[r]_{1} & \field \ar[ur]^0 & & \field^{2} \ar[ll]^{{\tiny\left[\begin{array}{cc}0 & 1\end{array}\right]}} \ar[r]_{{\tiny\left[\begin{array}{cc}1 & 0\end{array}\right]}} & \field \ar[uu]_{1}}
\end{displaymath}
which obviously satisfies the relations imposed by the potential $S(\sigma)=[\alpha\delta]\delta^*\alpha^*+[\eta\beta]\beta^*\eta^*+x\delta^*\eta^*d_1d_2d_3+y\beta^*\alpha^*b_1b_2b_3$.

A straightforward calculation shows that this representation can be obtained also by performing the mutation $\mu_{\arcone_1}$ to $\calMtauarc$. That is, the flip of $\arcone_1$ has the same effect as the $\arcone_1^{\operatorname{th}}$ QP-mutation.
\end{ex}

\section{Checking Jacobian relations}\label{Section:Jacobianidealsatisfied}

\subsection{Local decompositions}\label{subsec:localdecompositions}

In this subsection we shall see that our representations $\Mtauarc$ decompose \emph{locally} as the direct sum of \emph{simpler} representations that come from specific segments of the arc $\arc$. This will (relatively) simplify the proof of annihilation of $\Mtauarc$ by the Jacobian ideal, and the proof of the compatibility between flips of triangulations and mutations of representations.

Let $(Q,S)$ be an arbitrary QP, and let $j\in Q_0$ be any vertex. Define a quiver $Q(\partial)$ as follows: the set of vertices of $Q(\partial)$ consists of all the heads and tails of the arrows of $Q$ that are incident to $j$. For each $j$-hook $ab$ of $Q$, introduce one arrow $\alpha_{ab}:h(a)\rightarrow t(b)$; the set of arrows of $Q(\partial)$ consists of all the arrows of $Q$ that are incident to $j$ and all the arrows of the form $\alpha_{ab}:h(a)\rightarrow t(b)$.

Given a decorated representation $\mathcal{M}=(Q,S,\stau,M,V)$, let $M(\partial)$ be the representation of $Q(\partial)$ that attaches to each vertex of $Q(\partial)$ the same vector space $M$ attaches to it. As for the linear maps, for each arrow $a$ of $Q$ incident to $j$ let $M(\partial)_a=M_a$, and for each $j$-hook $ab$ of $Q$, let $M(\partial)_{\alpha_{ab}}=\partial_{[ab]}([S])_M:M_{h(a)}\rightarrow M_{t(b)}$.

A quick look at Subsection \ref{backgroundrepresentations} makes us see that in order to calculate the $j^{\operatorname{th}}$ mutation of $\mathcal{M}$, it is enough to apply the mutation process with respect to the data defining $M(\partial)$. The next Proposition, whose proof uses only basic linear algebra, tells us that if we decompose $M(\partial)$ as the direct sum of subrepresentations (which may be possible even when $M$ is indecomposable), then in order to calculate $\mu_{j}(\mathcal{M})$ it is enough to apply the mutation process to each of the summands of $M(\partial)$ separately.

\begin{prop}\label{prop:localdirectsum} Let $\mathcal{M}=(Q,S,M,V)$ be any decorated QP-representation. Fix a vertex $\arcone\in Q_0$ and, with respect to this vertex, define the quiver $Q(\partial)$ and the representation $M(\partial)$ as above. Suppose that $M(\partial)$ decomposes as
$$
M(\partial)=N^1\oplus\ldots\oplus N^t,
$$
where the representations $N^1,\ldots,N^t$, need not be indecomposable. For $1\leq l\leq t$ let $\mathcal{N}^l$ denote the representation $N^l$ with the zero decoration and $\mu_{\arcone}(\mathcal{N}^l)=(\overline{N^l},\overline{V^l})$ denote the decorated representation obtained from $\mathcal{N}^l$ by applying the mutation process with respect to the data $\al_{N^l}$, $\be_{N^l}$, $\ga_{N^l}$ of $N^l$. Then the mutation $\mu_{\arcone}(\mathcal{M})$ is isomorphic, as a representation of $\mu_{\arcone}(Q,S)$, to the direct sum of $\mu_{\arcone}(Q,S,0,V)$ and $(\mu_{\arcone}(Q,S),M_N,V_N)$, where $M_N$ is the representation obtained from
$\overline{N^1}\oplus\ldots\oplus\overline{N^t}$ by remembering the spaces and maps attached by $M$ to the arrows of $Q$ not incident to $\arcone$, and $V_N$ is the decoration obtained from $\overline{V^1}\oplus\ldots\oplus\overline{V^t}$ by attaching the zero vector space to each of the vertices of $Q$ that are not head or tail of an arrow incident to $\arcone$.
\end{prop}

\begin{remark} Notice that the representation $M(\partial)$ can be defined even when $M$ does not satisfy the Jacobian relations imposed by $S$.
\end{remark}

In the case when $M$ is an arc representation $\Mtauarc$, we can find a decomposition of $M(\partial)=\Mtauarc(\partial)$ as follows (here we adopt the notation of Figure \ref{paper2case1}).
\begin{figure}[!h]
                \caption{}\label{paper2case1}
                \centering
                \includegraphics[scale=.3]{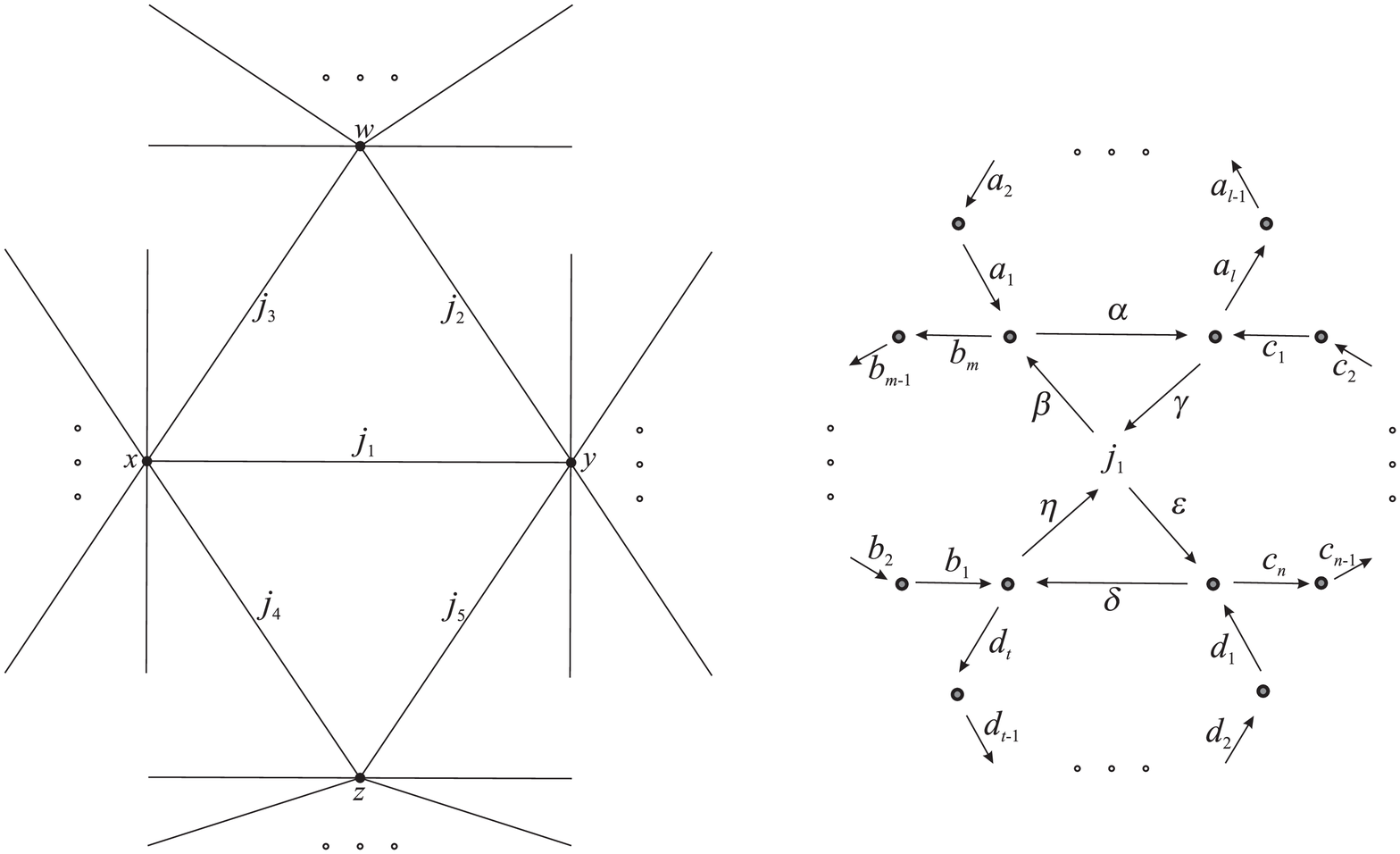}
        \end{figure}\\
Let
$$
\iota=\begin{cases}\iota_{\tau,\arc} \ \ \ \text{if $\arc$ is an arc cutting out a once-punctured monogon};\\
\arc \ \ \ \ \text{otherwise}.
\end{cases}
$$
Also, let $G(\partial)$ be the (unoriented) graph whose vertices are the points of intersection of $\iota$ with each of the arcs $\arcone_1$, $\arcone_2$, $\arcone_3$, $\arcone_4$ and $\arcone_5$. An edge of $G(\partial)$ will be either
\begin{enumerate}
\item\label{Item:segment} a segment of $\iota$ that connects a pair of vertices of $G(\partial)$ and is parallel (that is, homotopic) to either of the following paths on $\qtau$:
$$
\alpha, \ \beta, \ \gamma, \ \delta, \ \varepsilon, \ \eta, \ a_1\ldots a_l, \ b_1\ldots b_m, \ c_1\ldots c_n, \ d_1\ldots d_t; \ \text{or}
$$
\item\label{Item:combinededge} a curve parallel to either of the paths
$$
a_1\ldots a_l, \ b_1\ldots b_m, \ c_1\ldots c_n, \ d_1\ldots d_t,
$$
and obtained as the union of an $n$-detour $d^n$ (for any $n$) whose beginning point lies on $\arcone_2\cup\arcone_3\cup\arcone_4\cup\arcone_5$, and a segment of $\iota$ having one extreme at the ending point of $d^n$ and another at one of the intersection points of $\arc$ with $\tau$.
\end{enumerate}

Let $H^1,\ldots,H^t$ be the connected components of $G(\partial)$. For each connected component $H^s$ and each arc $\arcone_r$, $1\leq s\leq t$, $1\leq r\leq 5$, let $N^s_{\arcone_r}$ be the vector subspace of $\Mtauarc_{\arcone_r}$ spanned by the intersection points of $\iota$ with $\arcone_r$ that lie on $H_s$ and are different from the point $t$ of Figure \ref{deletingsegment} (in case $\arc$ is a loop cutting out a once-punctured monogon). It is easy to check that each $N^s$ is a subrepresentation of $M(\partial)=\Mtauarc(\partial)$ and that
$M(\partial)=N^1\oplus\ldots\oplus N^t$.

To know how the representations $N^1,\ldots,N^t$ can be, it suffices to find all possibilities for the components $H^1,\ldots,H^t$.

A connected component $H_s$ that contains an edge like the one described in number (\ref{Item:combinededge}) above must coincide with one of the graphs depicted in Figures \ref{Fig:component1} and \ref{Fig:component1punctmonogon}.
\begin{figure}[!h]
                \caption{}\label{Fig:component1}
                \centering
                \includegraphics[scale=.2]{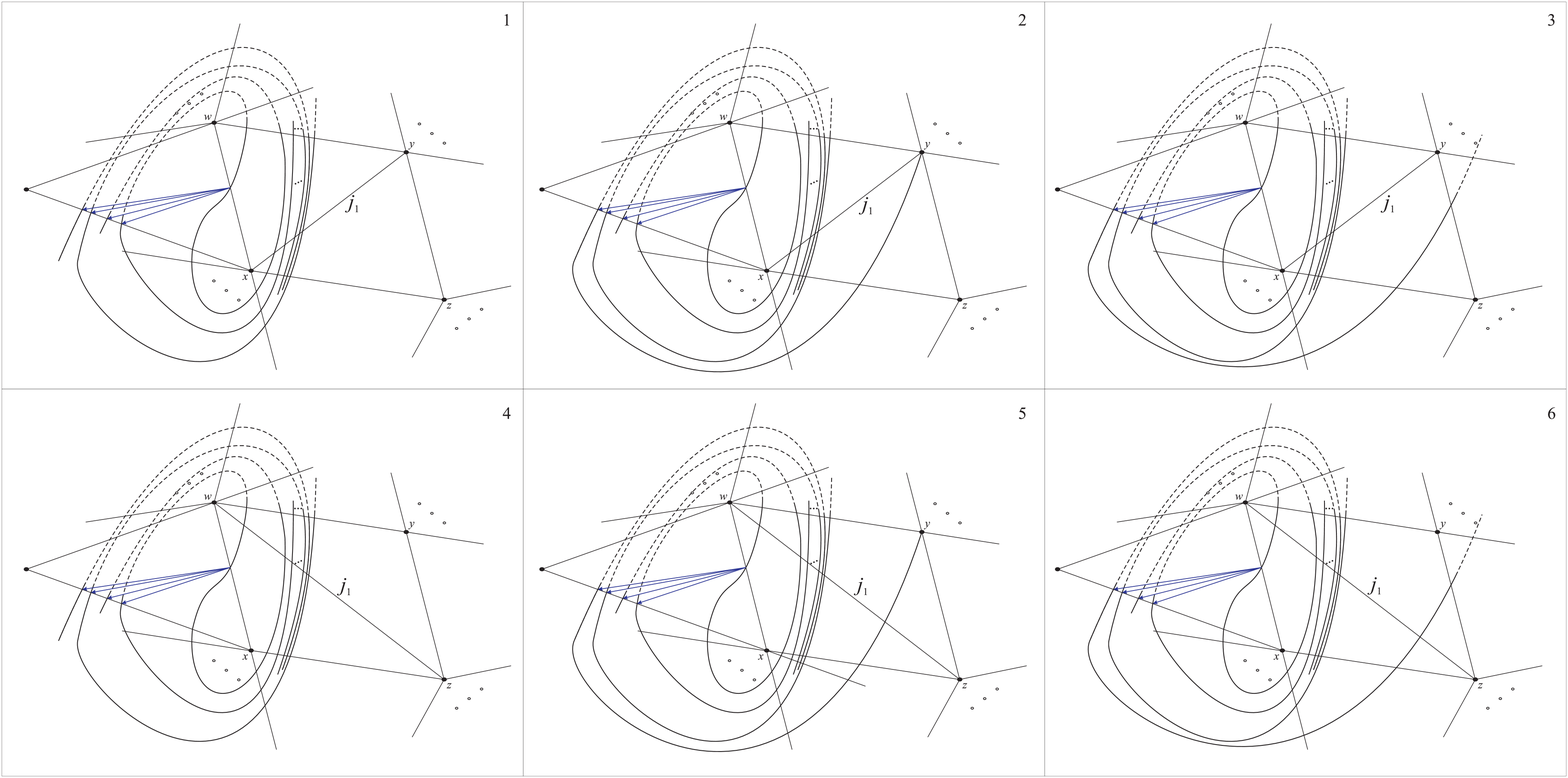}
        \end{figure}
\begin{figure}[!h]
                \caption{}\label{Fig:component1punctmonogon}
                \centering
                \includegraphics[scale=.2]{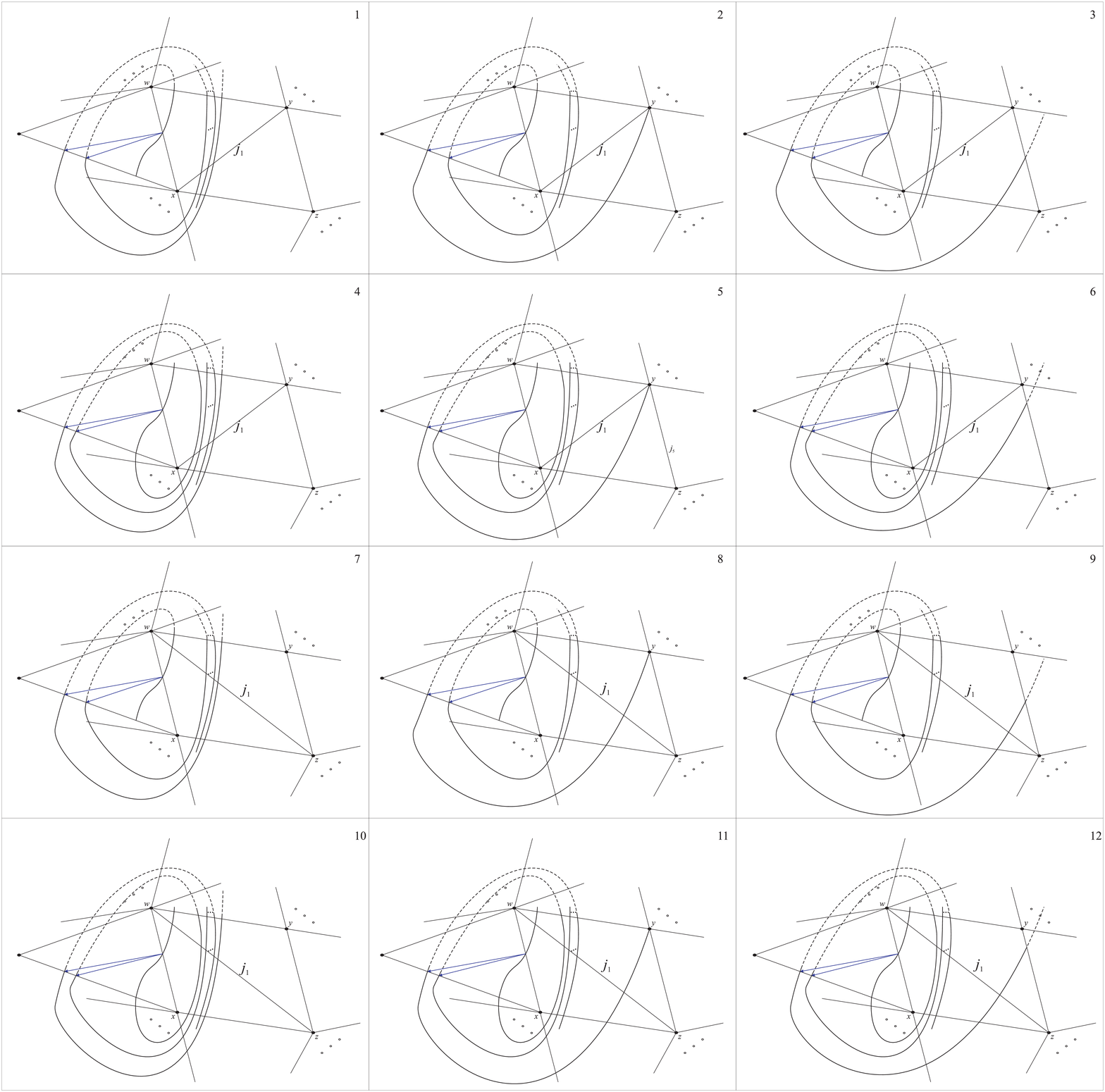}
        \end{figure}

If a connected component $H_s$ does not contain an edge as in number (\ref{Item:combinededge}) above, then there are two obvious possibilities: either there is a detour of $(\tau,\arc)$ connecting vertices of $H_s$, or there is not. If there is not, then $H_s$ must look as either of the curves depicted in Figure \ref{Fig:component2}.
\begin{figure}[!h]
                \caption{}\label{Fig:component2}
                \centering
                \includegraphics[scale=.18]{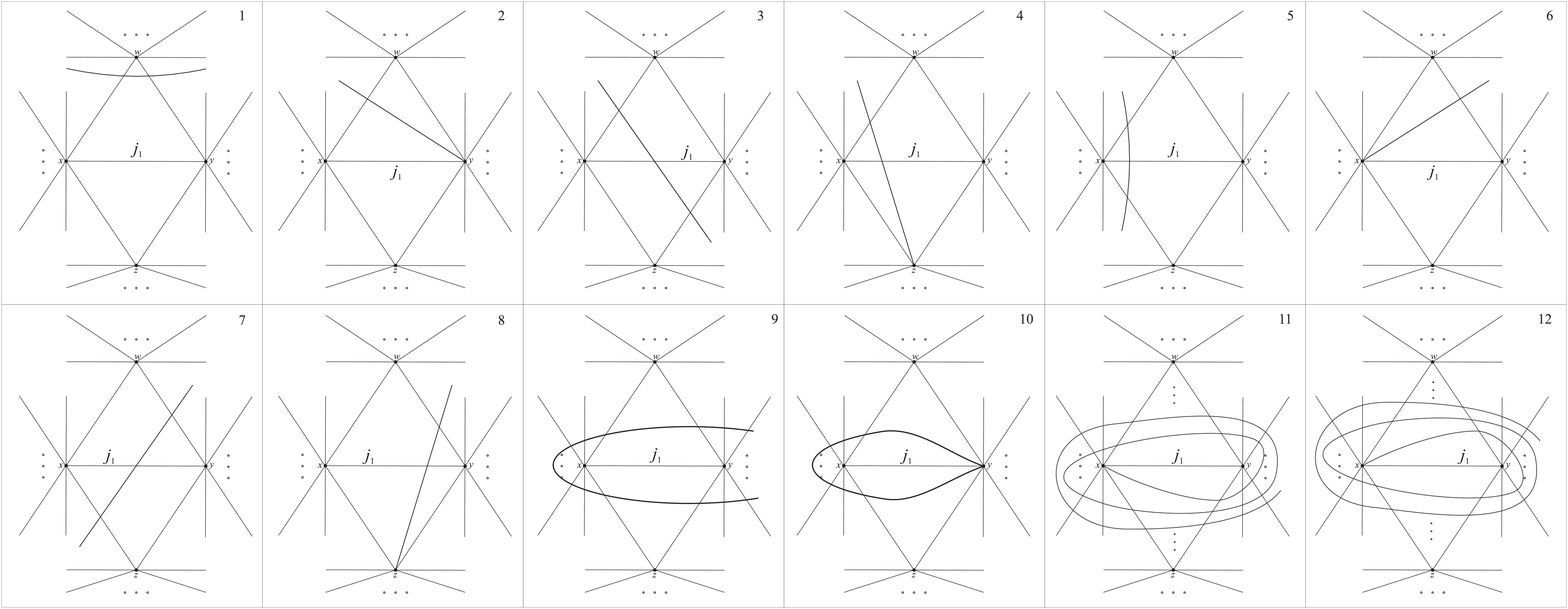}
        \end{figure}
And if there is, then $H_s$ must be one of the graphs appearing in Figures \ref{Fig:component3} and \ref{Fig:component3punctmonogon} (warning: the detours drawn in these two Figures are not part of the graph $H_s$).
\begin{figure}[!h]
                \caption{}\label{Fig:component3}
                \centering
                \includegraphics[scale=.18]{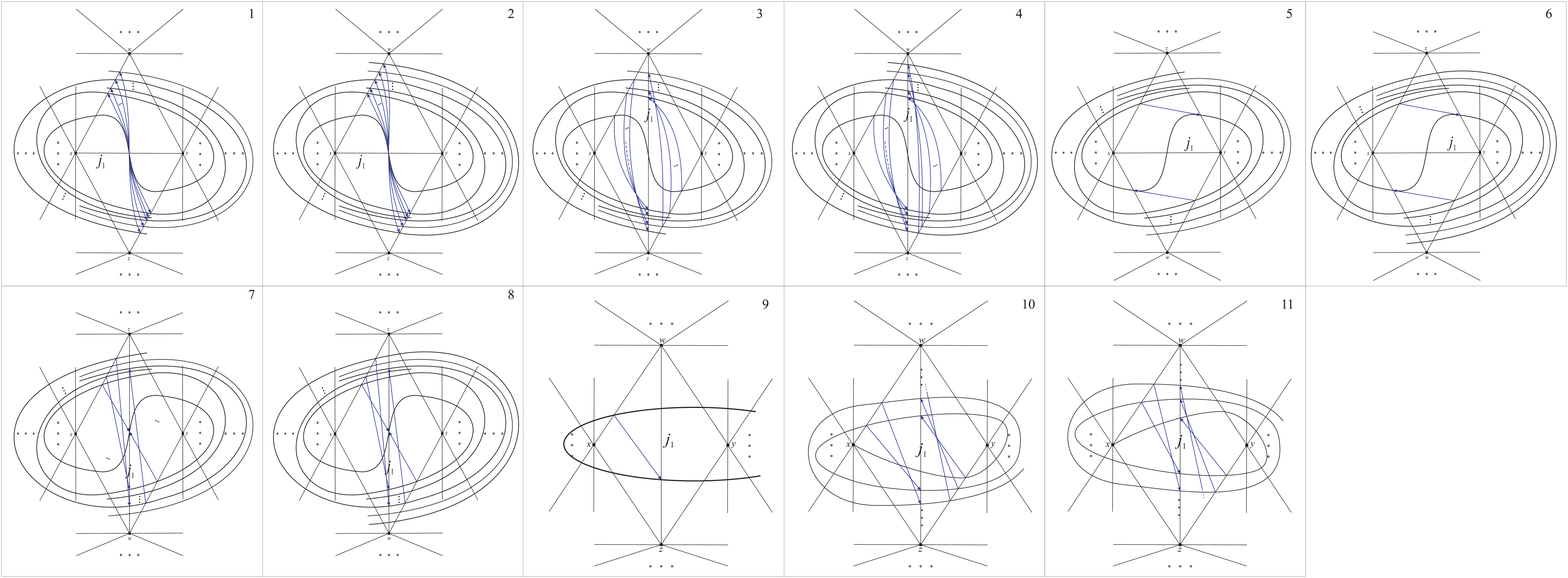}
        \end{figure}
\begin{figure}[!h]
                \caption{}\label{Fig:component3punctmonogon}
                \centering
                \includegraphics[scale=.2]{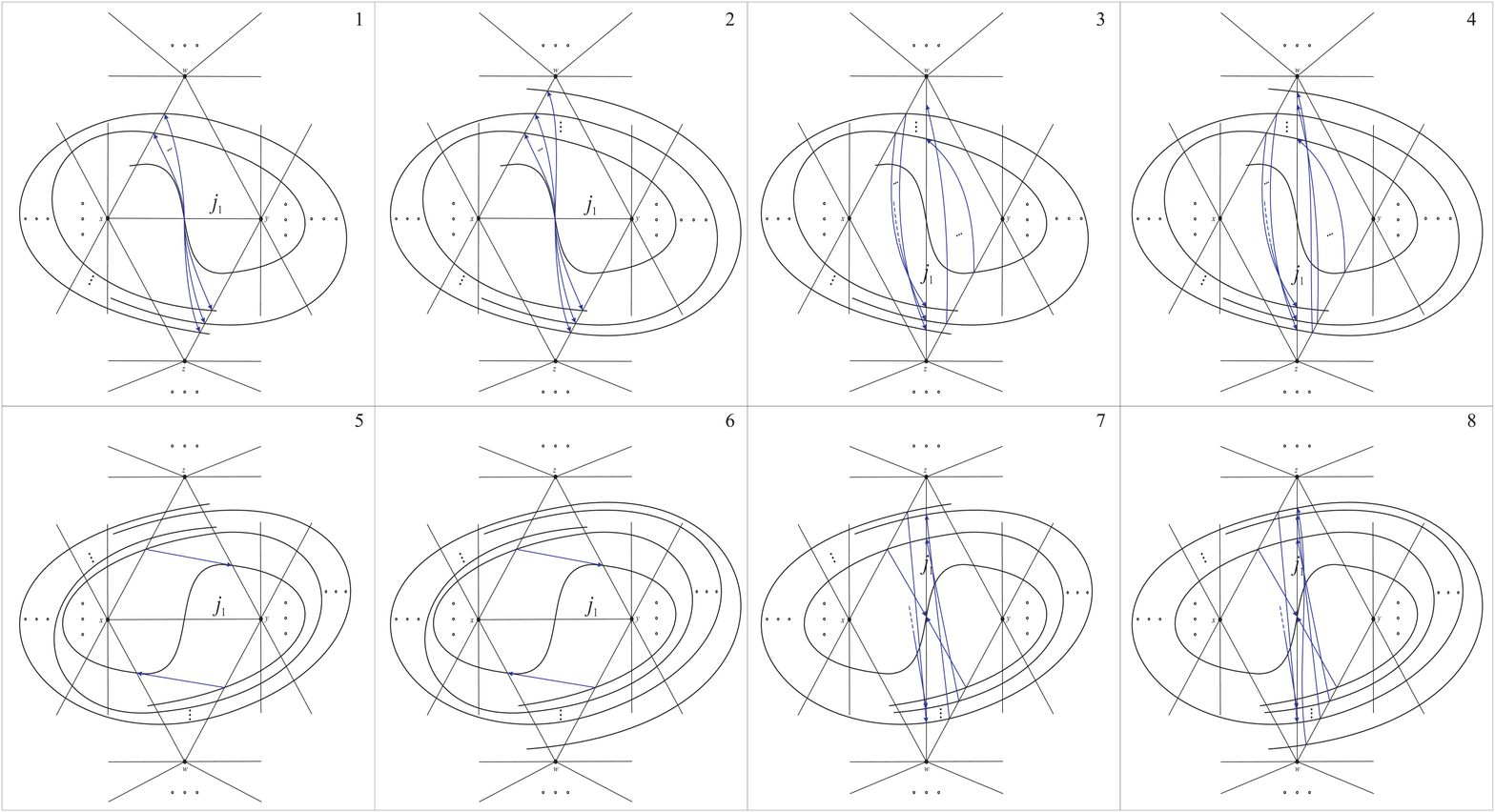}
        \end{figure}

\subsection{$\Mtauarc$ satisfies the Jacobian ideal}\label{subsec:mtauarcsatisfies}

Let us prove that our arc representations satisfy the Jacobian ideal.

\begin{prop}\label{jacobiansatisfied} Let $\tau$ be an ideal triangulation without self-folded triangles and $\arc$ be an arc. If $\arc$ satisfies(\ref{transversally}) and (\ref{minimalintersection}), then $\Mtauarc$ is a nilpotent representation of $\qtau$ that satisfies the cyclic derivatives of $\unredstau$ and is therefore annihilated by the Jacobian ideal $J(\unredstau)$.
\end{prop}

\begin{proof} We give the proof in case $\arc$ is not a loop cutting out a once-punctured monogon; the other case is similar.

For each $\arcone\in\tau$, let $n_\arcone$ be the total number of 1-detours of $(\tau,\arc)$ whose beginning point lies on $\arcone$. Let $(\arcone_1,\ldots,\arcone_r)$ be any ordering of the arcs of $\tau$. We are going to recursively define representations $M_0,\ldots,M_r$, of $\qtau$, with the following properties:
\begin{equation}
M_0=\mtauarc, \ M_r=\Mtauarc,
\end{equation}
\begin{equation}\label{dimWincreases}
\dim(W_l)\geq\dim(W_{l-1})+n_{\arcone_l} \ \text{for} \ 1\leq l\leq r,
\end{equation}
where $W_l=\{w\in M_l\suchthat J(S(\tau))w=0\}$ is the maximal vector subspace of $M_l$ satisfying the cyclic derivatives of $S(\tau)$. The proposition will then be a consequence of the fact that
\begin{equation}\label{dimW0controlled}
\dim(W_0)\geq\dim(\Mtauarc)-\underset{\arcone\in\tau}{\sum}n_\arcone.
\end{equation}
In all the representations $M_l$, $0\leq l\leq r$, the vector space attached to each $\arcone\in\tau$ will be $\Mtauarc_\arcone$. We define $M_0=\mtauarc$. For $0\leq l\leq r-1$, once $M_l$ has been constructed, let $\alpha^{l+1}_1$ and $\alpha^{l+1}_2$ be the arrows of $\unredqtau$ that have $\arcone_{l+1}$ as tail. Define $M_{l+1}$ as follows:
\begin{equation}
(M_{l+1})_{\alpha^{l+1}_1}=(D_{\arc,h(\alpha^{l+1}_1)}^{\triangle^{\alpha^{l+1}_1}})(\mtauarc_{\alpha^{l+1}_1}), \ (M_{l+1})_{\alpha^{l+1}_2}=(D_{\arc,h(\alpha^{l+1}_2)}^{\triangle^{\alpha^{l+1}_2}})(\mtauarc_{\alpha^{l+1}_2}),
\end{equation}
$$
\text{and} \ (M_{l+1})_{a}=(M_l)_a \ \text{for} \ a\notin\{\alpha^{l+1}_1,\alpha^{l+1}_2\}.
$$
We obviously have $M_r=\Mtauarc$. We have to prove that $\dim(W_l)\geq\dim(W_{l-1})+n_{\arcone_l}$ for $1\leq l\leq r$. Notice that $M_{l+1}\neq M_l$ only if at least one of $\alpha^{l+1}_1$ or $\alpha^{l+1}_2$ is parallel to a detour.

\begin{lemma}\label{Wincreases} For $0\leq l\leq r-1$, $W_l\subseteq W_{l+1}$.
\end{lemma}
\begin{proof} With the notation of Figure \ref{alpha1and2},
 \begin{figure}[!h]
                \caption{}\label{alpha1and2}
                \centering
                \includegraphics[scale=.4]{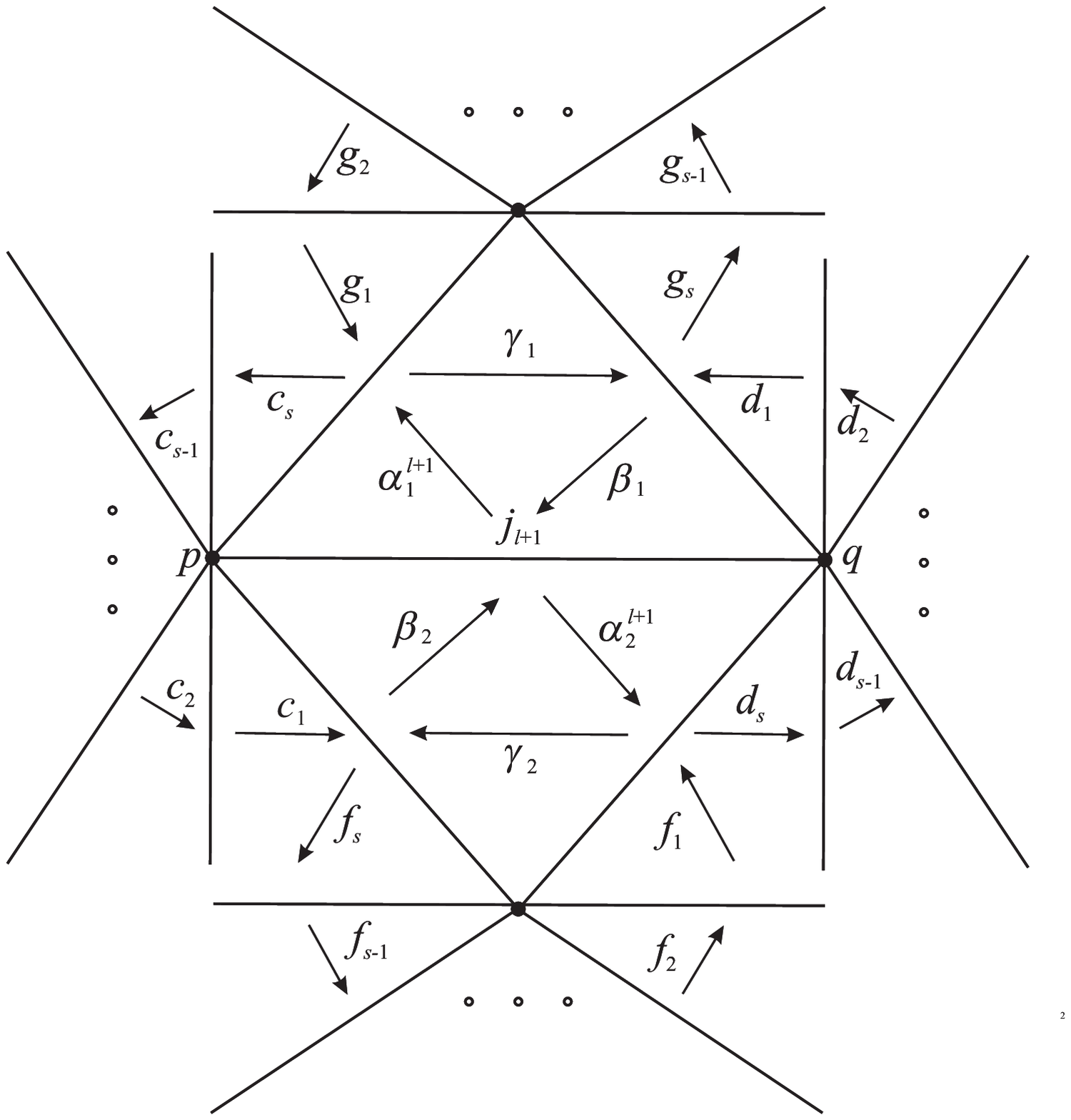}
        \end{figure}
 the arrow $\alpha^{l+1}_1$ (resp. $\alpha^{l+1}_2$) appears as a factor of two terms of the potential $\unredstau$, namely $\alpha^{l+1}_1\beta_1\gamma_1$ and $x_p\alpha^{l+1}_1\beta_2c$ (resp. $\alpha^{l+1}_2\beta_2\gamma_2$ and $x_q\alpha^{l+1}_2\beta_1d$), where we are writing $c=c_1\ldots c_{s_c}$ and $d=d_1\ldots d_{s_d}$. If we were to have $W_l\nsubseteq W_{l+1}$, then there would exist an arc $\arctwo$, a basis element $v$ of $(M_{l+1})_\arctwo$ corresponding to an intersection point of $\arc$ with $\arctwo$, and $\xi\in\rho(\arctwo)=\{\partial_a(\unredstau)\suchthat a\in Q_1(\tau), \ h(a)=\arctwo\}$ such that $\xi_{M_l}v=0$ but $\xi_{M_{l+1}}v\neq 0$. Since $M_{l+1}$ may differ from $M_l$ only by the action of $\alpha^{l+1}_1$ and $\alpha^{l+1}_2$, this would force $\xi$ to have the form $\xi=\partial_a(\unredstau)$ for some $a\in\{c_1,\ldots,c_{s_c},d_1,\ldots,d_{s_c},\gamma_1,\gamma_2,\beta_1,\beta_2\}$, and $\alpha^{l+1}_1$ or $\alpha^{l+1}_2$ (or both) would then be parallel to some detour of $(\tau,\arc)$. Therefore, Lemma \ref{Wincreases} will follow if we establish $\ker(\xi_{M_l})\subseteq\ker(\xi_{M_{l+1}})$ when $\xi$ has the form $\xi=\partial_a(\unredstau)$ for some $a\in\{c_1,\ldots,c_{s_c},d_1,\ldots,d_{s_d},\gamma_1,\gamma_2\}$, and $\xi_{M_{l+1}}=0$ when $\xi=\partial_a(\unredstau)$ for $a\in\{\beta_1,\beta_2\}$.

\begin{lemma}\label{beta1and2equalzero} $\xi_{M_{l+1}}=0$ when $\xi=\partial_a(\unredstau)$ for $a\in\{\beta_1,\beta_2\}$.
\end{lemma}

\begin{proof} We will unravel the definition of the linear maps $(M_{l+1})_{\alpha^{l+1}_1}=(D_{\arc,\arcone_{l+1}}^{\triangle^{\alpha^{l+1}_1}})(\mtauarc_{\alpha^{l+1}_1})$ and $(M_{l+1})_{\alpha^{l+1}_2}=(D_{\arc,\arcone_{l+1}}^{\triangle^{\alpha^{l+1}_2}})(\mtauarc_{\alpha^{l+1}_2})$. It is enough to check that $\partial_{\beta_1}(\unredstau)$ and $\partial_{\beta_2}(\unredstau)$ act as zero in each of the possible summands of the representation $M(\partial)$ corresponding to $\arcone_{l+1}$. In such a summand, if none of $\alpha^{l+1}_1$ and $\alpha^{l+1}_2$ is parallel to a detour, there is nothing to prove. Otherwise, we have either of the situations sketched in entries 1 and 2 of Figure \ref{Fig:component3}. Let us analyze the configuration of the latter entry, the former one being completely analogous. It is represented, with all the necessary notation, in Figure \ref{detourssatisfyjacobian},
\begin{figure}[!h]
                \caption{$\partial_{\beta_1}(\unredstau)$ and $\partial_{\beta_2}(\unredstau)$ act as zero in $M_{l+1}$}\label{detourssatisfyjacobian}
                \centering
                \includegraphics[scale=.6]{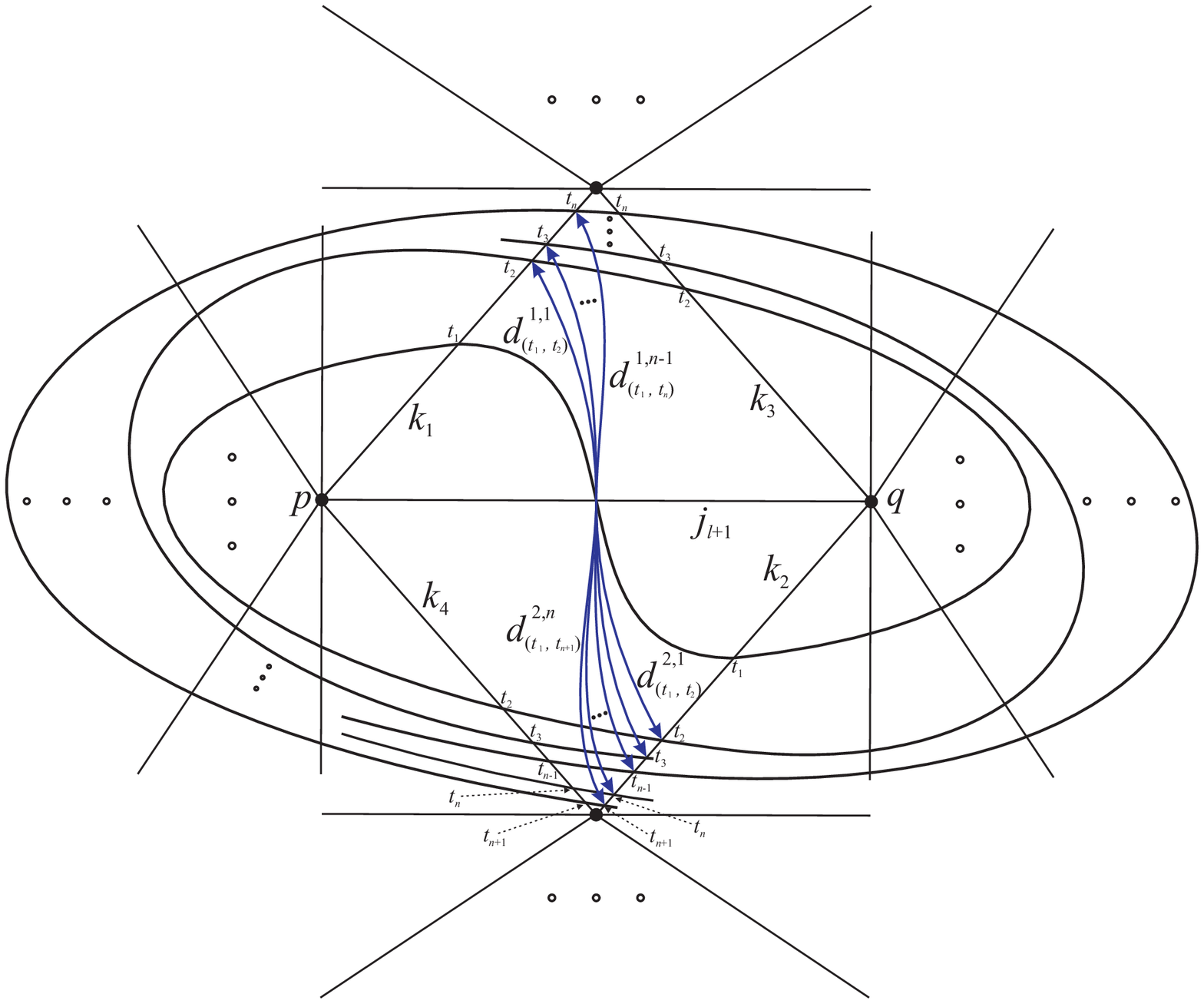}
        \end{figure}
where $\beta_1:k_3\rightarrow\arcone_{l+1}$ and $\beta_2:k_4\rightarrow\arcone_{l+1}$. Whence, with the numbering of intersection points shown in Figure \ref{detourssatisfyjacobian}, the action of $\partial_{\beta_1}(\unredstau)$ on $M_{l+1}$ is
$$
\partial_{\beta_1}(\unredstau)_{M_{l+1}}=(\gamma_1\alpha^{l+1}_1)_{M_{l+1}}+(x_qd_1\ldots d_{s_d}\alpha^{l+1}_2)_{M_{l+1}}=
$$
$$
=\left[\begin{array}{cc}0 & \mathbf{1}_{(n-1)\times (n-1)}\end{array}\right]
\left[\begin{array}{c}1\\
                      -x_q\\
                      x_px_q\\
                      \vdots\\
                      (-1)^{n-1}x_p^{\lfloor\frac{n-1}{2}\rfloor}x_q^{\lfloor\frac{n}{2}\rfloor}\end{array}\right]+x_q
\left[\begin{array}{cc}\mathbf{1}_{(n-1)\times (n-1)} & \mathbf{0}_{(n-1)\times 2}\end{array}\right]
\left[\begin{array}{c}1\\
                      -x_p\\
                      x_px_q\\
                      \vdots\\
                      (-1)^{n-1}x_p^{\lfloor\frac{n}{2}\rfloor}x_q^{\lfloor\frac{n-1}{2}\rfloor}\\
                      (-1)^nx_p^{\lfloor\frac{n+1}{2}\rfloor}x_q^{\lfloor\frac{n}{2}\rfloor}\end{array}\right]=
$$
$$
=\left[\begin{array}{c}-x_q\\
                      x_px_q\\
                      \vdots\\
                      (-1)^{n-1}x_p^{\lfloor\frac{n-1}{2}\rfloor}x_q^{\lfloor\frac{n}{2}\rfloor}\end{array}\right]+
\left[\begin{array}{c}x_q\\
                      -x_px_q\\
                      x_px_q^2\\
                      \vdots\\
                      (-1)^{n-2}x_p^{\lfloor\frac{n-1}{2}\rfloor}x_q^{\lfloor\frac{n-2}{2}\rfloor+1}\end{array}\right]=
\left[\begin{array}{c}0\\
                      \vdots\\
                      0\end{array}\right].
$$
Similarly, the action of $\partial_{\beta_2}(\unredstau)$ on $M_{l+1}$ is zero.
\end{proof}

Now we prove that $\ker(\xi_{M_{l}})\subseteq\ker(\xi_{M_{l+1}})$ when $\xi=\partial_a(\unredstau)$ for $a\in\{c_1,\ldots,c_{s_c},d_1,\ldots,d_{s_d},\gamma_1,\gamma_2\}$.

 \begin{casea} Both $\alpha^{l+1}_1$ and $\alpha^{l+1}_2$ parallel to a detour.

We already noted that if an arrow is parallel to a detour, then it is parallel to a 1-detour (See Remark \ref{remaboutdetours}). So in this case both $\alpha^{l+1}_1$ and $\alpha^{l+1}_2$ are parallel to a 1-detour. There are two ways this can happen, depending on whether $\alpha^{l+1}_1$ and $\alpha^{l+1}_2$ are parallel to detours with the same beginning point or not. See Figure \ref{bothparallel2detour}.
 \begin{figure}[!h]
                \caption{}\label{bothparallel2detour}
                \centering
                \includegraphics[scale=.35]{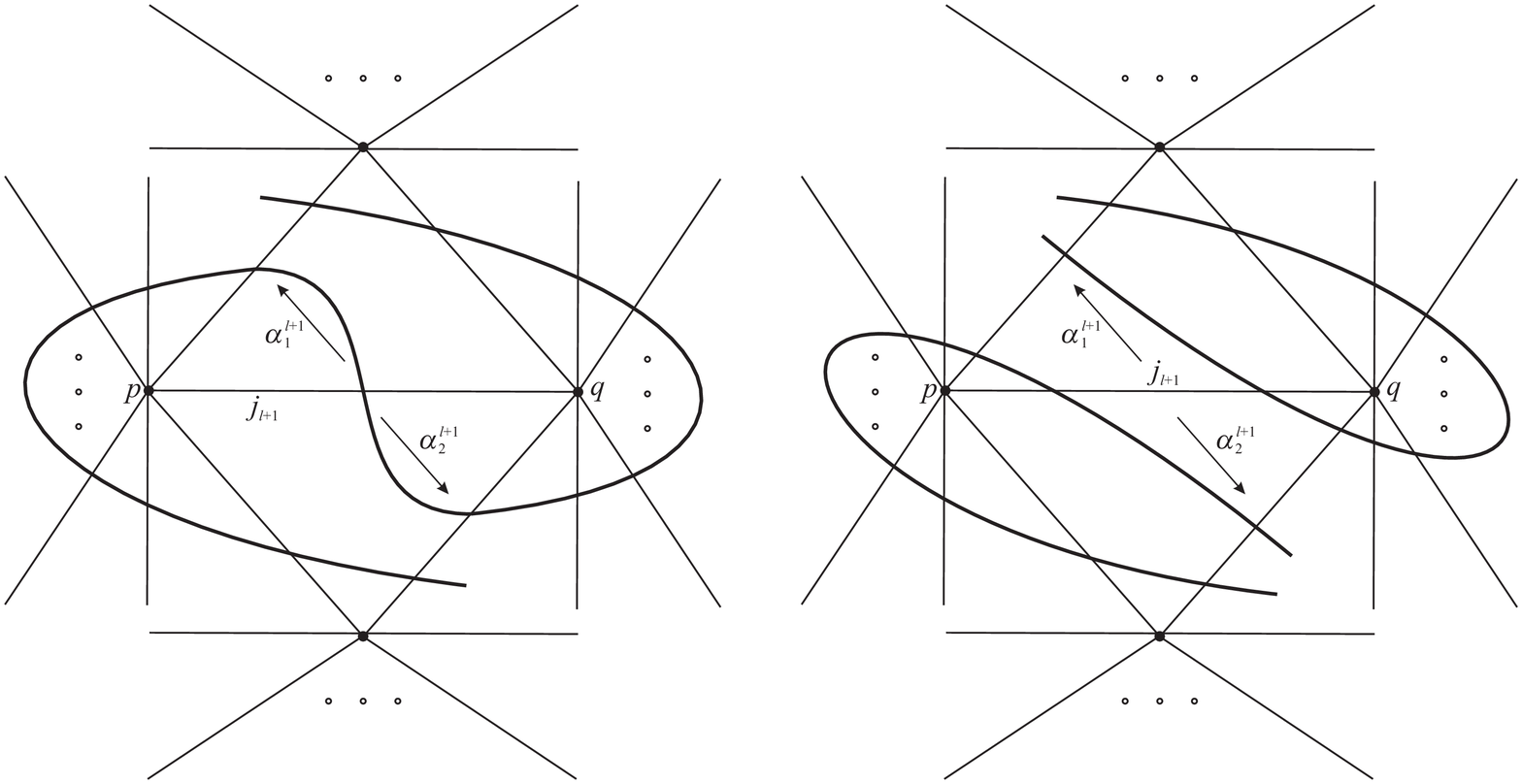}
        \end{figure}
In both situations it is easy to see that $\mtauarc_{\beta_1}$ and $\mtauarc_{\beta_2}$ are zero, which implies $(\beta_1)_{M_{l+1}}=0$ and $(\beta_2)_{M_{l+1}}=0$.

 \begin{subcase} $\xi=\partial_a(\unredstau)$ for some $a\in\{c_1,\ldots,c_{s_c},d_1,\ldots,d_{s_d}\}$.

 If $a=c_t$ then, with the notation of Figure \ref{bothparallelperipheraltriangle},
 \begin{figure}[!h]
                \caption{}\label{bothparallelperipheraltriangle}
                \centering
                \includegraphics[scale=.35]{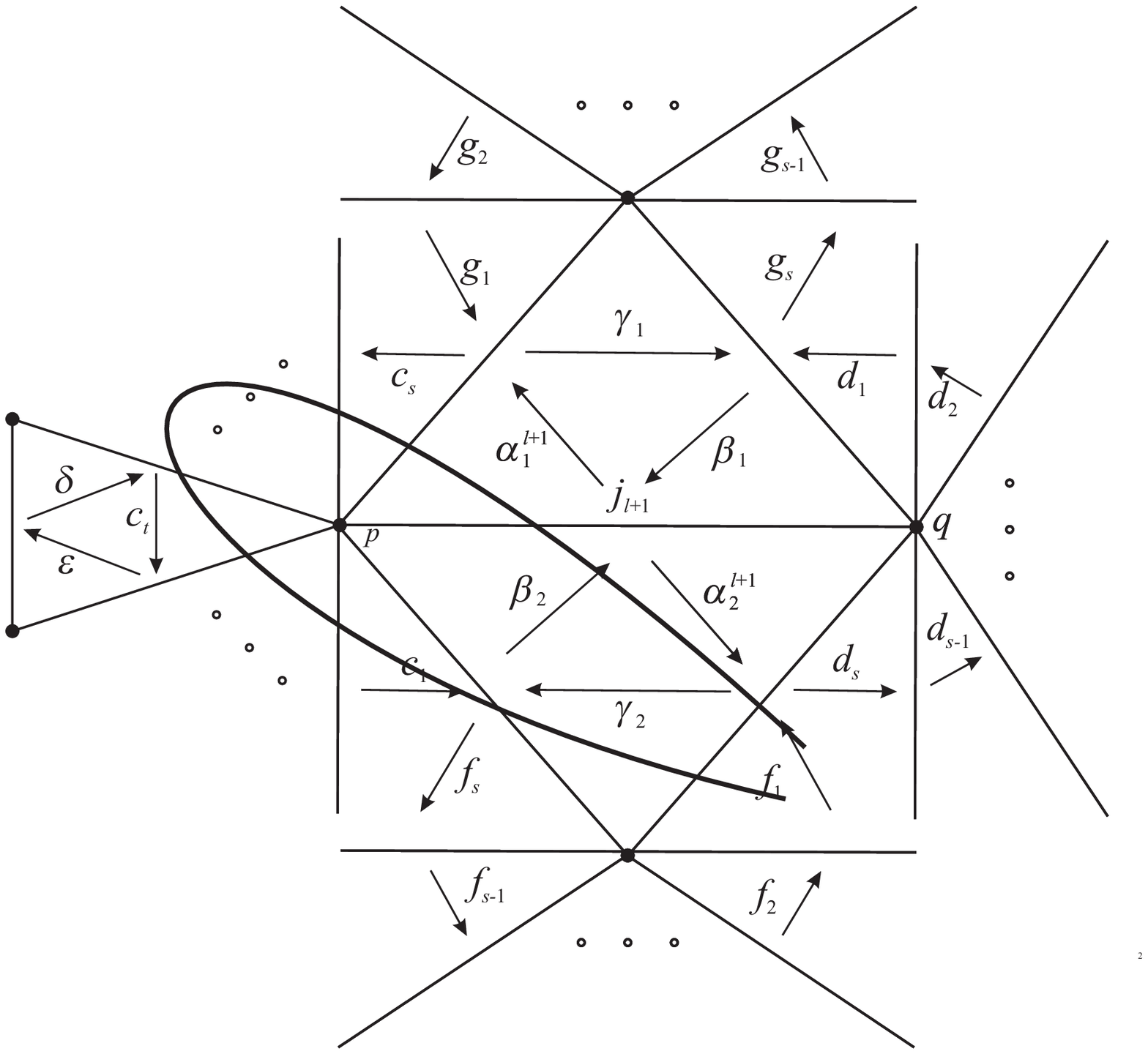}
        \end{figure}
  we have $\xi=\partial_a(\unredstau)=x_pc_{t+1}\ldots c_{s_c}\alpha^{l+1}_1\beta_2c_1\ldots c_{t-1}$ $+\delta\varepsilon$. Notice that $\varepsilon$ cannot be parallel to a detour of $(\tau,\arc)$, and this implies that the composition $(\delta\varepsilon)_{M_{l+1}}$ is zero. Therefore, $\xi_{M_{l+1}}=0$. The situation $a\in\{d_1,\ldots,d_{s_d}\}$ leads to a similar conclusion.
 \end{subcase}

 \begin{subcase} $\xi=\partial_a(\unredstau)$ for $a\in\{\gamma_1,\gamma_2\}$.

 Since $\mtauarc_{\beta_1}=0$, every intersection of $\arc$ with $t(\beta_1)$ is part of a segment of $\arc$ with endpoints in $t(\beta_1)$ and $t(\gamma_1)$, and this implies that the composition $(g_1\ldots g_{s_g})_{M_{l+1}}$ is zero (even if some $g_t$ is parallel to some detour). Thus $\partial_{\gamma_1}(\unredstau)_{M_{l+1}}=0$. A similar argument shows that $\partial_{\gamma_2}(\unredstau)_{M_{l+1}}=0$
 \end{subcase}
 \end{casea}

 \begin{casea} $\alpha^{l+1}_1$ parallel to a detour, $\alpha^{l+1}_2$ not parallel to any detour.

 \begin{subcasea} $\xi=\partial_a(\unredstau)$ for some $a\in\{c_1,\ldots,c_{s_c}\}$.

 The image of $(\beta_{2})_{M_{l+1}}$ has zero intersection with the vector subspace of $(M_{l+1})_{j_{l+1}}$ spanned by (the basis vectors corresponding to) intersection points of $\arc$ with $\arcone_{l+1}$ that are beginning points of detours. And the linear maps $(\alpha^{l+1}_{1})_{M_l}$ and $(\alpha^{l+1}_{1})_{M_{l+1}}$ agree on the rest of basis vectors of $(M_{l+1})_{\arc_{l+1}}$ that correspond to intersection points of $\arc$ with $\arcone_1$. Therefore $\ker(\xi_{M_l})\subseteq\ker(\xi_{M_{l+1}})$.
 \end{subcasea}

 \begin{subcasea} $\xi=\partial_a(\unredstau)$ for some $a\in\{d_1,\ldots,d_{s_d}\}$.

 Since $\alpha^{l+1}_2$ is not parallel to any detour of $(\tau,\arc)$, we have $\partial_a(\unredstau)_{M_{l}}=\partial_a(\unredstau)_{M_{l+1}}$.
 \end{subcasea}

 \begin{subcasea} $\xi=\partial_a(\unredstau)$ for $a\in\{\gamma_1,\gamma_2\}$.

 Again, since $\alpha^{l+1}_2$ is not parallel to any detour of $(\tau,\arc)$, we have $\partial_{\gamma_2}(\unredstau)_{M_{l}}=\partial_{\gamma_2}(\unredstau)_{M_{l+1}}$.
 \end{subcasea}
 \end{casea}

This finishes the proof of Lemma \ref{Wincreases}.
\end{proof}

Now, for each 1-detour $d^1$ whose beginning point lies on $\arcone_{l+1}$, the basis element of $\mtauarc$ corresponding to the beginning point $b(d^1)$ belongs to $W_{l+1}$ (by Lemma \ref{beta1and2equalzero}) but not to $W_l$. This fact and Lemma \ref{Wincreases} prove that property (\ref{dimWincreases}) is satisfied. The proof of property (\ref{dimW0controlled}) follows from the observation that if $v\in\mtauarc$ is a basis element corresponding to an intersection point that is not a beginning point of a 1-detour of $(\tau,\arc)$, then $v\in W_0$.

The nilpotency of $\Mtauarc$ follows by induction on $l=0,\ldots,r$. Proposition \ref{jacobiansatisfied} is proved.
\end{proof}

\begin{remark} It is not true that any representation annihilated by the cyclic derivatives of $\stau$ is nilpotent. That is, it is possible to construct $R\langle\qtau\rangle$-modules that satisfy the cyclic derivatives of $\stau$ but cannot be given the structure of $\mathcal{P}(\qtau,\stau)$-module. An example of this is given by the representation
$$
\xymatrix{
 & \field \ar@/^/@{->}[dr]^{1}\ar[dr]_ {1} & \\
\field \ar@/^/@{->}[ur]^{1}\ar[ur]_{1} &  & \field \ar@/^/@{->}[ll]^{1}\ar[ll]_{1}}
$$
of the the quiver
$$
\xymatrix{
 & 3 \ar@/^/@{->}[dr]^{c_1}\ar[dr]_ {c_2} & \\
1 \ar@/^/@{->}[ur]^{a_1}\ar[ur]_{a_2} &  & 2 \ar@/^/@{->}[ll]^{b_1}\ar[ll]_{b_2}}
$$
which obviously satisfies the cyclic derivatives of the potential $S=a_1b_1c_1+a_2b_2c_2-a_1b_2c_1a_2b_1c_2$, but is not nilpotent.
\end{remark}

Now we know that $\Mtauarc$ is annihilated by the Jacobian ideal $J(\unredstau)\subseteq R\langle\langle \unredqtau\rangle\rangle$. By definition of $\qstau$, there is a right-equivalence $\varphi$ between $(\unredqtau,\unredstau)$ and the direct sum of $\qstau$ with a trivial QP. By Remark \ref{rem:restrictingaction}, the action of $R\langle\langle \qtau\rangle\rangle$ on $\Mtauarc$ induced by $\varphi$ is given by simply forgetting the action of the 2-cycles of $\unredqtau$ on $\Mtauarc$. Moreover, under this action, $\Mtauarc$ is annihilated by the Jacobian ideal $J(\stau)$.

To close the section, let us decorate the representations defined in Subsections \ref{case1.1} and \ref{subsec:cutsout}.

\begin{defi} Let $\tau$ be an ideal triangulation of $\surf$ and $\arc$ be any arc on $\surf$.
We define the \emph{decorated arc representation} $\calMtauarc$ to be $(Q(\tau),S(\tau),\Mtauarc,V(\tau,\arc))$, where $V(\tau,\arc)_{\arcone}=\delta_{\arc,\arcone}\field$ ($\delta_{\arc,\arcone}$ being the \emph{Kronecker delta}). In other words, $\calMtauarc$ is the arc representation $\Mtauarc$ with the zero decoration if $\arc\notin\tau$, and is the $\arc^{\operatorname{th}}$ negative simple representation if $\arc\in\tau$.
\end{defi}

\section{Flip $\leftrightarrow$ mutation compatibility}\label{Section:Flip-mutationcompatibility}

We now turn to investigate the compatibility between flips of triangulations and mutations of representations. Throughout this section we will be interested in flipping the arc $\arcone$ of $\tau$. We will work under the assumption that none of the ideal triangulations $\tau$ and $\sigma=f_\arcone(\tau)$ has self-folded triangles.

\subsection{Effect of flips on detour matrices}\label{subsec:effectonmatrices}

From their very definition, detour matrices depend on the triangles of $\tau$. Let us be more specific; take an arc $\arcone\in\tau$ and let $\triangle_1$ and $\triangle_2$ be the triangles of $\tau$ that contain $\arcone$, let also $\diamondsuit=\triangle_1\cup\triangle_2$ be the quadrilateral in $\tau$ of which $\arcone$ is a diagonal. Given an arc $\arctwo\in\tau$, $\arctwo\neq\arcone$, $\arctwo\subseteq\triangle_1$, the detour matrix $D^{\triangle_1}_{\arc,\arctwo}$ has been defined with respect to $\tau$, but it does not even make sense to talk of such the matrix $D^{\triangle_1}_{\arc,\arctwo}$ with respect to $\sigma=f_{\arcone}(\tau)$ because $\triangle_1$ is not a triangle of $\sigma$. This of course does not mean that the arc $\arctwo$ does not have two detour matrices attached according to $\sigma$, but rather that, strictly speaking, we should use some notation like $D^{\triangle,\tau}_{\arc,\arctwo}$ to indicate the dependence on the triangulation (we will do so only when it is really necessary).

But the above mentioned change of detour matrices is not the only expectable one when we flip $\arcone$: the existence of many detours contained in the triangles of $\tau$ adjacent to $\diamondsuit$ (which in most cases will remain triangles of $\sigma$) depends on the existence of detours contained in $\triangle_1$ or $\triangle_2$. So, in principle, the existence could be possible of a triangle $\triangle$ present in both $\tau$ and $\sigma$ and an arc $\arctwo\in\tau\cap\sigma$, $\arctwo\subseteq\triangle$, such that the detour matrices $D^{\triangle,\tau}_{\arc,\arctwo}$ and $D^{\triangle,\sigma}_{\arc,\arctwo}$ were different. This would ultimately lead to the existence of an arrow $a$ not incident to $\arcone$ (hence belonging to both $\qtau$ and $\qsigma$) such that the linear maps $\Mtauarc_a$ and $\Msigmaarc_a$ do not coincide. This subsection is devoted to show that this does not happen, that is, that the detour matrices that should not change actually do not. For time and space reasons, we show this only when $\arc$ is not a loop cutting out a once-punctured monogon, and leave to the reader the task of doing the necessary checks when $\arc$ is such a loop.

\begin{lemma}\label{flipdoesnotaffect1} If the arc $\arctwo$ is not contained in any of the ideal triangles that contain $\arcone$, then the flip of $\arcone$ does not affect any of the two detour matrices attached to $\arctwo$.
\end{lemma}

\begin{proof} Let $\sigma=f_\arcone(\tau)$ be the ideal triangulation obtained from $\tau$ by flipping $\arcone$. Since $\arctwo$ is not
contained in any of the ideal triangles of $\tau$ that contain $\arcone$, both of the ideal triangles of $\tau$ that contain $\arctwo$ are ideal triangles
of $\sigma$ as well. Fix one such triangle $\triangle$, and denote by $D^{\triangle,\tau}_{\arc,\arctwo}$ (resp.
$D^{\triangle,\sigma}_{\arc,\arctwo}$) the detour matrix attached to $\arctwo$ using the detours of $(\tau,\arc)$ (resp. $(\sigma,\arc)$) that
are contained in $\triangle$. The assertion of the lemma is that $D^{\triangle,\tau}_{\arc,\arctwo}=D^{\triangle,\sigma}_{\arc,\arctwo}$.

Let $\arctwo_1$ be the (unique) arc contained in $\triangle$ such that there is an arrow $\alpha:\arctwo_1\rightarrow\arctwo$
in $\unredqtau$. Let $\triangle_1$ be the (unique) triangle of $\tau$ that contains $\arctwo_1$ and is different from $\triangle$ (see Figure \ref{adjtriangs}).
\begin{figure}[!h]
                \caption{}\label{adjtriangs}
                \centering
                \includegraphics[scale=.35]{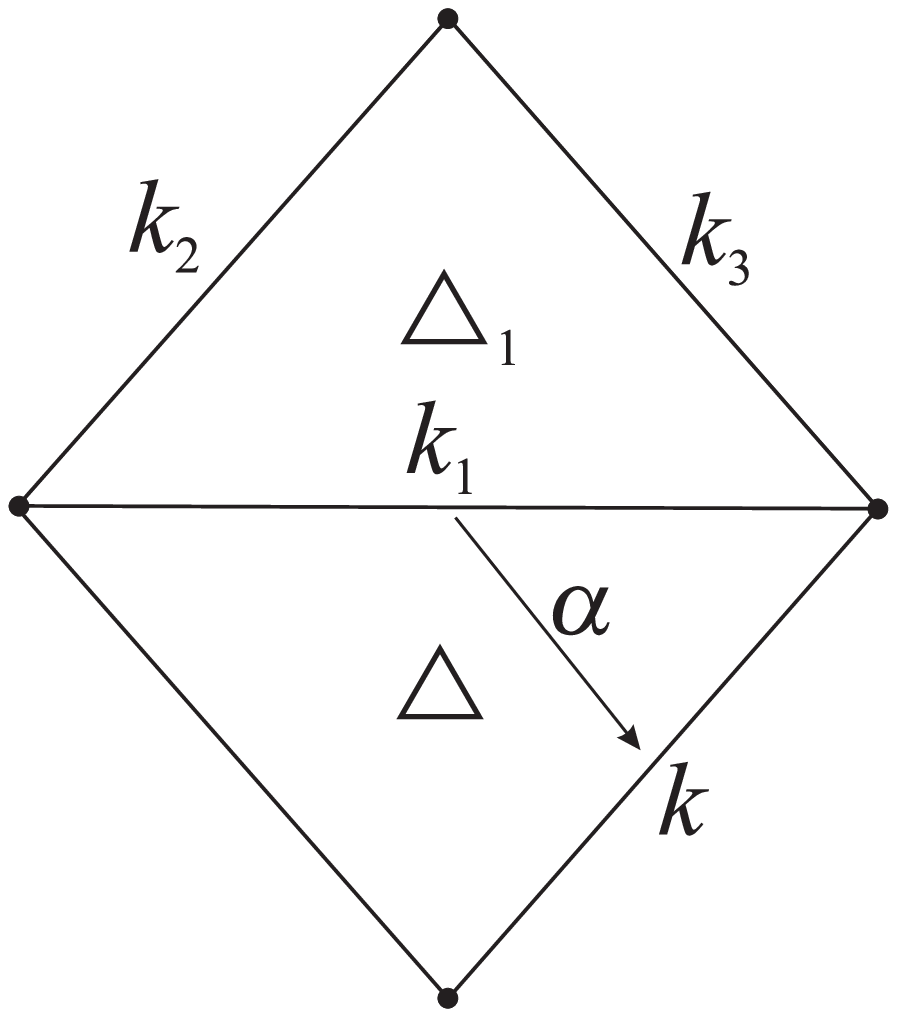}
        \end{figure}
Since $\arctwo$ is not contained in any of the ideal triangles of $\tau$ that contain $\arcone$, we have $\arcone\neq\arctwo_1$. If $\arcone\notin\{\arctwo_2,\arctwo_3\}$, then all the $n$-detours of $(\tau,\arc)$ whose beginning point lies on $\arctwo_1$ are $n$-detours of $(\sigma,\arc)$, and clearly $D^{\triangle,\tau}_{\arc,\arctwo}=D^{\triangle,\sigma}_{\arc,\arctwo}$. To see what happens when $\arcone\in\{\arctwo_2,\arctwo_3\}$, let us begin assuming that $\arcone=\arctwo_2$, and that the detours of $(\tau,\arc)$ are determined in terms of the situation described in Figure \ref{detmatricesnochange1}.
\begin{figure}[!h]
                \caption{}\label{detmatricesnochange1}
                \centering
                \includegraphics[scale=.5]{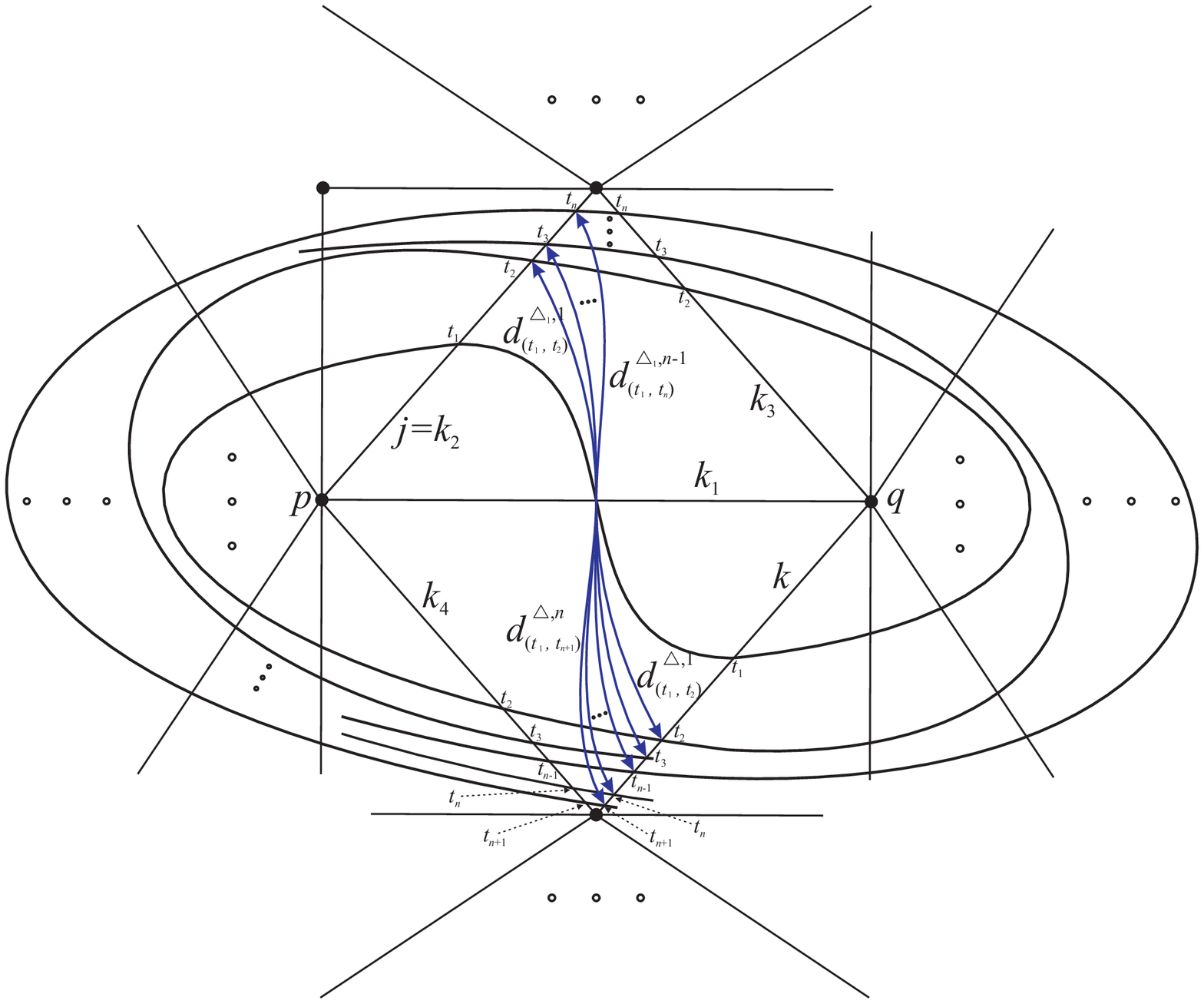}
        \end{figure}

When we flip $\arcone$ we get the configuration shown in Figure \ref{detmatricesnochange2}.
\begin{figure}[!h]
                \caption{}\label{detmatricesnochange2}
                \centering
                \includegraphics[scale=.5]{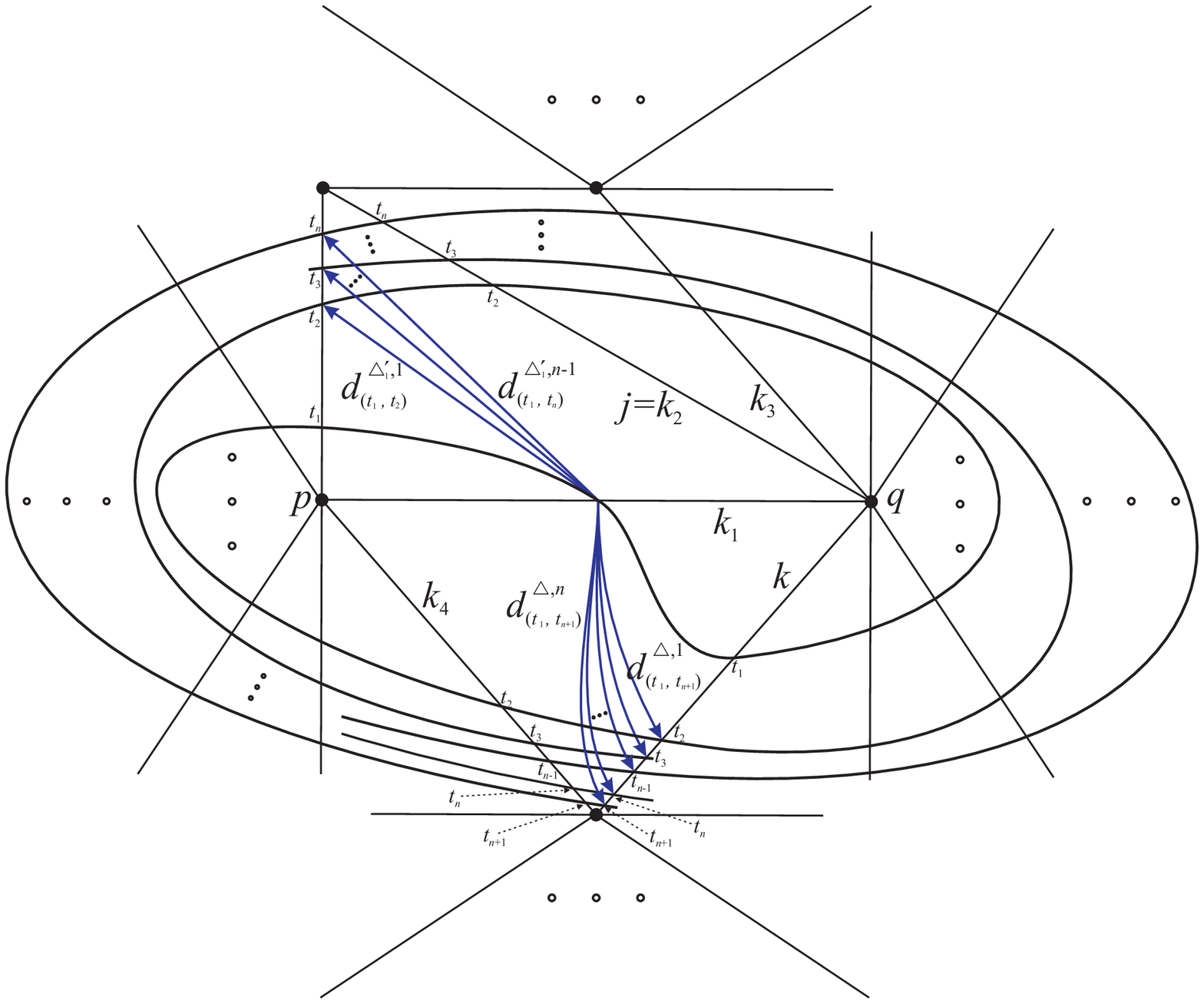}
        \end{figure}
The detours of $(\tau,\arc)$ contained in the triangle $\triangle$ are detours of $(\sigma,\arc)$. More precisely, if $(d^{\triangle,m}_{(t_1,t_m)})$ is an $m$-detour of $(\tau,\arc)$, then it is also an $m$-detour of $(\sigma,\arc)$; conversely, all $m$-detours of $(\sigma,\arc)$ contained in $\triangle$ are $m$-detours of $(\tau,\arc)$. Moreover, if $d^{\triangle,m}_{(t_1,t_m)}$ detours the puncture $p$ with respect to $(\tau,\arc)$, then it detours $p$ with respect to $(\sigma,\arc)$ as well. Therefore, the $t_l^{\operatorname{th}}$ column of $D^{\triangle,\tau}_{\arc,\arctwo}$ coincides with the $t_l^{\operatorname{th}}$ column of $D^{\triangle,\sigma}_{\arc,\arctwo}$ for $l=1,\ldots,n+1$. Applying this argument to each connected component of the graph $G(\partial)$ defined with respect to $\arctwo_1$ (see Subsection \ref{subsec:localdecompositions}), we obtain $D^{\triangle,\tau}_{\arc,\arctwo}=D^{\triangle,\sigma}_{\arc,\arctwo}$.

A similar argument also proves that $D^{\triangle,\tau}_{\arc,\arctwo}=D^{\triangle,\sigma}_{\arc,\arctwo}$ when $\arcone=\arctwo_3$.
\end{proof}

\begin{lemma}\label{flipdoesnotaffect2} Let $\triangle$ be one of the ideal triangles of $\tau$ that contain $\arcone_1$ and $\arctwo\in\tau$, $\arctwo\neq\arcone$, be an arc contained in $\triangle$. If $\triangle'$ denotes the unique triangle of $\tau$ that contains $\arctwo$ and is different from $\triangle$, then $\triangle'$ is a triangle of $\sigma=f_{\arcone_1}(\tau)$ and the detour matrices $D^{\triangle',\tau}_{\arc,\arctwo}$ and $D^{\triangle',\sigma}_{\arc,\arctwo}$ coincide.
\end{lemma}

\begin{proof} Similar to the proof of Lemma \ref{flipdoesnotaffect1}.
\end{proof}

\begin{coro} Let $\arctwo_1,\arctwo_2\in\tau$. If $a:\arctwo_1\rightarrow\arctwo_2$ is an arrow of $\unredqtau$ contained in an ideal triangle of $\tau$ that does not contain $\arcone_1$, then $a$ is also an arrow of $\unredqsigma$, where $\sigma=f_{\arcone_1}(\tau)$, and the linear maps $\Mtauarc_a:\Mtauarc_{\arctwo_1}\rightarrow\Mtauarc_{\arctwo_2}$ and $\Msigmaarc_a:\Msigmaarc_{\arctwo_1}\rightarrow\Msigmaarc_{\arctwo_2}$ coincide.
\end{coro}

\subsection{Main result: Statement and proof}\label{subsec:proof}

Throughout this subsection we assume that the arc $\arc$ satisfies (\ref{transversally}), (\ref{minimalintersection}) and that
\begin{equation}\label{sigmatransversally} \text{$\arc$ intersects transversally each of the arcs of $\sigma=f_\arcone(\tau)$ (if at all), and}
\end{equation}
\begin{equation}\label{sigmaminimalintersection}\text{the number of intersection points of $\arc$ with each of the arcs of $\sigma$ is minimal.}
\end{equation}

As a final step in the preparation for the proof of our main result, we point out the fact that in such proof we can restrict our attention to surfaces without boundary.

\begin{lemma}\label{lemma:allarerestrictions1.1} For every QP-representation of the form $\calMtauarc$ there exists an ideal triangulation $\widetilde{\tau}$ of a surface $(\widetilde{\Sigma},\widetilde{M})$ with empty boundary with the following properties:
\begin{itemize}\item $\surfnoM\subseteq\widetilde{\Sigma}$ and $M\subseteq\widetilde{M}$;
\item $\widetilde{\tau}$ contains all the arcs of $\tau$;
\item $\mathcal{M}(\widetilde{\tau},\arc)$ is $\tau$-path-restrictable, and the restriction of $\mathcal{M}(\widetilde{\tau},\arc)$ to $\tau$ is $\calMtauarc$.
\end{itemize}
\end{lemma}

The proof of this Lemma is identical to that of Lemma 29 of \cite{Lqps}. Remark \ref{remarkisolated} applies here just as it applies in Lemma 29 of \cite{Lqps}.

The following theorem is the main result of this paper. Its proof is long, due to the separation into several cases, ending in page \pageref{page:endofproof}.

\begin{thm}\label{thm:flip<->mut} Let $\tau$ and $\sigma$ be ideal triangulations without self-folded triangles and $\arc$ be an arc satisfying (\ref{transversally}), (\ref{minimalintersection}), (\ref{sigmatransversally}),  and (\ref{sigmaminimalintersection}). If $\sigma=f_{\arcone}(\tau)$ for an arc $\arcone\in\tau$, then the decorated arc representations $\mu_{\arcone}(\calMtauarc)$ and $\calMsigmaarc$ are right-equivalent.
\end{thm}

\begin{proof} Our \label{page:beginningofproof} assumptions mean that none of $\tau$ and $\sigma$ has self-folded triangles, and that $\arc$ intersects transversally each of the arcs in $\tau$ and each of the arcs in $\sigma$; moreover, $\arc$ is a representative of its isotopy class that minimizes intersection numbers with each of the arcs in $\tau$ and each of the arcs in $\sigma$.

Now, $\arcone$ is the diagonal of a quadrilateral whose sides are arcs in $\tau\cap\sigma$. The vertices of this quadrilateral are marked points of $\surf$, which, by Proposition \ref{thm:resmut=mutres1.1} and Lemma \ref{lemma:allarerestrictions1.1}, we can assume to be punctures. After exchanging $\tau$ and $\sigma$ if necessary, we can suppose that each of these punctures is incident to at least three arcs of $\tau$, and so, the configuration of $\tau$ near the arc $\arcone=\arcone_1$ to be flipped looks like the one shown in Figure \ref{paper2case1},
where each of the punctures (labeled with the scalars) $w$ and $z$ is incident to at least three arcs of $\tau$ and each of the punctures (labeled with the scalars) $x$ and $y$ is incident to at least four arcs of $\tau$.
The quiver $Q(\tau)$ is shown in that Figure as well. As for the potential, we have
$$
S(\tau)=\alpha\beta\gamma+\delta\varepsilon\eta+w\alpha a+x\beta\eta b+y\varepsilon\gamma c+z\delta d+\sptau.
$$

\begin{remark}\label{rem:markedpointsmaycoincide} As observed in Remark 7 of \cite{Lqps}, some of the points (labeled with the scalars) $w,x,y$ and $z$ may actually represent the same marked point of $\surf$, and the potentials $\stau$ and $\ssigma$ immediately reflect this fact. For instance, if (the marked points labeled) $x$ and $y$ coincide, with no other coincidences among (the points labeled) $w,x,y$ and $z$, then the paths $\beta\eta$ and $\varepsilon\gamma$, instead of appearing as factors of two different summands of $\stau$ as in the previous paragraph, will appear as factors of a single summand of $\surf$. The cases we shall consider below will have the implicit supposition that the four marked points appearing in Figure \ref{paper2case1} are indeed different; we leave to the reader the task of making the suitable modifications of the right-equivalences we will give to adjust them to the cases when some of the mentioned marked points coincide.
\end{remark}

If we apply the premutation $\widetilde{\mu}_{\arcone_1}$ to $(Q(\tau),S(\tau))$ we obtain the QP $(\widetilde{\mu}_{\arcone}(Q(\tau)), \tildestau)$, where $\widetilde{\mu}_{\arcone}(Q(\tau))$ is the quiver obtained from $Q(\tau)$ by performing an ordinary quiver premutation (see, e.g., \cite{Lqps}), and $\tildestau=\alpha[\beta\gamma]+\delta[\varepsilon\eta]+w\alpha a+\gamma^*\beta^*[\beta\gamma]+z\delta d+\eta^*\varepsilon^*[\varepsilon\eta]+x[\beta\eta]b+y[\varepsilon\gamma]c+ \eta^*\beta^*[\beta\eta]+\gamma^*\varepsilon^*[\varepsilon\gamma]+\sptau\in
R\langle\langle\widetilde{\mu}_{\arcone}(Q(\tau))\rangle\rangle$.

The $R$-algebra automorphism $\lambda$ of $R\langle\langle\widetilde{\mu}_{\arcone}(Q(\tau))\rangle\rangle$ whose action on the arrows is given by
$$
\alpha\mapsto\alpha-\gamma^*\beta^*,\ [\beta\gamma]\mapsto[\beta\gamma]-wa,\ \delta\mapsto\delta-\eta^*\varepsilon^*,\
[\varepsilon\eta]\mapsto[\varepsilon\eta]-zd,
$$
and the identity in the rest of the arrows, sends $\tildestau$ to
$$
\lambda(\tildestau)=\alpha[\beta\gamma]+\delta[\varepsilon\eta]-w\gamma^*\beta^*a+x[\beta\eta] b+y[\varepsilon\gamma] c-z\eta^*\varepsilon^*
d+\sptau+\eta^*\beta^*[\beta\eta]+\gamma^*\varepsilon^*[\varepsilon\gamma].
$$
That is, $\lambda$ is a right-equivalence between $(\widetilde{\mu}_{\arcone}(Q(\tau)),\tildestau)$ and the direct sum of
$(\widetilde{\mu}_{\arcone}(Q(\tau))_{\ored}',\lambda(\tildestau)-\alpha[\beta\gamma]-\delta[\varepsilon\eta])$ and
$((\widetilde{\mu}_{\arcone}(Q(\tau))_{\triv}',\alpha[\beta\gamma]+\delta[\varepsilon\eta])$, where $(\widetilde{\mu}_{\arcone}(Q(\tau))_{\ored}'$
is obtained from $\widetilde{\mu}_{\arcone}(Q(\tau))$ by deleting the arrows $\alpha$, $[\beta\gamma]$, $\delta$ and $[\varepsilon\eta]$ and
$(\widetilde{\mu}_{\arcone}(Q(\tau))_{\triv}'$ is the quiver that has $\tau$ as its vertex set and whose only arrows are $\alpha$, $[\beta\gamma]$, $\delta$ and $[\varepsilon\eta]$. Notice that $(\widetilde{\mu}_{\arcone}(Q(\tau))_{\ored}',\lambda(\tildestau)-\alpha[\beta\gamma]-\delta[\varepsilon\eta])$ is not necessarily reduced since $[\beta\eta]b$ and $[\varepsilon\gamma]c$ may be 2-cycles. On the other hand, only $[\beta\eta]b$ and $[\varepsilon\gamma]c$ may be 2-cycles in the quiver $((\widetilde{\mu}_{\arcone}(Q(\tau))_{\ored}'$, and there is an inclusion of quivers $((\widetilde{\mu}_{\arcone}(Q(\tau))_{\ored}'\hookrightarrow\widehat{Q}(\sigma)$.

Now, as pointed out in the paragraph that follows Example 28 of \cite{Lqps}, the process of reducing a QP can be done in steps, taking care of 2-cycles one by one. From this and Remark \ref{rem:restrictingaction} we see that the QP-representation $(\widetilde{\mu}_{\arcone}(Q(\tau)),\tildestau,\overline{\Mtauarc},\overline{V(\tau,\arc)})$ gives rise, by reduction, to a QP-representation $\mathcal{M}_1=(((\widetilde{\mu}_{\arcone}(Q(\tau))_{\ored}',\lambda(\tildestau)-\alpha[\beta\gamma]-\delta[\varepsilon\eta],
\overline{\Mtauarc},\overline{V(\tau,\arc)})$, where the action of $R\langle\langle ((\widetilde{\mu}_{\arcone}(Q(\tau))_{\ored}'\rangle\rangle$ on $\overline{\Mtauarc}$ is induced by the inclusion of quivers $((\widetilde{\mu}_{\arcone}(Q(\tau))_{\ored}'\hookrightarrow
\widetilde{\mu}_{\arcone}(Q(\tau))$.  Similarly, the QP-representation
$(\widehat{Q}(\sigma),\widehat{S}(\sigma),\Msigmaarc,V(\sigma,\arc))$ gives rise, by reduction, to a QP-representation
$\mathcal{M}_2=(\check{Q}(\sigma),\check{S}(\sigma),\Msigmaarc,V(\sigma,\arc))$, where $\check{Q}(\sigma)=((\widetilde{\mu}_{\arcone}(Q(\tau))_{\ored}'$, $\check{S}(\sigma)=w\gamma^*\beta^*a+x[\beta\eta] b+y[\varepsilon\gamma] c+z\eta^*\varepsilon^*
d+\sptau+\eta^*\beta^*[\beta\eta]+\gamma^*\varepsilon^*[\varepsilon\gamma]$ and the action of $R\langle\langle\check{Q}(\sigma)\rangle\rangle$ on $\Msigmaarc$ is induced by the inclusion of quivers $((\widetilde{\mu}_{\arcone}(Q(\tau))_{\ored}'\hookrightarrow\widehat{Q}(\sigma)$. Since the reduced part of $\mathcal{M}_1$ is $\mu_{\arcone_1}(\calMtauarc)$ and the reduced part of $\mathcal{M}_2$ is $\calMsigmaarc$, we deduce that, to prove the theorem, it is enough to show that $\mathcal{M}_1$ is right-equivalent to $\mathcal{M}_2$. Notice that this discussion is unnecessary if $[\beta\eta]b$ and $[\varepsilon\gamma]c$ are not 2-cycles.

As said above, we assume, without loss of generality, that the boundary of $\surfnoM$ is empty. We may further assume, by Proposition \ref{prop:localdirectsum} and the classification of the possible summands of a direct sum decomposition of $\Mtauarc(\partial)$ given at the end of Subsection \ref{subsec:localdecompositions}, that the configuration $\arc$ presents around $\arcone$ is given by either of the Figures \ref{Fig:component1}, \ref{Fig:component2} and \ref{Fig:component3}. For time and space reasons, we are not going to include the  \emph{flip $\leftrightarrow$ mutation} analysis of each one of these configurations; we will analyze some of them and leave the rest as an exercise for the reader. Having said all this, let us proceed to check some cases.

\begin{case}\label{casenodetours} We are going to deal with the configurations 1, 2, 3, 4 and 5 of Figure \ref{Fig:component2} at once. Assume that, around the arc $\arcone_1$ to be flipped, $\tau$ and $\arc$ look as shown in Figure \ref{curvecase1nodetours} (to make the exposition less tedious, throughout the analysis of this case we will not make emphasis in the \emph{local decomposition} of $M(\partial)$ as the direct sum of five subrepresentations).
\begin{figure}[!h]
                \caption{}\label{curvecase1nodetours}
                \centering
                \includegraphics[scale=.3]{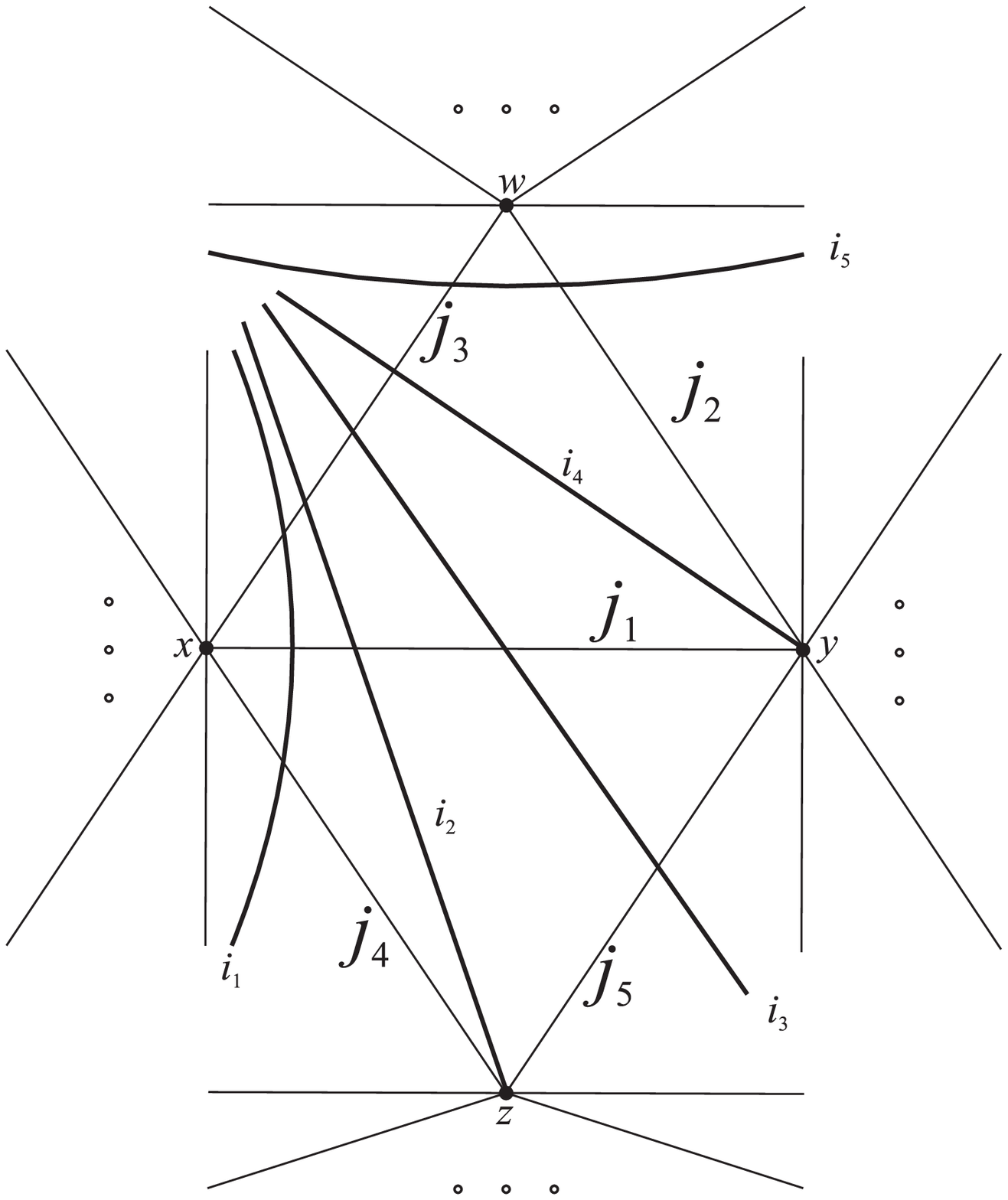}
        \end{figure}
The relevant vector spaces assigned in $\Mtauarc$ to the vertices of $Q(\tau)$ are
$$
M_{\arcone_{1}}=\Mtauarc_{\arcone_{1}}=\field^3, \ M_{\arcone_{2}}=\Mtauarc_{\arcone_{2}}=\field,
$$
$$
M_{\arcone_{3}}=\Mtauarc_{\arcone_{3}}=\field^{5}, \ M_{\arcone_{4}}=\Mtauarc_{\arcone_{4}}=\field,
$$
$$
\text{and} \ M_{\arcone_{5}}=\Mtauarc_{\arcone_{5}}=\field.
$$

Since none of $\alpha$, $\beta$, $\gamma$, $\delta$, $\varepsilon$ and $\eta$ is parallel to any detour of $(\tau,\arc)$, the detour matrices
$D^{\triangle_1}_{\arc,\arcone_1}$, $D^{\triangle_1}_{\arc,\arcone_2}$, $D^{\triangle_1}_{\arc,\arcone_3}$, $D^{\triangle_2}_{\arc,\arcone_1}$,
$D^{\triangle_2}_{\arc,\arcone_4}$ and $D^{\triangle_2}_{\arc,\arcone_5}$ are identities (of the corresponding sizes). Hence the arrows $\alpha$, $\beta$, $\gamma$, $\delta$, $\varepsilon$ and $\eta$ act on $\Mtauarc$ according to the following linear maps:
$$
\Mtauarc_\alpha={\tiny\left[\begin{array}{ccccc}
0 & 0 & 0 & 0 & 1\end{array}\right]}:\field^{5}\rightarrow\field, \ \
\Mtauarc_\beta={\tiny\left[\begin{array}{ccc}
1 & 0 & 0\\
0 & 1 & 0\\
0 & 0 & 1\\
0 & 0 & 0\\
0 & 0 & 0\end{array}\right]}:\field^{3}\rightarrow \field^{5},
$$
$$
\Mtauarc_\gamma=\mathbf{0}:\field\rightarrow \field^{3}, \ \
\Mtauarc_\delta=\mathbf{0}:\field\rightarrow \field,
$$
$$
\Mtauarc_\varepsilon={\tiny\left[\begin{array}{ccc}
0 & 0 & 1\end{array}\right]}:\field^{3}\rightarrow \field, \ \
\text{and} \ \
\Mtauarc_\eta={\tiny\left[\begin{array}{c}
1\\
0\\
0\end{array}\right]}:\field\rightarrow \field^{3}.
$$

Let us investigate the effect of the $\arcone_1^{\operatorname{th}}$ QP-mutation on $\calMtauarc$. An easy check using the information about $\calMtauarc$ we have collected thus far yields
$$
M_{\oin}=M_{\arcone_2}\oplus M_{\arcone_4}=\field\oplus\field,
$$
$$
M_{\oout}=M_{\arcone_3}\oplus M_{\arcone_5}=\field^{5}\oplus\field.
$$
$$
\al={\tiny\left[\begin{array}{cccc}
0 & 1 \\
0 & 0 \\
0 & 0 \end{array}\right]}:M_{\oin}=\field\oplus\field\rightarrow \field^{3}=M_{\arcone_1},
$$
$$
\be={\tiny\left[\begin{array}{ccc}
1 & 0 & 0\\
0 & 1 & 0\\
0 & 0 & 1\\
0 & 0 & 0\\
0 & 0 & 0\\
0 & 0 & 1\end{array}\right]}:M_{\arcone_1}=\field^{3}\rightarrow\field^{5}\oplus\field=M_{\oout},
$$
$$
\ga={\tiny\left[\begin{array}{ccccc|c}
0 & 0 & 0 & 0 & 1 & 0\\
\hline
0 & 0 & 0 & 0 & 0 & 0\end{array}\right]}:
M_{\oout}=\field^{5}\oplus\field\rightarrow \field\oplus\field=M_{\oin}.
$$

We see that $\be$ is injective and that $\ker\ga=\{(v_1,v_2,v_3,v_4,0,v_5)\suchthat v_1,v_2,v_3,v_4,v_5\in\field\}$, which is obviously isomorphic to $\field^5$ under the map $\ell:(v_1,v_2,v_3,v_4,0,v_5)\mapsto(v_1,v_2,v_3,v_4,v_5)$. The image of $\image\be=\{(u_1,u_2,u_3,0,0,u_3)\suchthat u_1,u_2,u_3\in\field\}$ under $\ell$ is $\ell(\image\be)=\{(u_1,u_2,u_3,0,u_3)\suchthat u_1,u_2,u_3\in\field\}$. Hence $\frac{\ker\ga}{\image\be}\cong\field^2$ and we can describe the canonical projection $\ker\ga\twoheadrightarrow\frac{\ker\ga}{\image\be}$ by means of the matrix
$$
\pipi=\left[\begin{array}{ccccc}
0 & 0 & -1 & 0 & 1\\
0 & 0 & 0  & 1 & 0\end{array}\right]:
\field^5\rightarrow\field^2.
$$

We also have $\ker\al=\image\ga=\{(u,0)\suchthat u\in\field\}\cong\field$. Therefore,
\begin{displaymath}
\overline{M}_{\arcone_1}=\field^2\oplus\field\oplus 0\oplus 0 \ \
\text{and} \ \ \overline{V}_{\arcone_1}=0.
\end{displaymath}

Let us compute the action of the arrows of $\widetilde{\mu}_{\arcone_1}(Q(\tau))$ on $\overline{\Mtauarc}$.
A straightforward check shows that $[\beta\gamma]$ and $[\varepsilon\eta]$ act as zero on $\overline{\Mtauarc}$, whereas
$$
[\beta\eta]={\tiny\left[\begin{array}{c}
1\\
0\\
0\\
0\\
0\end{array}\right]}:\overline{M}_{\arcone_4}=\field\rightarrow\field^5=\overline{M}_{\arcone_3} \ \
\text{and} \ \
[\varepsilon\gamma]=\mathbf{0}:\overline{M}_{\arcone_2}=\field\rightarrow\field=\overline{M}_{\arcone_5}.
$$

Since the arrows of $\widetilde{\mu}_{\arcone_1}(\qtau)$ not incident to $\arcone_1$ act on $\overline{\Mtauarc}$ in the exact same way they act on $\Mtauarc$, we just have to find out how the arrows $\beta^*$, $\gamma^*$, $\varepsilon^*$ and $\eta^*$ of $\widetilde{\mu}_{\arcone_1}(\qtau)$ act on $\overline{\Mtauarc}$. To this end, we choose the zero section $\si=\zero:\frac{\ker\al}{\image\ga}=0\rightarrow\ker\al$ and the retraction $\rh:M_{\oout}=\field^5\oplus\field\rightarrow\field^5\cong\ker\ga$ given by $(v_1,v_2,v_3,v_4,v_5,v_6)\mapsto(v_1,v_2,v_3,v_4,v_6)$ (here we are thinking of $\ell:\ker\ga\overset{\cong}{\longrightarrow}\field^{4}$ as an identification). A straightforward check yields
$$
-\pipi\rh=\left[\begin{array}{cccccc}0 & 0 & 1 & 0 & 0 & -1\\
0 & 0 & 0 & -1 & 0 & 0\end{array}\right].
$$
The action of $\beta^*$ and $\varepsilon^*$ is therefore encoded by the matrix
$$
[\beta^* \ \varepsilon^*]={\tiny\left[\begin{array}{ccccc|c}0 & 0 & 1 & 0 & 0 & -1\\
                                                            0 & 0 & 0 & -1 & 0 & 0\\
                                                            0 & 0 & 0 & 0 & -1 & 0\\
                                                            - & - & - & - & - & -\\
                                                            - & - & - & - & - & -\end{array}\right]}:\field^{5}\oplus\field^{1}\rightarrow \field^{2}\oplus\field\oplus 0\oplus 0,
$$
whereas the arrows $\gamma^*$ and $\eta^*$ act according to the matrix
$$
\left[\begin{array}{c}\gamma^*\\
                      \eta^*\end{array}\right]=
{\tiny\left[\begin{array}{ccccc}0 & 0 & 1 & - & -\\
                                \hline
                                0 & 0 & 0 & - & - \end{array}\right]}:\field^2\oplus\field\oplus 0\oplus 0\rightarrow \field\oplus\field.
$$

This completes the computation of the action of the arrows of $\widetilde{\mu}_{\arcone_1}(\qtau)$ on $\overline{\Mtauarc}$.
We have thus computed the premutation $\widetilde{\mu}_{\arcone_1}(\calMtauarc)=(\widetilde{\mu}_{\arcone_1}(\qtau),\tildestau,$ $\overline{\Mtauarc},0)$.

On the other hand, if we flip the arc $\arcone_1$ of $\tau$ we obtain the ideal triangulation $\sigma=f_{\arcone_1}(\tau)$ sketched at the left of Figure \ref{curvecase1nodetoursflipped}  (in a clear abuse of notation, we are using the same symbol $\arcone_1$ in both $\tau$ and $\sigma$).
\begin{figure}[!h]
                \caption{}\label{curvecase1nodetoursflipped}
                \centering
                \includegraphics[scale=.3]{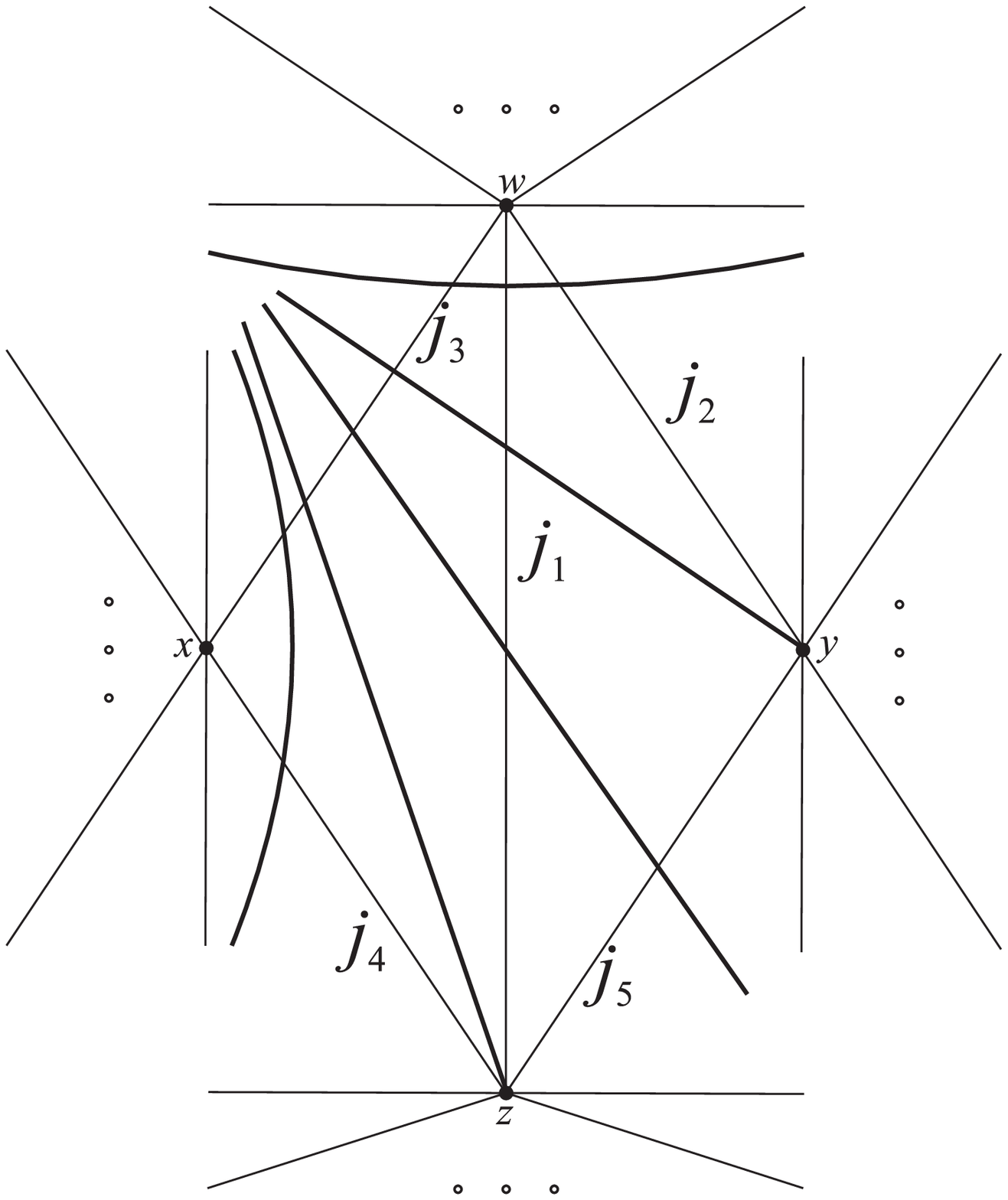}
        \end{figure}\\
The relevant vector spaces attached to the vertices of $Q(\sigma)$ are
$$
N_{\arcone_{1}}=\Msigmaarc_{\arcone_{1}}=\field^{3}, \ N_{\arcone_{2}}=\Msigmaarc_{\arcone_2}=\field,
$$
$$
N_{\arcone_{3}}=\Msigmaarc_{\arcone_{3}}=\field^{5}, \ N_{\arcone_{4}}=\Msigmaarc_{\arcone_4}=\field,
$$
$$
\text{and} \ N_{\arcone_5}=\Msigmaarc_{\arcone_5}=\field.
$$

Since none of $\beta^*$, $\gamma^*$, $\varepsilon^*$, $\eta^*$, $[\beta\eta]$ and $[\varepsilon\gamma]$ is parallel to any detour of $(\sigma,\arc)$, the detour matrices
$D^{\triangle'}_{\arc,\arcone_1}$, $D^{\triangle''}_{\arc,\arcone_2}$, $D^{\triangle'}_{\arc,\arcone_3}$, $D^{\triangle''}_{\arc,\arcone_1}$,
$D^{\triangle'}_{\arc,\arcone_4}$ and $D^{\triangle''}_{\arc,\arcone_5}$ are identities (of the corresponding sizes). Hence the arrows $\beta^*$, $\gamma^*$, $\varepsilon^*$, $\eta^*$, $[\beta\eta]$ and $[\varepsilon\gamma]$ act on $\Mtauarc$ according to the following linear maps:
$$
\Msigmaarc_{\beta^*}={\tiny\left[\begin{array}{ccccc}
0 & 0 & 1 & 0 & 0\\
0 & 0 & 0 & 1 & 0\\
0 & 0 & 0 & 0 & 1\end{array}\right]}:\field^{5}\rightarrow\field^{3}, \ \
\Msigmaarc_{\gamma^*}={\tiny\left[\begin{array}{ccc}
0 & 0 & 1\end{array}\right]}:\field^3\rightarrow \field,
$$
$$
\Msigmaarc_{\varepsilon^*}={\tiny\left[\begin{array}{c}
1\\
0\\
0\end{array}\right]}:\field\rightarrow \field^{3}, \ \
\Msigmaarc_{\eta^*}=\mathbf{0}:\field^{3}\rightarrow \field,
$$
$$
\Msigmaarc_{[\beta\eta]}={\tiny\left[\begin{array}{c}
1\\
0\\
0\\
0\\
0\end{array}\right]}:\field\rightarrow \field^{5}, \ \
\text{and} \ \
\Msigmaarc_{[\varepsilon\gamma]}=\mathbf{0}:\field\rightarrow \field.
$$

We have thus computed the spaces and linear maps of $\calMsigmaarc$ relevant to the flip of the arc $\arcone_1$. Now we have to compare it to $\mu_{\arcone_1}(\calMtauarc)$. The triple $\Phi=(\varphi,\psi,\eta)$ is a right-equivalence between these QP-representations, where
\begin{itemize}\item $\varphi:(\overline{Q(\tau)},\lambda(\tildestau)-\alpha[\beta\gamma]-\delta[\varepsilon\eta])\rightarrow \qssigma$ is the right-equivalence whose action on the arrows is given by
$$
\beta^*\mapsto-\beta^*, \ \eta^*\mapsto-\eta^*,
$$
and the identity in the rest of the arrows;
\item $\psi:\overline{\Mtauarc}\rightarrow\Msigmaarc$ is the vector space isomorphism given by the identity $\id:\overline{\Mtauarc}_{\arctwo}\rightarrow\Msigmaarc_{\arctwo}$ for $\arctwo\neq\arcone_1$, and the matrix
    $$
    \psi_{\arcone_{1}}={\tiny\left[\begin{array}{ccc}
    -1        & 0         & 0\\
    0         & 1         & 0\\
    0         & 0         & 1\end{array}\right]}:\overline{\Mtauarc}_{\arcone_1}\rightarrow\Msigmaarc_{\arcone_1}.
    $$
\item $\eta$ is the zero map (of the zero space).
\end{itemize}

This finishes the proof of Theorem \ref{thm:flip<->mut} for the configurations 1, 2, 3, 4 and 5 of Figure \ref{Fig:component2}. Furthermore, if $[\beta\eta]b$ and $[\varepsilon\gamma]c$ are not 2-cycles, the case just analyzed deals also with configurations 6, 7 and 8 of that Figure.
\end{case}

\begin{case}\label{casesomedetours} Now we are going to deal with configuration 1 from Figure \ref{Fig:component3}. Assume that, around the arc $\arcone_1$ to be flipped, $\tau$ and $\arc$ look as shown in Figure \ref{curvecase1var1}.
\begin{figure}[!h]
                \caption{}\label{curvecase1var1}
                \centering
                \includegraphics[scale=.5]{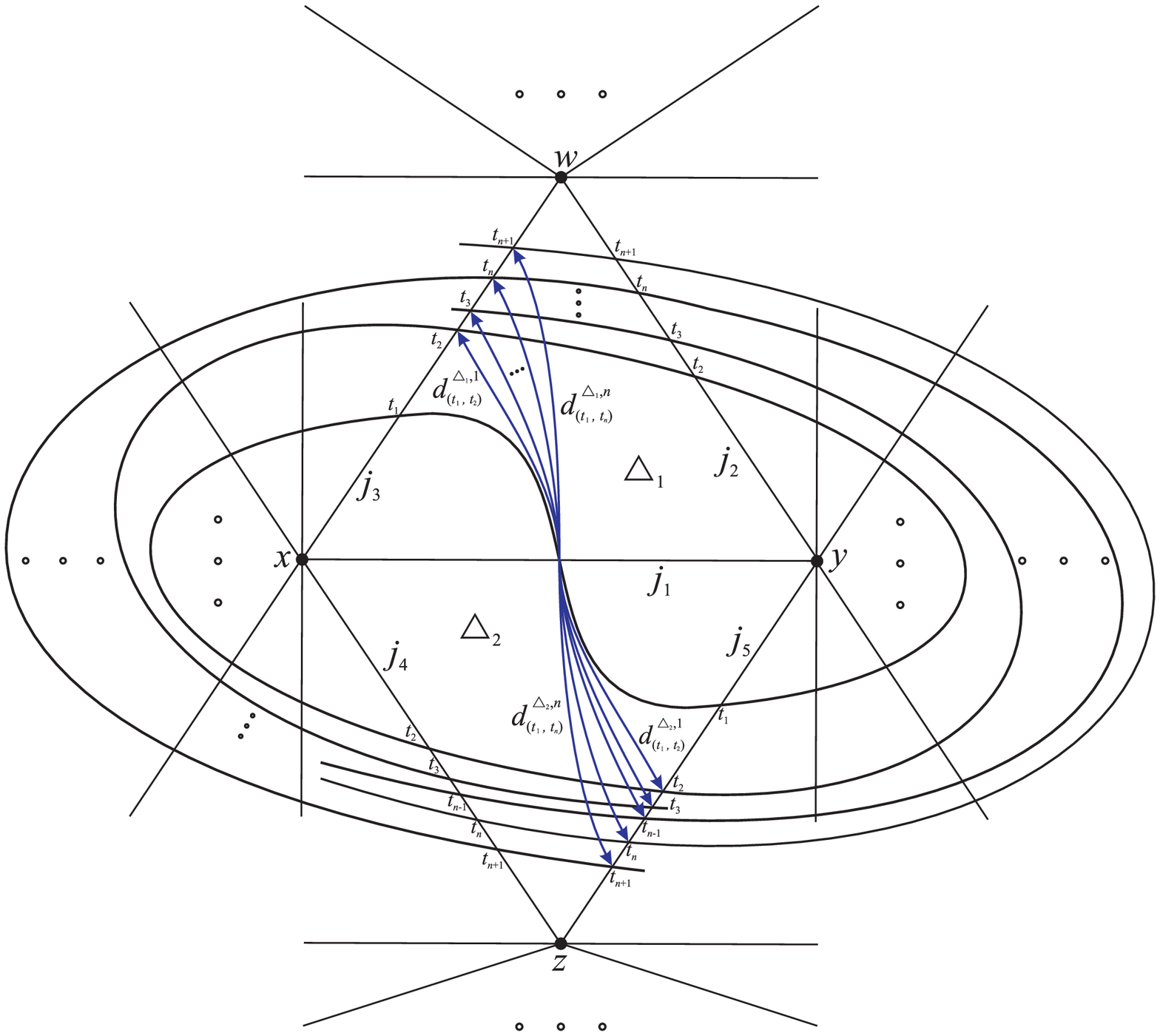}
        \end{figure}

The relevant vector spaces assigned in $\Mtauarc$ to the vertices of $Q(\tau)$ are
$$
M_{\arcone_{1}}=\Mtauarc_{\arcone_{1}}=\field, \ M_{\arcone_{2}}=\Mtauarc_{\arcone_{2}}=\field^{n},
$$
$$
M_{\arcone_{3}}=\Mtauarc_{\arcone_{3}}=\field^{n+1}, \ M_{\arcone_{4}}=\Mtauarc_{\arcone_{4}}=\field^{n},
$$
$$
\text{and} \ M_{\arcone_{5}}=\Mtauarc_{\arcone_{5}}=\field^{n+1}.
$$

We also have (with some notational abuse regarding the intersection points of $\arc$ with the arcs of $\tau$)
$$
\badintersection^{\triangle_1,l}_{\arc,\arcone_1}=\badintersection^{\triangle_2,l}_{\arc,\arcone_1}=
\badintersection^{\triangle_1,l}_{\arc,\arcone_2}=\badintersection^{\triangle_2,l}_{\arc,\arcone_4}=\varnothing \ \text{for} \ l\geq 1;
$$
$$
\badintersection^{\triangle_1,l}_{\arc,\arcone_3}=\{(t_1,t_{l+1},b(d^{\triangle_1,l}_{(t_1,t_{l+1})},y))\} \ \text{and} \ \badintersection^{\triangle_2,l}_{\arc,\arcone_5}=\{(t_1,t_{l+1},b(d^{\triangle_1,l}_{(t_1,t_{l+1})},x)\} \ \text{for} \ 1\leq l\leq n.
$$

The relevant detour matrices are therefore defined as follows. The matrices $D^{\triangle_1}_{\arc,\arcone_1}$, $D^{\triangle_2}_{\arc,\arcone_1}$, $D^{\triangle_1}_{\arc,\arcone_2}$ and $D^{\triangle_2}_{\arc,\arcone_4}$ are identities (of the corresponding sizes), whereas
$$
D^{\triangle_1}_{\arc,\arcone_3}={\tiny\left[
\begin{array}{cccc}1 &  & \mathbf{0}_{1\times n} &  \\
-y                    &  & &  \\
xy                    &  &  &  \\
\vdots                &   &      &   \\
(-1)^{l-1}x^{\lfloor\frac{l-1}{2}\rfloor}y^{\lfloor\frac{l}{2}\rfloor} &  & \text{{\Huge{\textbf{1}}$_{n\times n}$}} &  \\
\vdots                &   &        &   \\
(-1)^{n}x^{\lfloor\frac{n}{2}\rfloor}y^{\lfloor\frac{n+1}{2}\rfloor}   &  &  &  \end{array}\right]}, \ D^{\triangle_2}_{\arc,\arcone_5}={\tiny\left[
\begin{array}{ccccc}1 & & \mathbf{0}_{1\times n} &  \\
-x                    &  &  &  \\
xy                    &  &  & & \\
\vdots                &   &       & \\
(-1)^{l-1}x^{\lfloor\frac{l}{2}\rfloor}y^{\lfloor\frac{l-1}{2}\rfloor} &  & \text{{\Huge{\textbf{1}}$_{n\times n}$}} &  \\
\vdots                &   &        &   &  \\
(-1)^{n}x^{\lfloor\frac{n+1}{2}\rfloor}y^{\lfloor\frac{n}{2}\rfloor}   &  &  &  \end{array}\right]}.
$$

Hence the arrows $\alpha$, $\beta$, $\gamma$, $\delta$, $\varepsilon$ and $\eta$ act on $\Mtauarc$ according to the following linear maps:
$$
\Mtauarc_\alpha=(D_{\arc,\arcone_2}^{\triangle_1})(\mtauarc_\alpha)=
\left[\begin{array}{cc}\mathbf{0}_{n\times1} & \mathbf{1}_{n\times n}\end{array}\right]:\field^{n+1}\rightarrow\field^{n},
$$
$$
\Mtauarc_\beta=(D_{\arc,\arcone_3}^{\triangle_1})(\mtauarc_\beta)
={\tiny\left[\begin{array}{c}1\\
                       -y\\
                       xy\\
                       \vdots\\
                       (-1)^{l-1}x^{\lfloor\frac{l-1}{2}\rfloor}y^{\lfloor\frac{l}{2}\rfloor}\\
                       \vdots\\
                       (-1)^{n}x^{\lfloor\frac{n}{2}\rfloor}y^{\lfloor\frac{n+1}{2}\rfloor}\end{array}\right]}:\field\rightarrow \field^{n+1},
$$
$$
\Mtauarc_\gamma=(D_{\arc,\arcone_1}^{\triangle_1})(\mtauarc_\gamma)
=\zero:\field^{n}\rightarrow \field,
$$
$$
\Mtauarc_\delta=(D_{\arc,\arcone_4}^{\triangle_2})(\mtauarc_\delta)=
\left[\begin{array}{cc}\zero_{n\times 1} & \id_{n\times n}\end{array}\right]:
\field^{n+1}\rightarrow \field^{n},
$$
$$
\Mtauarc_\varepsilon=
(D_{\arc,\arcone_5}^{\triangle_2})(\mtauarc_\varepsilon)
={\tiny\left[\begin{array}{c}1\\
                       -x\\
                       xy\\
                       \vdots\\
                       (-1)^{l-1}x^{\lfloor\frac{l}{2}\rfloor}y^{\lfloor\frac{l-1}{2}\rfloor}\\
                       \vdots\\
                       (-1)^{n}x^{\lfloor\frac{n+1}{2}\rfloor}y^{\lfloor\frac{n}{2}\rfloor}\end{array}\right]}:\field\rightarrow \field^{n+1},
$$
$$
\text{and} \
\Mtauarc_\eta=(D_{\arc,\arcone_1}^{\triangle_2})(\mtauarc_\eta)=\zero:
\field^{n}\rightarrow \field.
$$

Let us investigate the effect of the $\arcone_1^{\operatorname{th}}$ QP-mutation on $\calMtauarc$. An easy check using the information about $\calMtauarc$ we have collected thus far yields
$$
M_{\oin}=M_{\arcone_2}\oplus M_{\arcone_4}=\field^{n}\oplus\field^{n},
$$
$$
M_{\oout}=M_{\arcone_3}\oplus M_{\arcone_5}=\field^{n+1}\oplus\field^{n+1}.
$$
$$
\al=\zero:M_{\oin}=\field^{n}\oplus\field^{n}\rightarrow \field=M_{\arcone_1},
$$
$$
\be={\tiny\left[\begin{array}{c}1\\
                          -y\\
                          xy\\
                          \vdots\\
                          (-1)^{n}x^{\lfloor\frac{n}{2}\rfloor}y^{\lfloor\frac{n+1}{2}\rfloor}\\
                           1\\
                           -x\\
                           xy\\
                           \vdots\\
                           (-1)^{n}x^{\lfloor\frac{n+1}{2}\rfloor}y^{\lfloor\frac{n}{2}\rfloor}\end{array}\right]}:
                           M_{\arcone_1}=\field\rightarrow\field^{n+1}\oplus\field^{n+1}=M_{\oout},
$$
$$
\ga={\tiny\left[\begin{array}{cc|cc}
\text{{\LARGE{\textbf{0}}$_{n\times 1}$}} & \text{{\LARGE{\textbf{1}}$_{n\times n}$ \ \ }} & \text{{\LARGE{ \ \ \ $y$\textbf{1}}$_{n\times n}$}} & \text{{\LARGE{\textbf{0}}$_{n\times 1}$}} \\
                                               & & & \\
                                              \hline
                                               & & & \\
\text{{\LARGE{ \ \ \ $x$\textbf{1}}$_{n\times n}$}} & \text{{\LARGE{\textbf{0}}$_{n\times 1}$}} & \text{{\LARGE{\textbf{0}}$_{n\times 1}$}} &
\text{{\LARGE{\textbf{1}}$_{n\times n}$ \ \ }} \end{array}\right]}:
M_{\oout}=\field^{n+1}\oplus\field^{n+1}\rightarrow \field^{n}\oplus\field^{n}=M_{\oin}.
$$

It is easily seen that $\be$ is injective and that $\ker\ga=\{(u,v)\in\field^{n+1}\oplus\field^{n+1}\suchthat u_{l+1}+yv_l=xu_{l}+v_{l+1}=0 \ \forall l\in\{1,\ldots,n\}\}=
\{(u_1,-yv_1,xyu_1,\ldots,v_1,-xu_1,xyv_1,\ldots)\suchthat u_1,v_1\in\field\}$, which is isomorphic to $\field^2$ under the assignment $\ell:(u_1,-yv_1,xyu_1,\ldots,v_1,-xu_1,xyv_1,\ldots)\mapsto(u_1,v_1-u_1)$. Together with a standard dimension counting, this yields surjectivity of $\ga$.

The image of $\image\be=\{u,-yu,xyu,\ldots,(-1)^{n}x^{\lfloor\frac{n}{2}\rfloor}y^{\lfloor\frac{n+1}{2}\rfloor},u,-xu,\ldots
(-1)^nx^{\lfloor\frac{n+1}{2}\rfloor}y^{\lfloor\frac{n}{2}\rfloor}u)\suchthat u\in\field\}$ under $\ell$ is $\ell(\image\be)=
\{(u,0)\suchthat u\in\field\}$. Hence $\frac{\ker\ga}{\image\be}\cong\field$ and we can describe the canonical projection
$\ker\ga\twoheadrightarrow\frac{\ker\ga}{\image\be}$ by means of the matrix
$$
\pipi=\left[\begin{array}{cc}0 & -1\end{array}\right]:\field^2\rightarrow\field.
$$

From the previous two paragraphs we deduce that
\begin{equation}
\overline{M}_{\arcone_1}=\field\oplus(\field^{n}\oplus\field^{n})\oplus 0\oplus 0 \ \ \text{and} \ \
\overline{V}_{\arcone_1}=0.
\end{equation}

And from the fact that $\gamma$ and $\eta$ act as zero on $\Mtauarc$, we conclude that the arrows $[\beta\gamma]$, $[\varepsilon\gamma]$, $[\beta\eta]$ and $[\varepsilon\eta]$ of $\widetilde{\mu}_{\arcone_1}(\qtau)$ act as zero on $\overline{\Mtauarc}$. Since the arrows of $\widetilde{\mu}_{\arcone_1}(\qtau)$ not incident to $\arcone_1$ act on $\overline{\Mtauarc}$ in the exact same way they act on $\Mtauarc$, we just have to find out how the arrows $\beta^*$, $\gamma^*$, $\varepsilon^*$ and $\eta^*$ of $\widetilde{\mu}_{\arcone_1}(\qtau)$ act on $\overline{\Mtauarc}$. To this end, we choose the zero section $\si=\zero:\frac{\ker\al}{\image\ga}=0\rightarrow\ker\al$ and the retraction $\rh:M_{\oout}\rightarrow\ker\ga$ given by the matrix
$$
\rh={\tiny\left[\begin{array}{c}1\\
                                -1\end{array}
\begin{array}{c}\text{{\LARGE{\textbf{0}}$_{2\times n}$}}\end{array}
\begin{array}{c} 0 \\ 1 \end{array}
\begin{array}{c}\text{{\LARGE{\textbf{0}}$_{2\times n}$}}\end{array}\right]}:\field^{n+1}\oplus\field^{n+1}\rightarrow\field^2
$$
(here we are thinking of $\ell:\ker\ga\overset{\cong}{\longrightarrow}\field^{2}$ as an identification). A straightforward check yields
$$
-\pipi\rh=\left[\begin{array}{cccc}-1 & \mathbf{0}_{1\times n} & 1 & \mathbf{0}_{1\times n}\end{array}\right].
$$
The action of $\beta^*$ and $\varepsilon^*$ is therefore encoded by the matrix
$$
[\beta^* \ \varepsilon^*]=\left[\begin{array}{c|} \\
\text{{\Huge$-${\textbf{1}}$_{(n+1)\times (n+1)}$}} \\
                                               \\
\text{{\LARGE{$-x$\textbf{1}}$_{n\times n}$}} \ \ \text{{\LARGE{\textbf{0}}$_{n\times 1}$}} \\
 - \\
 - \end{array}
\begin{array}{c}
1\\
-y\\
\text{{\LARGE{\textbf{0}}$_{(n-1)\times 1}$}}\\
\text{{\LARGE{\textbf{0}}$_{n\times 1}$}}\\
-\\
-
\end{array}
\begin{array}{cc}
 & \\
\text{{\LARGE\textbf{0}$_{2\times (n-1)}$}} & \text{{\LARGE\textbf{0}$_{2\times1}$}} \\
\text{{\LARGE{$-y$\textbf{1}}$_{(n-1)\times (n-1)}$}} & \text{{\LARGE{\textbf{0}}$_{(n-1)\times 1}$}}\\
\text{{\LARGE{ \ \ \ \ \ \ \ \ \ \ \ \ \  $-$\textbf{1}}$_{n\times n}$}} & \\
- & - \\
- & -
\end{array}\right]:
$$
$$
:\field^{n+1}\oplus\field^{n+1}\rightarrow \field\oplus(\field^{n}\oplus\field^{n})\oplus 0\oplus 0,
$$
whereas the arrows $\gamma^*$ and $\eta^*$ act according to the matrix
$$
\left[\begin{array}{c}\gamma^*\\
                      \eta^*\end{array}\right]=
\left[\begin{array}{ccccc}
\mathbf{0}_{n\times 1} & \mathbf{1}_{n\times n} & \mathbf{0}_{n\times n} & - & -\\
\hline
\mathbf{0}_{n\times 1} & \mathbf{0}_{n\times n} & \mathbf{1}_{n\times n} & - & -
\end{array}\right]
:\field\oplus(\field^{n+1}\oplus\field^{n+1})\oplus 0\oplus 0\rightarrow \field^{n+1}\oplus\field^{n+1}.
$$

This completes the computation of the action of the arrows of $\widetilde{\mu}_{\arcone_1}(Q(\tau))$ on $\overline{\Mtauarc}$.
We have thus computed the premutation $\widetilde{\mu}_{\arcone_1}(\calMtauarc)=(\widetilde{\mu}_{\arcone_1}(Q(\tau)),\tildestau,$ $\overline{\Mtauarc},0)$.

On the other hand, if we flip the arc $\arcone_1$ of $\tau$ we obtain the ideal triangulation $\sigma=f_{\arcone_1}(\tau)$ sketched in Figure \ref{curvecase1var1flipped}  (in a clear abuse of notation, we are using the same symbol $\arcone_1$ in both $\tau$ and $\sigma$).
\begin{figure}[!h]
                \caption{}\label{curvecase1var1flipped}
                \centering
                \includegraphics[scale=.5]{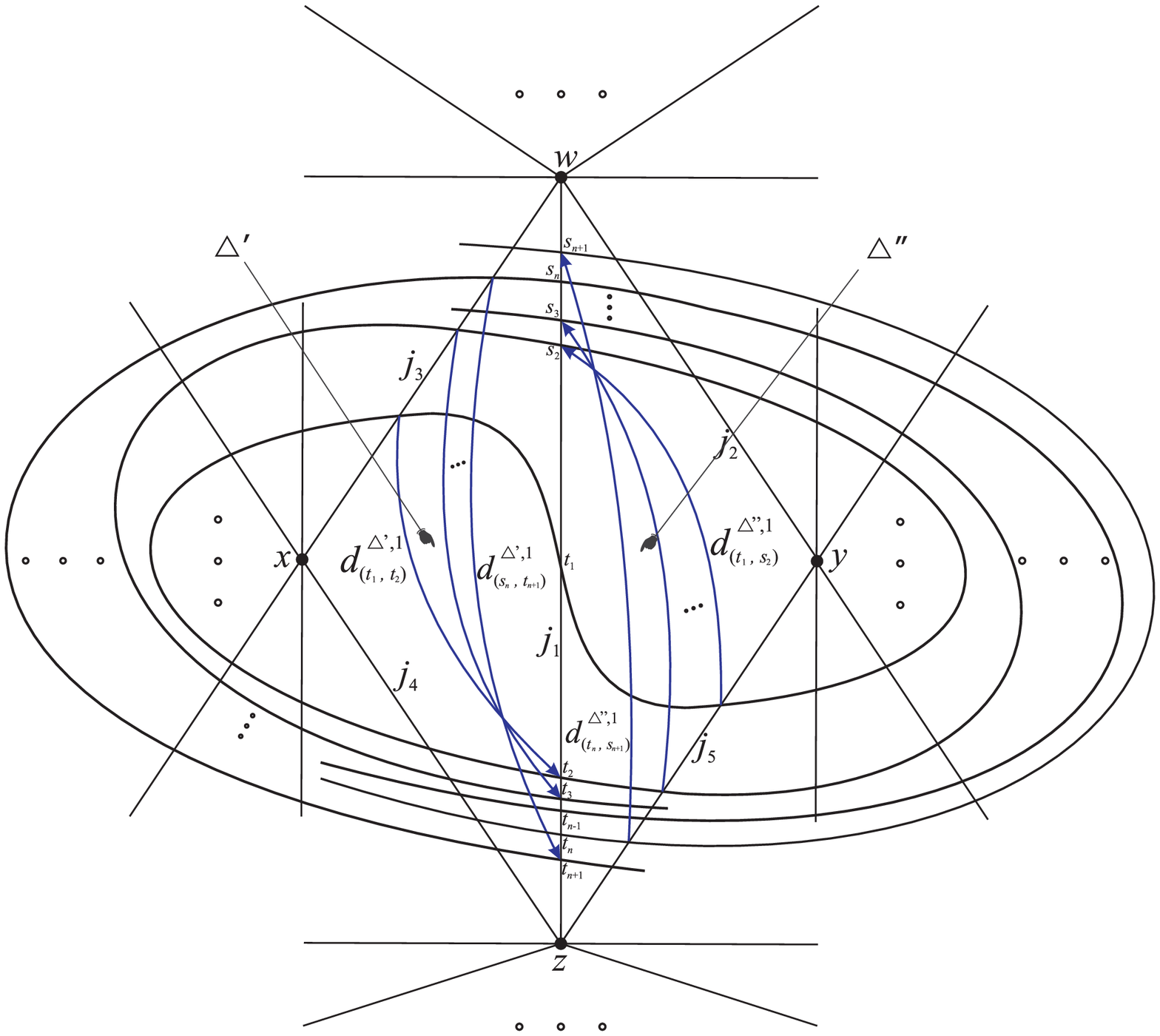}
        \end{figure}\\
The relevant vector spaces attached to the vertices of $Q(\sigma)$ are
$$
N_{\arcone_{1}}=\Msigmaarc_{\arcone_{1}}=\field^{2n+1}, \ N_{\arcone_{2}}=\Msigmaarc_{\arcone_2}=\field^{n},
$$
$$
N_{\arcone_{3}}=\Msigmaarc_{\arcone_{3}}=\field^{n+1}, \ N_{\arcone_{4}}=\Msigmaarc_{\arcone_4}=\field^{n},
$$
$$
\text{and} \ N_{\arcone_5}=\Msigmaarc_{\arcone_5}=\field^{n+1}.
$$
We also have (again with some notational abuse regarding intersection points)
$$
\badintersection^{\triangle',1}_{\arc,\arcone_1}=
\{(t_1,t_2,b(d^{\triangle',1}_{(t_1,t_2)}),x)\}\cup\{(s_{l},t_{l+1},b(d^{\triangle',1}_{(s_l,s_{l+1})}),x)\suchthat 2\leq l\leq n\},
$$
$$
\badintersection^{\triangle'',1}_{\arc,\arcone_1}=
\{(t_1,s_2,b(d^{\triangle'',1}_{(t_1,s_2)}),y)\}\cup\{(t_l,s_{l+1},b(d^{\triangle'',1}_{(t_l,s_{l+1})}),y)\suchthat 2\leq l\leq n\};
$$
$$
\badintersection^{\triangle',r}_{\arc,\arcone_1}=\badintersection^{\triangle'',r}_{\arc,\arcone_1}=\varnothing \ \text{for} \ r\geq 2;
$$
$$
\text{and} \ \badintersection^{\triangle',r}_{\arc,\arcone_2}=\badintersection^{\triangle',r}_{\arc,\arcone_3}=
\badintersection^{\triangle',r}_{\arc,\arcone_4}=\badintersection^{\triangle',r}_{\arc,\arcone_5}=
$$
$$
\badintersection^{\triangle'',r}_{\arc,\arcone_2}=\badintersection^{\triangle'',r}_{\arc,\arcone_3}=
\badintersection^{\triangle'',r}_{\arc,\arcone_4}=\badintersection^{\triangle'',r}_{\arc,\arcone_5}=\varnothing \ \text{for} \ r\geq 1.
$$

The relevant detour matrices are therefore defined as follows. The matrices $D^{\triangle''}_{\arc,\arcone_2}$, $D^{\triangle'}_{\arc,\arcone_3}$, $D^{\triangle'}_{\arc,\arcone_4}$ and $D^{\triangle''}_{\arc,\arcone_5}$ are identities (of the corresponding sizes), whereas
$$
D_{\arc,\arcone_1}^{\triangle'}=\left[\begin{array}{cc}
\text{{\LARGE\textbf{1}$_{(n+1)\times (n+1)}$}} & \mathbf{0}_{(n+1)\times n}\\
-x\mathbf{1}_{n\times n} \ \ \mathbf{0}_{n\times 1} & \mathbf{1}_{n\times n}
\end{array}\right] \ \ \text{and}
$$
$$
D_{\arc,\arcone_1}^{\triangle''}=
{\tiny\left[\begin{array}{c}
1 \\
-y \\
 \\
\text{{\Huge\textbf{0}$_{(2n-1)\times 1}$}}
\end{array}
\begin{array}{c}
 \\
\mathbf{0}_{1\times n}\\
\text{{\Huge\textbf{1}$_{n\times n}$}}\\
\text{{\Huge\textbf{0}$_{n\times n}$}}
\end{array}
\begin{array}{cc}
 \\
\text{{\large \ \ \ \ \ \ \ \ \ \ \ \ \ \ \ \ \ \ \ \ \textbf{0}$_{2\times n}$}} & \\
\text{{\LARGE$-y$\textbf{1}$_{(n-1)\times (n-1)}$}} & \mathbf{0}_{(n-1)\times 1}\\
\text{{\Huge \ \ \ \ \ \ \ \ \ \ \ \ \ \textbf{1}$_{n\times n}$}} &
\end{array}
\right]},
$$
where the order in which the basis vectors of $N_{\arcone_1}$ are taken is $(t_1,s_2,\ldots,s_{n+1},t_2,\ldots,t_{n+1})$.

Hence the arrows $\beta^*$, $[\beta\eta]$, $\eta^*$, $\gamma^*$, $\varepsilon^*$ and $[\varepsilon\gamma]$ act on $\Msigmaarc$ according to the following linear maps:
$$
\Msigmaarc_{\beta^*}=(D_{\arc,\arcone_1}^{\triangle'})(\msigmaarc_{\beta^*})=
\left[\begin{array}{cc}
\text{{\LARGE\textbf{1}$_{(n+1)\times (n+1)}$}} \\
-x\mathbf{1}_{n\times n} \ \ \mathbf{0}_{n\times 1}
\end{array}\right]:\field^{n+1}\rightarrow\field^{2n+1},
$$
$$
\Msigmaarc_{[\beta\eta]}=
(D_{\arc,\arcone_3}^{\triangle'})(\msigmaarc_{[\beta\eta]})=\zero:\field^{n}\rightarrow \field^{n+1},
$$
$$
\Msigmaarc_{\eta^*}=(D_{\arc,\arcone_4}^{\triangle'})(\msigmaarc_{\eta^*})=
\left[\begin{array}{cc}
\mathbf{0}_{n\times{(n+1)}} & \mathbf{1}_{n\times n}\end{array}\right]
:\field^{2n+1}\rightarrow \field^{n},
$$
$$
\Msigmaarc_{\gamma^*}=(D_{\arc,\arcone_2}^{\triangle''})(\msigmaarc_{\gamma^*})=
\left[\begin{array}{ccc}
\mathbf{0}_{n\times 1} & \mathbf{1}_{n\times n} & \mathbf{0}_{n\times n}\end{array}\right]
:\field^{2n+1}\rightarrow \field^{n},
$$
$$
\Msigmaarc_{\varepsilon^*}=(D_{\arc,\arcone_1}^{\triangle''})(\msigmaarc_{\varepsilon^*})=
{\tiny\left[\begin{array}{c}
1 \\
-y \\
 \\
\text{{\Huge\textbf{0}$_{(2n-1)\times 1}$}}
\end{array}
\begin{array}{c}
 \\
\text{{\large \ \ \ \ \ \ \ \ \ \ \ \ \ \ \ \ \ \textbf{0}$_{2\times n}$}} \\
\text{{\LARGE$-y$\textbf{1}$_{(n-1)\times (n-1)}$}} \ \ \  \text{{\large\textbf{0}$_{(n-1)\times 1}$}}\\
\text{{\Huge \ \ \ \ \ \ \ \ \ \ \textbf{1}$_{n\times n}$}}
\end{array}
\right]}
:\field^{n+1}\rightarrow \field^{2n+1},
$$
$$
\Msigmaarc_{[\varepsilon\gamma]}=
(D_{\arc,\arcone_5}^{\triangle''})(\msigmaarc_{[\varepsilon\gamma]})=\zero:\field^{n}\rightarrow \field^{n+1}.
$$

We have thus computed the spaces and linear maps of $\calMsigmaarc$ relevant to the flip of the arc $\arcone_1$. Now we have to compare it to $\mu_{\arcone_1}(\calMtauarc)$. The triple $\Phi=(\varphi,\psi,\eta)$ is a right-equivalence between these QP-representations, where
\begin{itemize}\item $\varphi:(\overline{Q(\tau)},\lambda(\tildestau)-\alpha[\beta\gamma]-\delta[\varepsilon\eta])\rightarrow \qssigma$ is the right-equivalence whose action on the arrows is given by
$$
\beta^*\mapsto-\beta^*, \ \eta^*\mapsto-\eta^*,
$$
and the identity in the rest of the arrows;
\item $\psi:\overline{\Mtauarc}\rightarrow\Msigmaarc$ is the vector space isomorphism given by the identity $\id:\overline{\Mtauarc}_{\arctwo}\rightarrow\Msigmaarc_{\arctwo}$ for $\arctwo\neq\arcone_1$, and the matrix
    $$
    \psi_{\arcone_{1}}=\left[\begin{array}{cc}
    \mathbf{1}_{(n+1)\times(n+1)} & \mathbf{0}_{(n+1)\times n}\\
    \mathbf{0}_{n\times(n+1)} & -\mathbf{1}_{n\times n}\end{array}\right]
:\overline{\Mtauarc}_{\arcone_1}\rightarrow\Msigmaarc_{\arcone_1};
    $$
\item $\eta$ is the zero map (of the zero space).
\end{itemize}

This finishes the proof of Theorem \ref{thm:flip<->mut} for configuration 1 of Figure \ref{Fig:component3}. Furthermore, if $[\beta\eta]b$ and $[\varepsilon\gamma]c$ are not 2-cycles, the case just analyzed deals also with configuration 3 of that Figure.
\end{case}

\begin{case}\label{casemanydetours} Now we are going to deal with configuration 2 from Figure \ref{Fig:component1}. Assume that, around the arc $\arcone_1$ to be flipped, $\tau$ and $\arc$ look as shown in Figure \ref{curvecase1var2}.
\begin{figure}[!h]
                \caption{}\label{curvecase1var2}
                \centering
                \includegraphics[scale=.5]{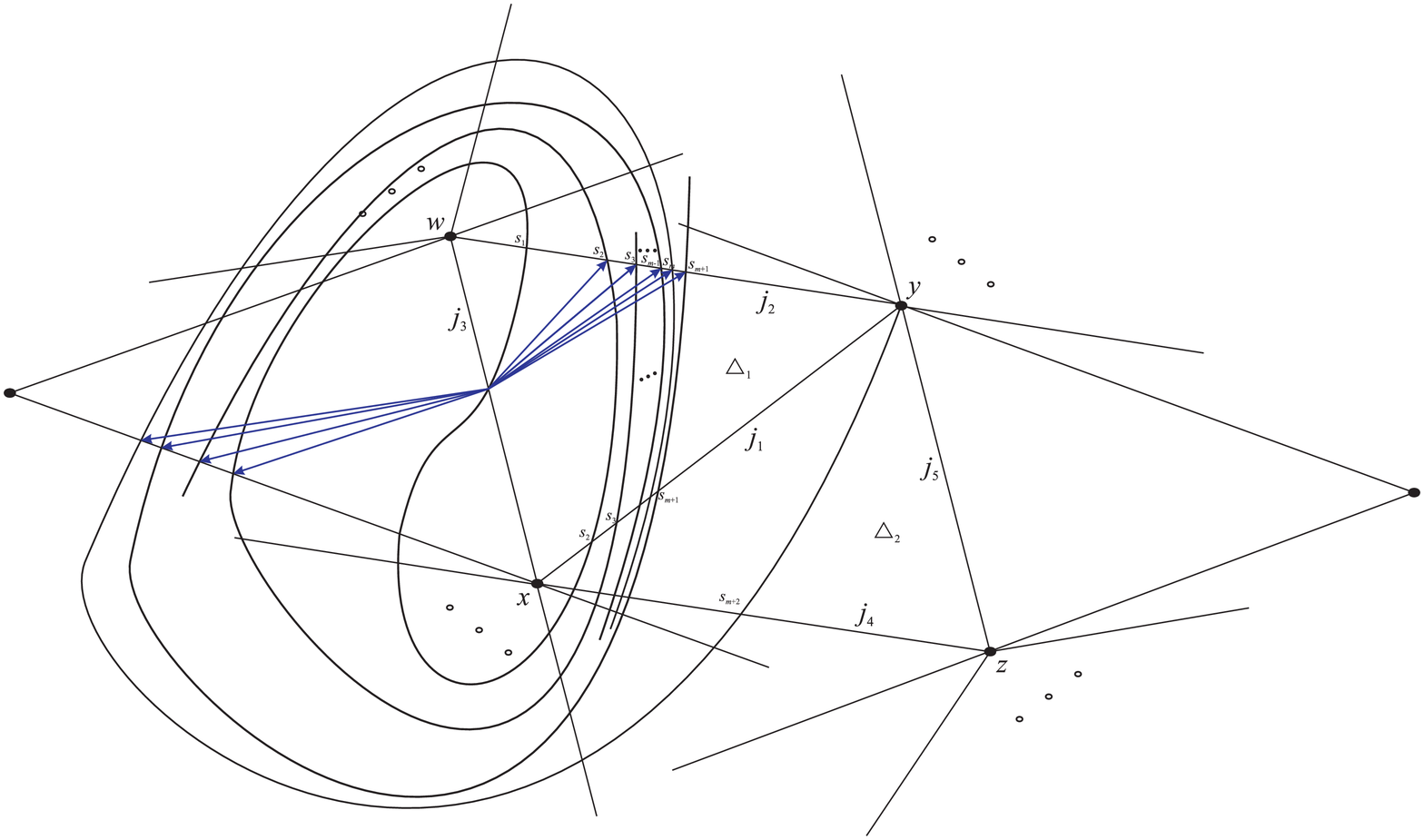}
        \end{figure}

The relevant vector spaces assigned in $\Mtauarc$ to the vertices of $Q(\tau)$ are
$$
M_{\arcone_{1}}=\Mtauarc_{\arcone_{1}}=\field^{m}, \ M_{\arcone_{2}}=\Mtauarc_{\arcone_{2}}=\field^{m+1},
$$
$$
M_{\arcone_{3}}=\Mtauarc_{\arcone_{3}}=\field, \ M_{\arcone_{4}}=\Mtauarc_{\arcone_{4}}=\field^{m+1},
$$
$$
\text{and} \ M_{\arcone_{5}}=\Mtauarc_{\arcone_{5}}=0.
$$

We also have
$$
\badintersection^{\triangle_1,l}_{\arc,\arcone_1}=\badintersection^{\triangle_2,l}_{\arc,\arcone_1}=
\badintersection^{\triangle_1,l}_{\arc,\arcone_3}=\badintersection^{\triangle_2,l}_{\arc,\arcone_5}=
\badintersection^{\triangle_2,l}_{\arc,\arcone_4}=\varnothing \ \text{for} \ l\geq 1, \ \ \text{and}
$$
$$
\badintersection^{\triangle_1,l}_{\arc,\arcone_2}=\{(s_1,s_{l+1},b(d^{\triangle_1,l}_{\arc,\arcone_2}),x)\} \ \text{for} \ 1\leq l\leq m \ \text{and} \
$$

The relevant detour matrices are therefore defined as follows. The matrices $D^{\triangle_1}_{\arc,\arcone_1}$, $D^{\triangle_2}_{\arc,\arcone_1}$, $D^{\triangle_1}_{\arc,\arcone_3}$, $D^{\triangle_2}_{\arc,\arcone_4}$ and $D^{\triangle_2}_{\arc,\arcone_5}$ are identities (of the corresponding sizes), whereas
$$
D^{\triangle_1}_{\arc,\arcone_2}=
{\tiny\left[\begin{array}{ccccc}
1 & & \mathbf{0}_{1\times m} &  \\
-x                    &  &  &  \\
xw                    &  &  & & \\
\vdots                &   &       & \\
(-1)^{l-1}x^{\lfloor\frac{l}{2}\rfloor}w^{\lfloor\frac{l-1}{2}\rfloor} &  & \text{{\Huge{\textbf{1}}$_{m\times m}$}} &  \\
\vdots                &   &        &   &  \\
(-1)^{m}x^{\lfloor\frac{m+1}{2}\rfloor}w^{\lfloor\frac{m}{2}\rfloor}   &  &  &  \end{array}\right]}.
$$
Hence the arrows $\alpha$, $\beta$, $\gamma$, $\delta$, $\varepsilon$ and $\eta$ act on $\Mtauarc$ according to the following linear maps:
$$
\Mtauarc_\alpha=(D_{\arc,\arcone_2}^{\triangle_1})(\mtauarc_\alpha)=
{\tiny\left[\begin{array}{c}
1 \\
-x \\
xw \\
\vdots \\
(-1)^{l-1}x^{\lfloor\frac{l}{2}\rfloor}w^{\lfloor\frac{l-1}{2}\rfloor} \\
\vdots \\
(-1)^{m}x^{\lfloor\frac{m+1}{2}\rfloor}w^{\lfloor\frac{m}{2}\rfloor} \end{array}\right]}:\field\rightarrow\field^{m+1},
$$
$$
\Mtauarc_\beta=(D_{\arc,\arcone_3}^{\triangle_1})(\mtauarc_\beta)
=\zero:\field^{m}\rightarrow \field,
$$
$$
\Mtauarc_\gamma=(D_{\arc,\arcone_1}^{\triangle_1})(\mtauarc_\gamma)
=\left[\begin{array}{cc}\mathbf{0}_{m\times 1} & \id_{m\times m}\end{array}\right]:\field^{m+1}\rightarrow \field^{m},
$$
$$
\Mtauarc_\delta=(D_{\arc,\arcone_4}^{\triangle_2})(\mtauarc_\delta)=\zero:\zero\rightarrow\field^{m+1},
$$
$$
\Mtauarc_\varepsilon=
(D_{\arc,\arcone_5}^{\triangle_2})(\mtauarc_\varepsilon)=
\zero:\field^{m}\rightarrow \zero,
$$
$$
\text{and} \
\Mtauarc_\eta=(D_{\arc,\arcone_1}^{\triangle_2})(\mtauarc_\eta)=
$$
$$
=\left[\begin{array}{cc}
\mathbf{1}_{m\times m} & \mathbf{0}_{n\times 1}
\end{array}\right]:\field^{m+1}\rightarrow\field^{m}.
$$

Let us investigate the effect of the $\arcone_1^{\operatorname{th}}$ QP-mutation on $\calMtauarc$. An easy check using the information about $\calMtauarc$ we have collected thus far yields
$$
M_{\oin}=M_{\arcone_2}\oplus M_{\arcone_4}=\field^{m+1}\oplus\field^{m+1},
$$
$$
M_{\oout}=M_{\arcone_3}\oplus M_{\arcone_5}=\field\oplus 0.
$$
$$
\al=\left[\begin{array}{cccc}
\mathbf{0}_{m\times 1} & \mathbf{1}_{m\times m} & \mathbf{1}_{m\times m} & \mathbf{0}_{m\times 1} \end{array}\right]:M_{\oin}=\field^{m+1}\oplus\field^{m+1}\rightarrow \field^{m}=M_{\arcone_1},
$$
$$
\be=\zero:M_{\arcone_1}=\field^{m}\rightarrow\field\oplus 0=M_{\oout},
$$
$$
\ga={\tiny\left[\begin{array}{c|c}
1 & -\\
-x & -\\
wx & -\\
\vdots & \vdots \\
(-1)^{m-1}x^{\lfloor\frac{m}{2}\rfloor}w^{\lfloor\frac{m-1}{2}\rfloor} & -\\
(-1)^{m}x^{\lfloor\frac{m+1}{2}\rfloor}w^{\lfloor\frac{m}{2}\rfloor} & -\\
x & -\\
-wx & -\\
wx^2 & -\\
\vdots &  \vdots\\
(-1)^{m-1}w^{\lfloor\frac{m}{2}\rfloor}x^{\lfloor\frac{m+1}{2}\rfloor} & - \\
(-1)^{m}w^{\lfloor\frac{m+1}{2}\rfloor}x^{\lfloor\frac{m+2}{2}\rfloor} & -
\end{array}\right]}:
M_{\oout}=\field\oplus 0\rightarrow \field^{m+1}\oplus\field^{m+1}=M_{\oin}.
$$

It is easily seen that $\al$ is surjective and $\ga$ is injective. Moreover, a straightforward computation shows that $\ker\al=\{(u_1,u_2,\ldots,u_{m+1},-u_2,\ldots,-u_{m+1},v)\suchthat u_1,\ldots,u_{m+1},v\in\field\}$, which is isomorphic to $\field^{m+2}$ under the linear map $\ell:(u_1,u_2,\ldots,u_{m+1},-u_2,\ldots,-u_{m+1},v)\mapsto
(u_1,u_2,\ldots,u_{m+1},v)$. The image of $\image\ga$ under $\ell$ is $\ell(\image\ga)=
\{(u,-xu,xwu,\ldots,(-1)^mx^{\lfloor\frac{m+1}{2}\rfloor}w^{\lfloor\frac{m}{2}\rfloor}u,
(-1)^{m}x^{\lfloor\frac{m+2}{2}\rfloor}w^{\lfloor\frac{m+1}{2}\rfloor}u)\suchthat u\in\field\}$. Therefore,
$\frac{\ker\al}{\image\ga}\cong\field^{m+1}$, and we can describe the canonical projection $\ker\al\twoheadrightarrow\frac{\ker\al}{\image\ga}$ by means of the matrix
$$
{\tiny\left[\begin{array}{c}
x \\
-xw \\
\vdots \\
(-1)^{m+1}x^{\lfloor\frac{m+1}{2}\rfloor}w^{\lfloor\frac{m}{2}\rfloor} \\
(-1)^{m+1}x^{\lfloor\frac{m+2}{2}\rfloor}w^{\lfloor\frac{m+1}{2}\rfloor}
\end{array}
\begin{array}{c}
 \\
 \\
 \text{{\Huge{\textbf{1}}$_{(m+1)\times (m+1)}$}} \\
 \\
\end{array}\right]}:\field^{m+2}\rightarrow\field^{m+1}
$$
and the inclusion $\ker\al\hookrightarrow M_{\oin}$ by means of the matrix
$$
\io=
\left[\begin{array}{cc}
\text{{\Huge{\textbf{1}}$_{(m+1)\times (m+1)}$}} & \text{{\Large{\textbf{0}}$_{(m+1)\times 1}$}} \\
\text{{\Large{\textbf{0}}$_{m\times 1}$}} \ \ \ \ \ \ \ \ \ \text{{\LARGE $-${\textbf{1}}$_{m\times m}$}} & \text{{\Large{\textbf{0}}$_{m\times 1}$}}  \\
 \ \ \ \text{{\Large{\textbf{0}}$_{1\times (m+1)}$}} & 1
\end{array}\right]
:\field^{m+2}\rightarrow\field^{m+1}\oplus\field^{m+1}.
$$

We deduce that
\begin{displaymath}
\overline{M}_{\arcone_1}=0\oplus\field\oplus\field^{m+1}\oplus 0 \ \ \text{and} \ \
\overline{V}_{\arcone_1}=0.
\end{displaymath}

Now, from the fact that $\beta$, and $\varepsilon$ act as zero on $\Mtauarc$, we conclude that the arrows $[\beta\gamma]$, $[\varepsilon\gamma]$, $[\beta\eta]$ and $[\varepsilon\eta]$ of $\tildeqtau$ act as zero on $\overline{\Mtauarc}$. Since the arrows of $\widetilde{\mu}_{\arcone_1}(\qtau)$ not incident to $\arcone_1$ act on $\overline{\Mtauarc}$ in the exact same way they act on $\Mtauarc$, we just have to find out how the arrows $\beta^*$, $\gamma^*$, $\varepsilon^*$ and $\eta^*$ of $\tildeqtau$ act on $\overline{\Mtauarc}$. To this end, we choose the zero retraction $\rh:M_{\oout}\rightarrow\ker\ga=0$ and the section $\si:\frac{\ker\al}{\image\ga}\rightarrow\ker\al$ given by the matrix
$$
\si=\left[\begin{array}{c}
\text{{\large{\textbf{0}}$_{1\times (m+1)}$}}\\
\text{{\Large{\textbf{1}}$_{(m+1)\times (m+1)}$}}\end{array}\right]
:\field^{m+1}\rightarrow\field^{m+2}.
$$
A straightforward check yields
$$
\io\si=\left[\begin{array}{cc}
 \text{{\Large{\textbf{0}}$_{1\times m}$}} & 0 \\
\text{{\LARGE{\textbf{1}}$_{m\times m}$}} & \text{{\Large{\textbf{0}}$_{m\times 1}$}} \\
 \ \ \ \ \ \ \text{{\LARGE $-${\textbf{1}}$_{m\times m}$}} & \text{{\Large{\textbf{0}}$_{m\times 1}$}}  \\
 \text{{\Large{\textbf{0}}$_{1\times m}$}} & 1
\end{array}\right].
$$
From all these pieces of information we deduce that the action of $\beta^*$ and $\varepsilon^*$ is encoded by the matrix
$$
[\beta^* \ \varepsilon^*]={\tiny\left[\begin{array}{c|c}- & -\\
-1 & -\\
\zero_{(m+1)\times 1} & -\\
- & - \end{array}\right]}:\field\oplus0\rightarrow 0\oplus\field\oplus\field^{m+1}\oplus 0,
$$
whereas the arrows $\gamma^*$ and $\eta^*$ act according to the matrix
$$
\left[\begin{array}{c}\gamma^*\\
                      \eta^*\end{array}\right]=
\left[\begin{array}{cccc}
- & 1 & \mathbf{0}_{1\times m} & 0 \\
- & -x & & \\
- & wx & & \\
 & \vdots & \text{{\Huge{\textbf{1}}$_{m\times m}$}} & \text{{\Large{\textbf{0}}$_{m\times 1}$}}  \\
- & (-1)^{m-1}x^{\lfloor\frac{m}{2}\rfloor}w^{\lfloor\frac{m-1}{2}\rfloor} \\
- & (-1)^{m}x^{\lfloor\frac{m+1}{2}\rfloor}w^{\lfloor\frac{m}{2}\rfloor} \\
\hline
- & x & & \\
- & -wx & & \\
- & wx^2 & & \\
 & \vdots & \ \ \ \ \ \ \text{{\Huge $-${\textbf{1}}$_{m\times m}$}} & \text{{\Large{\textbf{0}}$_{m\times 1}$}}  \\
- & (-1)^{m-1}w^{\lfloor\frac{m}{2}\rfloor}x^{\lfloor\frac{m+1}{2}\rfloor} & & \\
- & (-1)^{m}w^{\lfloor\frac{m+1}{2}\rfloor}x^{\lfloor\frac{m+2}{2}\rfloor} &  \mathbf{0}_{1\times m} & 1
\end{array}\right]
:0\oplus\field\oplus\field^{m+1}\oplus 0\rightarrow \field^{m+1}\oplus\field^{m+1}.
$$

This completes the computation of the action of the arrows of $\widetilde{\mu}_{\arcone_1}(Q(\tau))$ on $\overline{\Mtauarc}$.
We have thus computed the premutation $\widetilde{\mu}_{\arcone_1}(\calMtauarc)=(\widetilde{\mu}_{\arcone_1}(Q(\tau)),\tildestau,$ $\overline{\Mtauarc},0)$.

On the other hand, if we flip the arc $\arcone_1$ of $\tau$ we obtain the ideal triangulation $\sigma=f_{\arcone_1}(\tau)$ sketched in Figure \ref{curvecase1var2flipped}  (in a clear abuse of notation, we are using the same symbol $\arcone_1$ in both $\tau$ and $\sigma$).
\begin{figure}[!h]
                \caption{}\label{curvecase1var2flipped}
                \centering
                \includegraphics[scale=.50]{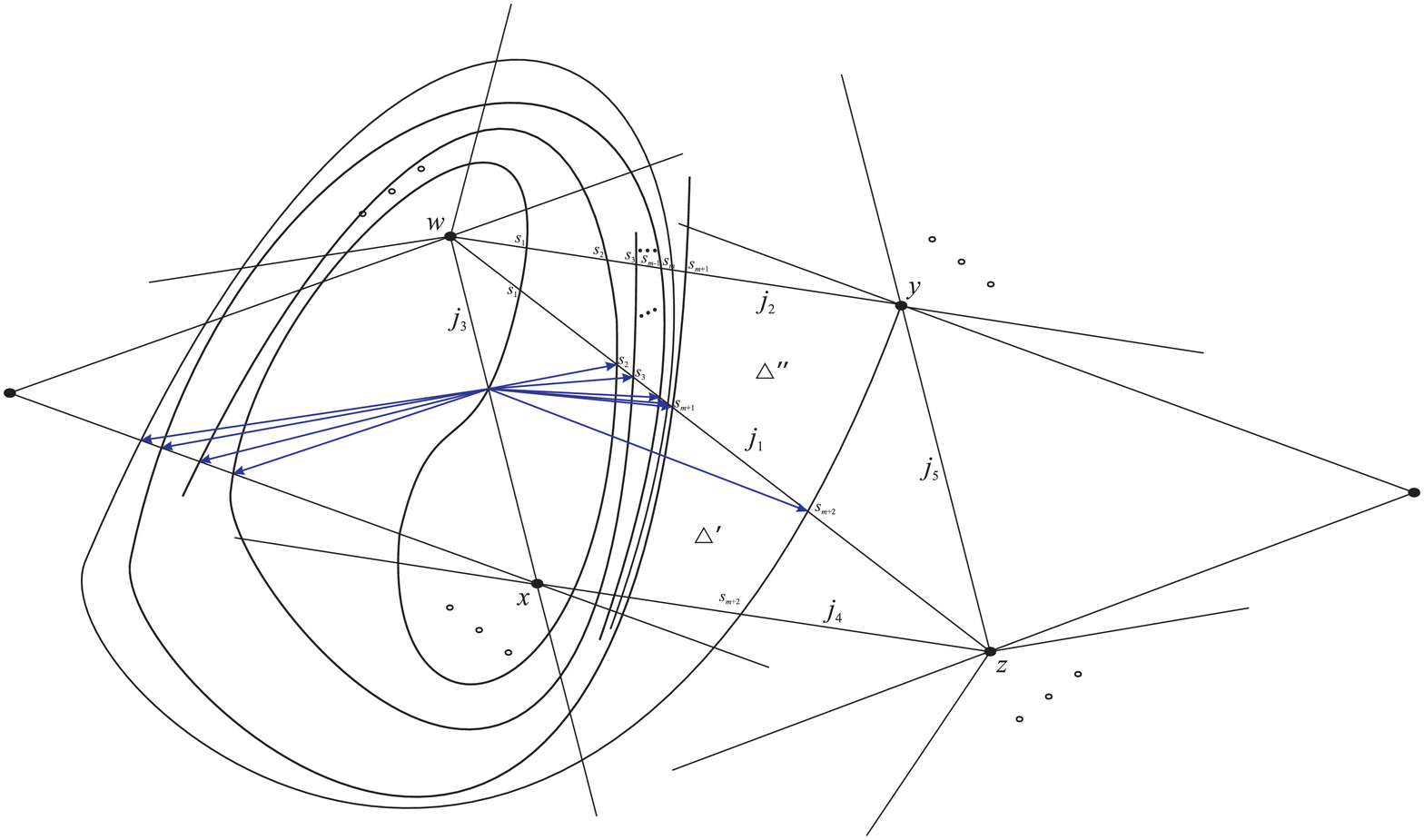}
        \end{figure}\\
The relevant vector spaces attached to the vertices of $Q(\sigma)$ are
$$
N_{\arcone_{1}}=\Msigmaarc_{\arcone_{1}}=\field^{m+2}, \ N_{\arcone_{2}}=\Msigmaarc_{\arcone_2}=\field^{m+1},
$$
$$
N_{\arcone_{3}}=\Msigmaarc_{\arcone_{3}}=\field, \ N_{\arcone_{4}}=\Msigmaarc_{\arcone_4}=\field^{m+1},
$$
$$
\text{and} \ N_{\arcone_5}=\Msigmaarc_{\arcone_5}=0.
$$
We also have
$$
\badintersection^{\triangle',l}_{\arc,\arcone_1}=\{(s_1,s_{l+1},b(d^{\triangle',l}_{\arc,\arcone_1}),x)\} \ \text{for} \ 1\leq l\leq m+1\},
$$
$$
\text{and} \ \badintersection^{\triangle',m+1+r}_{\arc,\arcone_1}=\badintersection^{\triangle',r}_{\arc,\arcone_2}=
\badintersection^{\triangle',r}_{\arc,\arcone_3}=
\badintersection^{\triangle',r}_{\arc,\arcone_4}=\badintersection^{\triangle',r}_{\arc,\arcone_5}=
$$
$$
=\badintersection^{\triangle'',r}_{\arc,\arcone_1}=\badintersection^{\triangle'',r}_{\arc,\arcone_2}=\badintersection^{\triangle'',r}_{\arc,\arcone_3}=
\badintersection^{\triangle'',r}_{\arc,\arcone_4}=\badintersection^{\triangle'',r}_{\arc,\arcone_5}=\varnothing \ \text{for} \ r\geq 1.
$$

The relevant detour matrices are therefore defined as follows. The matrices $D_{\arc,\arcone_1}^{\triangle''}$, $D^{\triangle''}_{\arc,\arcone_2}$, $D^{\triangle'}_{\arc,\arcone_3}$, $D^{\triangle'}_{\arc,\arcone_4}$ and $D^{\triangle''}_{\arc,\arcone_5}$ are identities (of the corresponding sizes), whereas
$$
D_{\arc,\arcone_1}^{\triangle'}=
{\tiny\left[\begin{array}{ccccc}
1 & & \mathbf{0}_{1\times (m+1)} &  \\
-x                    &  &  &  \\
xw                    &  &  & & \\
\vdots                &   & \text{{\Huge{\textbf{1}}$_{(m+1)\times (m+1)}$}} &  \\
(-1)^{m}x^{\lfloor\frac{m+1}{2}\rfloor}w^{\lfloor\frac{m}{2}\rfloor}   &  &  &  \\
(-1)^{m+1}x^{\lfloor\frac{m+2}{2}\rfloor}w^{\lfloor\frac{m+1}{2}\rfloor}  &  &  &  \end{array}\right]}.
$$
Hence the arrows $\beta^*$, $[\beta\eta]$, $\eta^*$, $\gamma^*$, $\varepsilon^*$ and $[\varepsilon\gamma]$ act on $\Msigmaarc$ according to the following linear maps:
$$
\Msigmaarc_{\beta^*}=(D_{\arc,\arcone_1}^{\triangle'})(\msigmaarc_{\beta^*})=
{\tiny\left[\begin{array}{c}
1 \\
-x \\
xw \\
\vdots \\
(-1)^{m}x^{\lfloor\frac{m+1}{2}\rfloor}w^{\lfloor\frac{m}{2}\rfloor} \\
(-1)^{m+1}x^{\lfloor\frac{m+2}{2}\rfloor}w^{\lfloor\frac{m+1}{2}\rfloor} \end{array}\right]}:\field\rightarrow\field^{m+2},
$$
$$
\Msigmaarc_{[\beta\eta]}=
(D_{\arc,\arcone_3}^{\triangle'})(\msigmaarc_{[\beta\eta]})=
\zero:\field^{m+1}\rightarrow \field,
$$
$$
\Msigmaarc_{\eta^*}=(D_{\arc,\arcone_4}^{\triangle'})(\msigmaarc_{\eta^*})=
\left[\begin{array}{cc}
\mathbf{0}_{(m+1)\times 1} & \mathbf{1}_{(m+1)\times(m+1)}\end{array}\right]:
\field^{m+2}\rightarrow \field^{m+1},
$$
$$
\Msigmaarc_{\gamma^*}=(D_{\arc,\arcone_2}^{\triangle''})(\msigmaarc_{\gamma^*})=
\left[\begin{array}{cc}
\mathbf{1}_{(m+1)\times(m+1)} & \mathbf{0}_{(m+1)\times 1}\end{array}\right]:
\field^{m+2}\rightarrow \field^{m+1},
$$
$$
\Msigmaarc_{\varepsilon^*}=(D_{\arc,\arcone_1}^{\triangle''})(\msigmaarc_{\varepsilon^*})=
\zero:
\zero\rightarrow \field^{m+2},
$$
$$
\Msigmaarc_{[\varepsilon\gamma]}=
(D_{\arc,\arcone_5}^{\triangle''})(\msigmaarc_{[\varepsilon\gamma]})=
\zero:\field^{m+1}\rightarrow \zero.
$$

We have thus computed the spaces and linear maps of $\calMsigmaarc$ relevant to the flip of the arc $\arcone_1$. Now we have to compare it to $\mu_{\arcone_1}(\calMtauarc)$. The triple $\Phi=(\varphi,\psi,\eta)$ is a right-equivalence between these QP-representations, where
\begin{itemize}\item $\varphi:(\overline{Q(\tau)},\lambda(\tildestau)-\alpha[\beta\gamma]-\delta[\varepsilon\eta])\rightarrow \qssigma$ is the right-equivalence whose action on the arrows is given by
$$
\beta^*\mapsto-\beta^*, \ \eta^*\mapsto-\eta^*,
$$
and the identity in the rest of the arrows;
\item $\psi:\overline{\Mtauarc}\rightarrow\Msigmaarc$ is the vector space isomorphism given by the identity $\id:\overline{\Mtauarc}_{\arctwo}\rightarrow\Msigmaarc_{\arctwo}$ for $\arctwo\neq\arcone_1$, and the matrix
    $$
    \psi_{\arcone_{1}}=
{\tiny\left[\begin{array}{ccc}
1 & \text{{\large{\textbf{0}}$_{1\times m}$}} & 0 \\
-x & & \\
wx & & \\
\vdots & \text{{\Huge{\textbf{1}}$_{m\times m}$}} & \text{{\Large{\textbf{0}}$_{m\times 1}$}}  \\
(-1)^{m}x^{\lfloor\frac{m+1}{2}\rfloor}w^{\lfloor\frac{m}{2}\rfloor} & & \\
(-1)^{m+1}x^{\lfloor\frac{m+2}{2}\rfloor}w^{\lfloor\frac{m+1}{2}\rfloor} & \text{{\large{\textbf{0}}$_{1\times m}$}} & -1 \\
\end{array}\right]}
:\overline{\Mtauarc}_{\arcone_1}\rightarrow\Msigmaarc_{\arcone_1};
    $$
\item $\eta$ is the zero map (of the zero space).
\end{itemize}

This finishes the proof of Theorem \ref{thm:flip<->mut} for configuration 2 of Figure \ref{Fig:component1}. Furthermore, ff $[\beta\eta]b$ and $[\varepsilon\gamma]c$ are not 2-cycles, the case just analyzed deals also with configuration 5 of that Figure.
\end{case}

\begin{case} This is configuration 1 from Figure \ref{Fig:component3punctmonogon}. Assume that, around the arc $\arcone_1$ to be flipped, $\tau$ and $\arc$ look as shown in Figure \ref{Fig:flippingi'nonnegative}.
\begin{figure}[!h]
                \caption{}\label{Fig:flippingi'nonnegative}
                \centering
                \includegraphics[scale=.5]{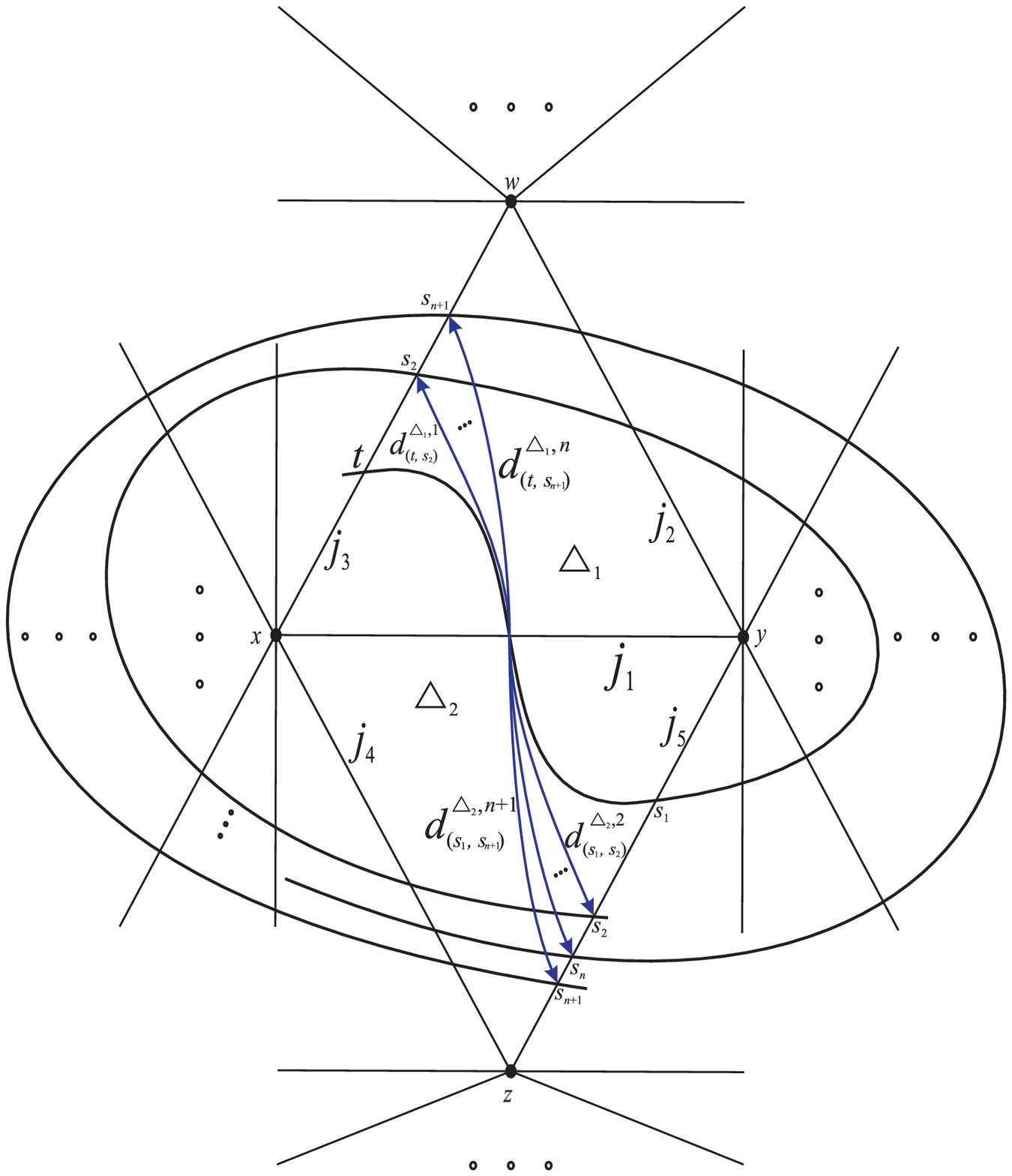}
        \end{figure}

The relevant vector spaces assigned in $\Mtauarc$ to the vertices of $Q(\tau)$ are
$$
M_{\arcone_{1}}=\Mtauarc_{\arcone_{1}}=\field, \ M_{\arcone_{2}}=\Mtauarc_{\arcone_{2}}=\field^{n},
$$
$$
M_{\arcone_{3}}=\Mtauarc_{\arcone_{3}}=\field^n, \ M_{\arcone_{4}}=\Mtauarc_{\arcone_{4}}=\field^{n},
$$
$$
\text{and} \ M_{\arcone_{5}}=\Mtauarc_{\arcone_{5}}=\field^{n+1}.
$$

We also have
$$
\badintersection^{\triangle_1,l}_{\arc,\arcone_1}=\badintersection^{\triangle_2,l}_{\arc,\arcone_1}=
\badintersection^{\triangle_1,l}_{\arc,\arcone_2}=\badintersection^{\triangle_2,l}_{\arc,\arcone_4}=
\badintersection^{\triangle_1,2l}_{\arc,\arcone_3}=
\badintersection^{\triangle_2,2l-1}_{\arc,\arcone_5}=
\varnothing \ \text{for} \ l\geq 1, \ \ \text{and}
$$
$$
\badintersection^{\triangle_1,2l-1}_{\arc,\arcone_3}=\{(t,s_{l+1},b(d^{\triangle_1,2l-1}_{\arc,\arcone_2}),y)\} \ \text{for} \ 1\leq l\leq n \ \ \text{and} \ \
\badintersection^{\triangle_2,2l}_{\arc,\arcone_5}=\{(s_1,s_{l+1},b(d^{\triangle_2,2l}_{\arc,\arcone_2}),x)\} \ \text{for} \ 2\leq l\leq n+1.
$$

The relevant detour matrices are therefore defined as follows. The matrices $D^{\triangle_1}_{\arc,\arcone_1}$, $D^{\triangle_2}_{\arc,\arcone_1}$, $D^{\triangle_1}_{\arc,\arcone_2}$ and $D^{\triangle_2}_{\arc,\arcone_4}$ are identities (of the corresponding sizes), whereas
$$
D^{\triangle_1}_{\arc,\arcone_3}=
{\tiny\left[\begin{array}{ccccc}
1 & & \mathbf{0}_{1\times n} &  \\
-y                    &  &  &  \\
-y^2x                    &  &  & & \\
\vdots                &   &       & \\
-y^{l}x^{l-1} &  & \text{{\Huge{\textbf{1}}$_{n\times n}$}} &  \\
\vdots                &   &        &   &  \\
-y^{n}x^{n-1}   &  &  &  \end{array}\right]}, \
D^{\triangle_2}_{\arc,\arcone_5}=
{\tiny\left[\begin{array}{ccccc}
1 & & \mathbf{0}_{1\times n} &  \\
xy                    &  &  &  \\
x^2y^2                    &  &  & & \\
\vdots                &   &       & \\
x^{l}y^{l} &  & \text{{\Huge{\textbf{1}}$_{n\times n}$}} &  \\
\vdots                &   &        &   &  \\
x^{n+1}y^{n+1}   &  &  &  \end{array}\right]}.
$$
Hence the arrows $\alpha$, $\beta$, $\gamma$, $\delta$, $\varepsilon$ and $\eta$ act on $\Mtauarc$ according to the following linear maps:
$$
\Mtauarc_\alpha=(D_{\arc,\arcone_2}^{\triangle_1})(\mtauarc_\alpha)
=\id:\field^n\rightarrow\field^{n},
$$
$$
\Mtauarc_\beta=\pi(D_{\arc,\arcone_3}^{\triangle_1})(\mtauarc_\beta)
={\tiny\left[\begin{array}{c}
-y \\
-y^2x \\
\vdots \\
-y^{l}x^{l-1} \\
\vdots \\
-y^{n}x^{n-1}\end{array}\right]}
:\field\rightarrow \field^n,
$$
$$
\Mtauarc_\gamma=(D_{\arc,\arcone_1}^{\triangle_1})(\mtauarc_\gamma)=
\zero:\field^{n}\rightarrow \field,
$$
$$
\Mtauarc_\delta=(D_{\arc,\arcone_4}^{\triangle_2})(\mtauarc_\delta)=[\zero_{n\times 1} \ \id_{n\times n}]:\field^{n+1}\rightarrow\field^{n},
$$
$$
\Mtauarc_\varepsilon=
(D_{\arc,\arcone_5}^{\triangle_2})(\mtauarc_\varepsilon)=
{\tiny\left[\begin{array}{c}
1 \\
xy \\
x^2y^2 \\
\vdots \\
x^{l}y^{l} \\
\vdots \\
x^{n}y^{n}\end{array}\right]}
:\field\rightarrow \field^{n+1},
$$
$$
\text{and} \
\Mtauarc_\eta=(D_{\arc,\arcone_1}^{\triangle_2})(\mtauarc_\eta)=
\zero:\field^{n}\rightarrow\field.
$$

Let us investigate the effect of the $\arcone_1^{\operatorname{th}}$ QP-mutation on $\calMtauarc$. An easy check using the information about $\calMtauarc$ we have collected thus far yields
$$
M_{\oin}=M_{\arcone_2}\oplus M_{\arcone_4}=\field^{n}\oplus\field^{n},
$$
$$
M_{\oout}=M_{\arcone_3}\oplus M_{\arcone_5}=\field^n\oplus\field^{n+1}.
$$
$$
\al=\zero:M_{\oin}=\field^{n}\oplus\field^{n}\rightarrow \field=M_{\arcone_1},
$$
$$
\be={\tiny\left[\begin{array}{c}
-y \\
-y^2x \\
\vdots \\
-y^{n}x^{n-1}\\
1 \\
xy \\
x^2y^2 \\
\vdots\\
x^{n}y^{n}
\end{array}\right]}:M_{\arcone_1}=\field\rightarrow\field^{n}\oplus \field^{n+1}=M_{\oout},
$$
$$
\ga={\tiny\left[\begin{array}{c|c}
\text{{\Huge{\textbf{1}}$_{n\times n}$}} & \text{{\Huge $y${\textbf{1}}$_{n\times n}$}} \ \text{{\large{\textbf{0}}$_{n\times 1}$}} \\
\hline
\text{{\Huge $x${\textbf{1}}$_{n\times n}$}} & \text{{\large{\textbf{0}}$_{n\times 1}$}} \ \ \  \text{{\Huge{\textbf{1}}$_{n\times n}$}}
\end{array}\right]}:
M_{\oout}=\field^n\oplus\field^{n+1}\rightarrow \field^{n}\oplus\field^{n}=M_{\oin}.
$$

It is trivially seen that $\be$ is injective and $\ker\al=\field{n}\oplus\field{n}$. Moreover, a straightforward computation shows that $\ker\ga=\{(u_1,\ldots,u_n,v_1,\ldots,v_{n+1})\in\field^{2n+1}\suchthat u_l=-yv_l \ \text{and} \ xu_l=v_{l+1} \ \text{for} \ 1\leq l\leq n\}=\{(-yv,xy^2v,\ldots,(-1)^nx^{n-1}y^nv,v,-xyv,\ldots,(-1)^nx^ny^nv)\}$, which is isomorphic to $\field$ under the linear map $\ell:(-yv,xy^2v,\ldots,(-1)^nx^{n-1}y^nv,v,-xyv,\ldots,(-1)^nx^ny^nv)\mapsto
v$. The image of $\image\be$ under $\ell$ is $\field$ and hence,
$\frac{\ker\ga}{\image\be}\cong 0$. Also, a standard dimension count yields surjectivity of $\ga$.
We deduce that
\begin{displaymath}
\overline{M}_{\arcone_1}=0\oplus(\field^n\oplus\field^n)\oplus0\oplus 0 \ \ \text{and} \ \
\overline{V}_{\arcone_1}=0.
\end{displaymath}

Now, from the fact that $\gamma$, and $\eta$ act as zero on $\Mtauarc$, we conclude that the arrows $[\beta\gamma]$, $[\varepsilon\gamma]$, $[\beta\eta]$ and $[\varepsilon\eta]$ of $\widetilde{\mu}_{\arcone_1}(\qtau)$ act as zero on $\overline{\Mtauarc}$. Since the arrows of $\widetilde{\mu}_{\arcone_1}(\qtau)$ not incident to $\arcone_1$ act on $\overline{\Mtauarc}$ in the exact same way they act on $\Mtauarc$, we just have to find out how the arrows $\beta^*$, $\gamma^*$, $\varepsilon^*$ and $\eta^*$ of $\tildeqtau$ act on $\overline{\Mtauarc}$. To this end, we choose the zero section $\si:\frac{\ker\al}{\image\ga}=0\rightarrow\ker\al$ and any retraction $\rh:M_{\oout}\rightarrow\ker\ga$.
A trivial check yields $\io\si=\mathbf{0}:0\rightarrow\field^n\oplus\field^n$ and $-\pipi\rh=\zero:\field^{n}\oplus\field{n+1}\rightarrow 0$
From all these pieces of information we deduce that the action of $\beta^*$ and $\varepsilon^*$ is encoded by the matrix
$$
[\beta^* \ \varepsilon^*]={\tiny\left[\begin{array}{c|c}- & -\\
\text{{\Huge $-${\textbf{1}}$_{n\times n}$}} & \text{{\Huge $-y${\textbf{1}}$_{n\times n}$}} \ \text{{\large{\textbf{0}}$_{n\times 1}$}} \\
\text{{\Huge $-x${\textbf{1}}$_{n\times n}$}} & \text{{\large{\textbf{0}}$_{n\times 1}$}} \ \ \  \text{{\Huge $-${\textbf{1}}$_{n\times n}$}} \\
- & -\\
- & - \end{array}\right]}:\field^n\oplus\field^{n+1}\rightarrow 0\oplus(\field^n\oplus\field^n)\oplus0\oplus 0,
$$
whereas the arrows $\gamma^*$ and $\eta^*$ act according to the matrix
$$
\left[\begin{array}{c}\gamma^*\\
                      \eta^*\end{array}\right]=\left[\begin{array}{ccccc}
- & \id_{n\times n} & \zero_{n\times n} & - & -\\
\hline
- & \zero_{n\times n} & \id_{n\times n} & - & -
\end{array}\right]
:0\oplus(\field^n\oplus\field^n)\oplus 0\oplus 0\rightarrow \field^{n}\oplus\field^{n}.
$$

This completes the computation of the action of the arrows of $\widetilde{\mu}_{\arcone_1}(Q(\tau))$ on $\overline{\Mtauarc}$.
We have thus computed the premutation $\widetilde{\mu}_{\arcone_1}(\calMtauarc)=(\widetilde{\mu}_{\arcone_1}(Q(\tau)),\tildestau,$ $\overline{\Mtauarc},0)$.

On the other hand, if we flip the arc $\arcone_1$ of $\tau$ we obtain the ideal triangulation $\sigma=f_{\arcone_1}(\tau)$ sketched in Figure \ref{Fig:flippingi'nonnegativeflipped}  (in a clear abuse of notation, we are using the same symbol $\arcone_1$ in both $\tau$ and $\sigma$).
\begin{figure}[!h]
                \caption{}\label{Fig:flippingi'nonnegativeflipped}
                \centering
                \includegraphics[scale=.50]{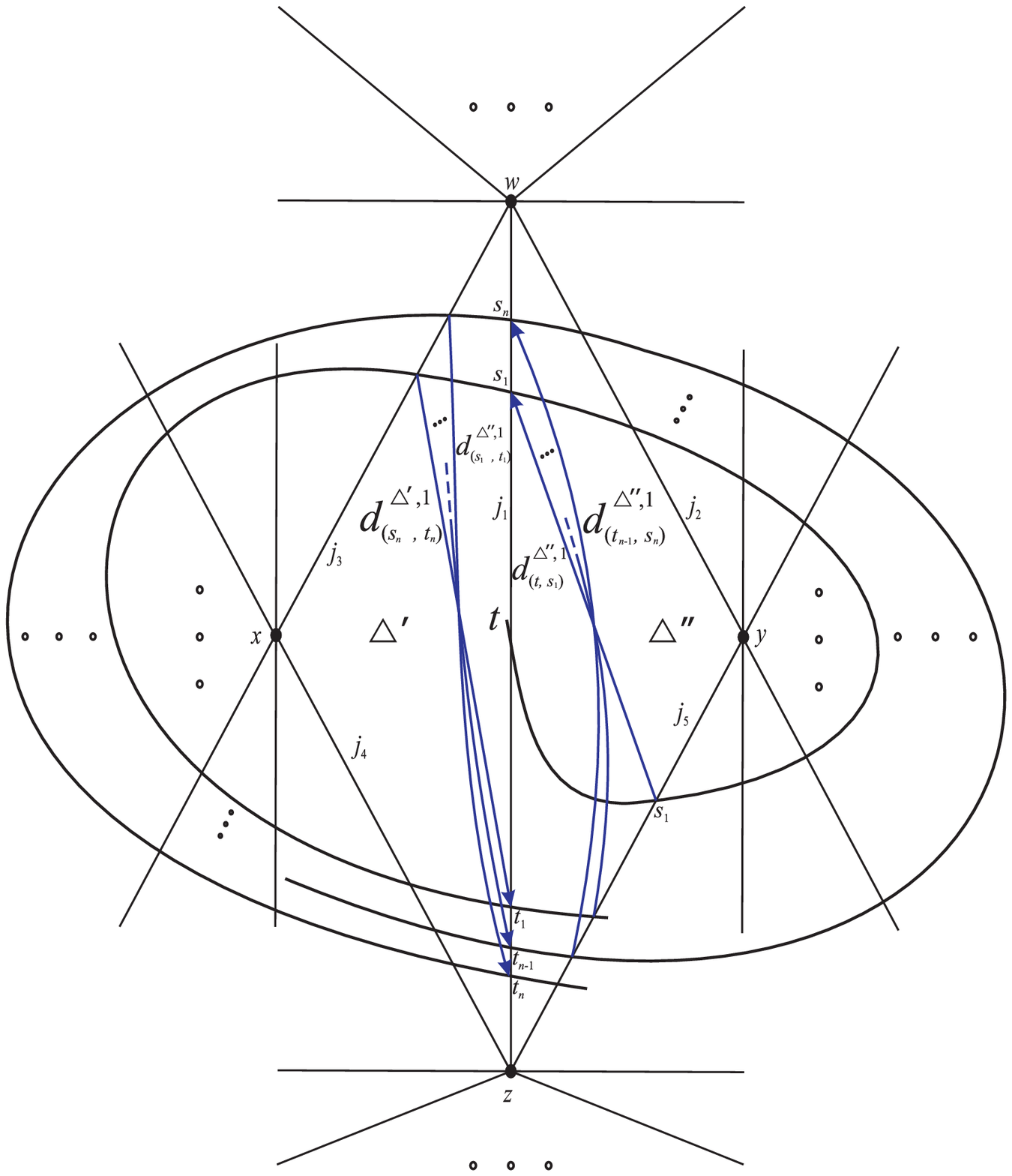}
        \end{figure}\\
The relevant vector spaces attached to the vertices of $Q(\sigma)$ are
$$
N_{\arcone_{1}}=\Msigmaarc_{\arcone_{1}}=\field^{2n}, \ N_{\arcone_{2}}=\Msigmaarc_{\arcone_2}=\field^{n},
$$
$$
N_{\arcone_{3}}=\Msigmaarc_{\arcone_{3}}=\field^n, \ N_{\arcone_{4}}=\Msigmaarc_{\arcone_4}=\field^{n},
$$
$$
\text{and} \ N_{\arcone_5}=\Msigmaarc_{\arcone_5}=\field^{n+1}.
$$
We also have
$$
\badintersection^{\triangle',1}_{\arc,\arcone_1}=\{(s_l,t_{l},b(d^{\triangle',1}_{s_l,t_{l}}),x) \suchthat 1\leq l\leq n\},
\ \badintersection^{\triangle'',1}_{\arc,\arcone_1}=\{(t,s_1,b(d^{\triangle'',1}_{t,s_1}),y)\}\cup
\{(t_l,s_{l+1},b(d^{\triangle'',1}_{t_l,s_{l+1}}),y) \suchthat 1\leq l\leq n-1\}
$$
$$
\text{and} \ \badintersection^{\triangle',r+1}_{\arc,\arcone_1}=\badintersection^{\triangle'',r+1}_{\arc,\arcone_1}=
\badintersection^{\triangle'',r}_{\arc,\arcone_2}=
\badintersection^{\triangle',r}_{\arc,\arcone_3}=
\badintersection^{\triangle',r}_{\arc,\arcone_4}=\badintersection^{\triangle'',r}_{\arc,\arcone_5}=\varnothing \ \text{for} \ r\geq 1.
$$

The relevant detour matrices are therefore defined as follows. The matrices $D^{\triangle''}_{\arc,\arcone_2}$, $D^{\triangle'}_{\arc,\arcone_3}$, $D^{\triangle'}_{\arc,\arcone_4}$ and $D^{\triangle''}_{\arc,\arcone_5}$ are identities (of the corresponding sizes), whereas
$$
D_{\arc,\arcone_1}^{\triangle'}=
{\tiny\left[\begin{array}{c}
\text{{\LARGE\textbf{1}$_{n\times n}$}}\\
\text{{\large\textbf{0}$_{1\times n}$}}\\
\text{{\LARGE$-x$\textbf{1}$_{n\times n}$}}
\end{array}
\begin{array}
{c}
\text{{\LARGE\textbf{0}$_{n\times (n+1)}$}}\\
\text{{\Huge\textbf{1}$_{(n+1)\times (n+1)}$}}
\end{array}\right]} \ \ \text{and}
$$
$$
D_{\arc,\arcone_1}^{\triangle''}=
{\tiny\left[\begin{array}{c}
\text{{\Huge\textbf{1}$_{n\times n}$}}\\
\text{{\Huge\textbf{0}$_{(n+1)\times n}$}}
\end{array}
\begin{array}{c}
\text{{\Huge$-y$\textbf{1}$_{n\times n}$}} \ \ \  \text{{\large\textbf{0}$_{n\times 1}$}}\\
\text{{\Huge\textbf{1}$_{(n+1)\times (n+1)}$}}
\end{array}
\right]}
$$
(here, the order in which the elements of $\arc\cap\arcone_1$ are taken is $(s_1,\ldots,s_n,t,t_1,\ldots,t_n)$)
Hence the arrows $\beta^*$, $[\beta\eta]$, $\eta^*$, $\gamma^*$, $\varepsilon^*$ and $[\varepsilon\gamma]$ act on $\Msigmaarc$ according to the following linear maps:
$$
\Msigmaarc_{\beta^*}=(\pi)(D_{\arc,\arcone_1}^{\triangle'})(\msigmaarc_{\beta^*})=
{\tiny\left[\begin{array}{c}
\id_{n\times n}\\
-x\id_{n\times n}
\end{array}\right]}:\field^n\rightarrow\field^{2n},
$$
$$
\Msigmaarc_{[\beta\eta]}=
(D_{\arc,\arcone_3}^{\triangle'})(\msigmaarc_{[\beta\eta]})=
\zero:\field^n\rightarrow \field^n,
$$
$$
\Msigmaarc_{\eta^*}=(D_{\arc,\arcone_4}^{\triangle'})(\msigmaarc_{\eta^*})(\ell)=
[\mathbf{0}_{n\times n} \ \mathbf{1}_{n\times n}]:
\field^{2n}\rightarrow \field^{n},
$$
$$
\Msigmaarc_{\gamma^*}=(D_{\arc,\arcone_2}^{\triangle''})(\msigmaarc_{\gamma^*})(\ell)=
[\id_{n\times n} \ \zero_{n\times n}]:
\field^{2n}\rightarrow \field^{n},
$$
$$
\Msigmaarc_{\varepsilon^*}=(\pi)(D_{\arc,\arcone_1}^{\triangle''})(\msigmaarc_{\varepsilon^*})=
{\tiny\left[\begin{array}{c}
\text{{\Huge$-y$\textbf{1}$_{n\times n}$}} \ \ \  \text{{\large\textbf{0}$_{n\times 1}$}}\\
\text{{\large\textbf{0}$_{n\times 1}$}} \ \ \ \text{{\Huge\textbf{1}$_{n\times n}$}}
\end{array}\right]}:
\field^{n+1}\rightarrow \field^{2n},
$$
$$
\Msigmaarc_{[\varepsilon\gamma]}=
(D_{\arc,\arcone_5}^{\triangle''})(\msigmaarc_{[\varepsilon\gamma]})=
\zero:\field^{n}\rightarrow \field^{n+1}.
$$

We have thus computed the spaces and linear maps of $\calMsigmaarc$ relevant to the flip of the arc $\arcone_1$. Now we have to compare it to $\mu_{\arcone_1}(\calMtauarc)$. The triple $\Phi=(\varphi,\psi,\eta)$ is a right-equivalence between these QP-representations, where
\begin{itemize}\item $\varphi:(\overline{Q(\tau)},\lambda(\tildestau)-\alpha[\beta\gamma]-\delta[\varepsilon\eta])\rightarrow \qssigma$ is the right-equivalence whose action on the arrows is given by
$$
\beta^*\mapsto-\beta^*, \ \eta^*\mapsto-\eta^*,
$$
and the identity in the rest of the arrows;
\item $\psi:\overline{\Mtauarc}\rightarrow\Msigmaarc$ is the vector space isomorphism given by the identity $\id:\overline{\Mtauarc}_{\arctwo}\rightarrow\Msigmaarc_{\arctwo}$ for $\arctwo\neq\arcone_1$, and the matrix
    $$
    \psi_{\arcone_{1}}=
\left[\begin{array}{cc}
\mathbf{1}_{n\times n} & \mathbf{0}_{n\times n}\\
\zero_{n\times n} & -\id_{n\times n}
\end{array}\right]
:\overline{\Mtauarc}_{\arcone_1}\rightarrow\Msigmaarc_{\arcone_1};
    $$
\item $\eta$ is the zero map (of the zero space).
\end{itemize}

This finishes the proof of Theorem \ref{thm:flip<->mut} for configuration 1 of Figure \ref{Fig:component3punctmonogon}.
\end{case}

\begin{case} This is going to be configuration 10 of Figure \ref{Fig:component2}. Assume that, around the arc $\arcone_1$ to be flipped, $\tau$ and $\arc$ look as shown in Figure \ref{Fig:punctmonogonnodetours}.
\begin{figure}[!h]
                \caption{}\label{Fig:punctmonogonnodetours}
                \centering
                \includegraphics[scale=.3]{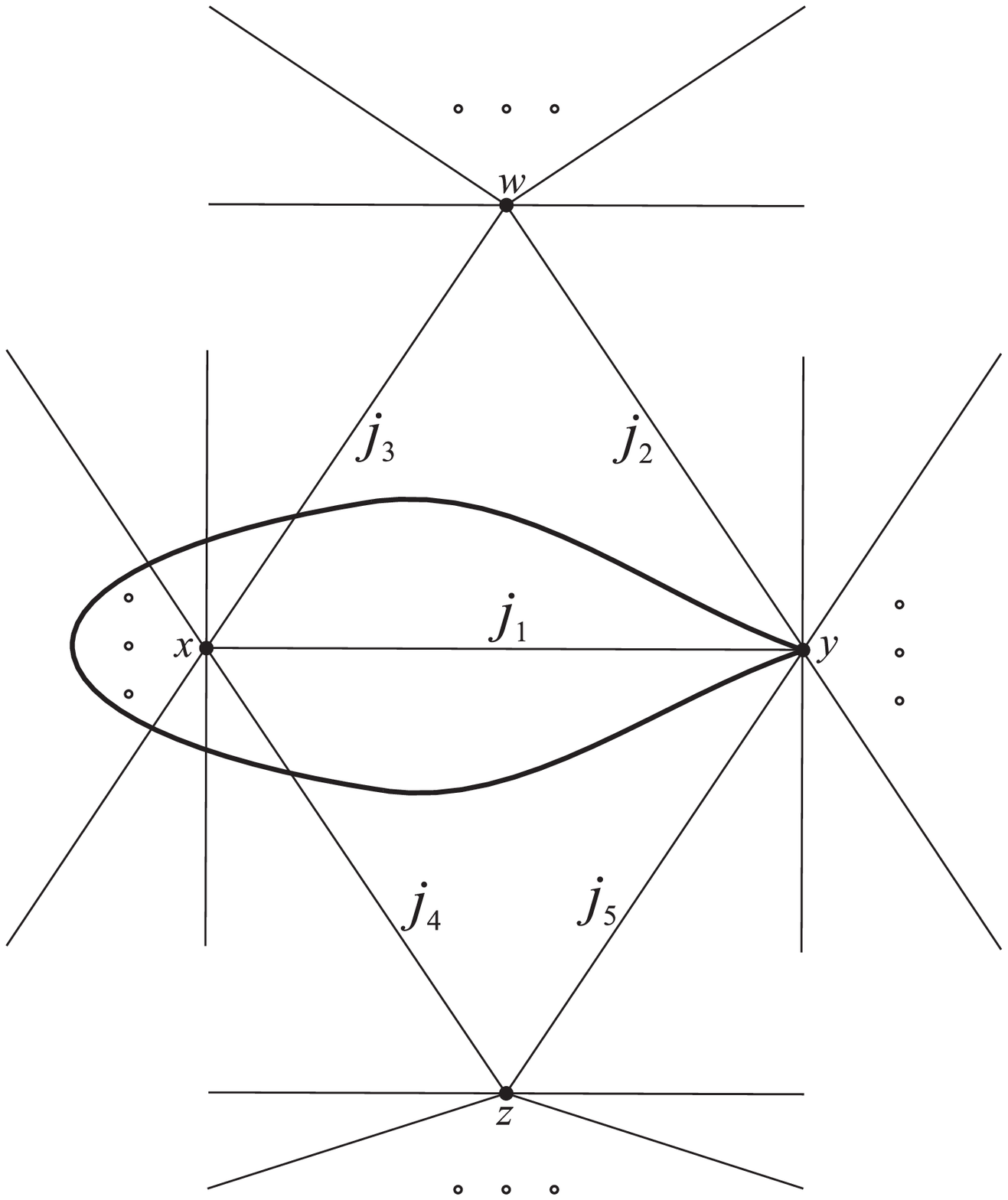}
        \end{figure}

The relevant vector spaces assigned in $\Mtauarc$ to the vertices of $Q(\tau)$ are
$$
M_{\arcone_{1}}=\Mtauarc_{\arcone_{1}}=0, \ M_{\arcone_{2}}=\Mtauarc_{\arcone_{2}}=0,
$$
$$
M_{\arcone_{3}}=\Mtauarc_{\arcone_{3}}=\field, \ M_{\arcone_{4}}=\Mtauarc_{\arcone_{4}}=\field,
$$
$$
\text{and} \ M_{\arcone_{5}}=\Mtauarc_{\arcone_{5}}=0.
$$

All relevant detour matrices are identities (of corresponding sizes)therefore defined as follows.
Hence the arrows $\alpha$, $\beta$, $\gamma$, $\delta$, $\varepsilon$ and $\eta$ act on $\Mtauarc$ according to the following linear maps:
$$
\Mtauarc_\alpha=(D_{\arc,\arcone_2}^{\triangle_1})(\mtauarc_\alpha)=
\zero:\field\rightarrow 0,
$$
$$
\Mtauarc_\beta=(D_{\arc,\arcone_3}^{\triangle_1})(\mtauarc_\beta)=
\zero:0\rightarrow \field,
$$
$$
\Mtauarc_\gamma=(D_{\arc,\arcone_1}^{\triangle_1})(\mtauarc_\gamma)=
\zero:0\rightarrow 0,
$$
$$
\Mtauarc_\delta=(D_{\arc,\arcone_4}^{\triangle_2})(\mtauarc_\delta)=
\zero:0\rightarrow\field,
$$
$$
\Mtauarc_\varepsilon=
(D_{\arc,\arcone_5}^{\triangle_2})(\mtauarc_\varepsilon)=
\zero:0\rightarrow 0,
$$
$$
\text{and} \
\Mtauarc_\eta=(D_{\arc,\arcone_1}^{\triangle_2})(\mtauarc_\eta)=
\zero:\field\rightarrow 0.
$$

Let us investigate the effect of the $\arcone_1^{\operatorname{th}}$ QP-mutation on $\calMtauarc$. An easy check using the information about $\calMtauarc$ we have collected thus far yields
$$
M_{\oin}=M_{\arcone_2}\oplus M_{\arcone_4}=0\oplus\field,
$$
$$
M_{\oout}=M_{\arcone_3}\oplus M_{\arcone_5}=\field\oplus 0.
$$
$$
\al=\zero:M_{\oin}=0\oplus\field\rightarrow 0=M_{\arcone_1},
$$
$$
\be=\zero:M_{\arcone_1}=0\rightarrow\field\oplus 0M_{\oout},
$$
$$
\ga=\left[\begin{array}{c|c}
- & -\\
\hline
x & -
\end{array}\right]:
M_{\oout}=\field\oplus 0\rightarrow 0\oplus\field=M_{\oin}.
$$

It is trivially seen that $\al$ and $\ga$ are surjective and $\be$ and $\ga$ are injective. We deduce that
\begin{displaymath}
\overline{M}_{\arcone_1}=0\oplus\field\oplus0\oplus 0 \ \ \text{and} \ \
\overline{V}_{\arcone_1}=0.
\end{displaymath}

Now, from the fact that $\beta$, $\gamma$, $\varepsilon$ and $\eta$ act as zero on $\Mtauarc$, we conclude that the arrows $[\beta\gamma]$, $[\varepsilon\gamma]$, $[\beta\eta]$ and $[\varepsilon\eta]$ of $\tildeqtau$ act as zero on $\overline{\Mtauarc}$. Since the arrows of $\widetilde{\mu}_{\arcone_1}(\qtau)$ not incident to $\arcone_1$ act on $\overline{\Mtauarc}$ in the exact same way they act on $\Mtauarc$, we just have to find out how the arrows $\beta^*$, $\gamma^*$, $\varepsilon^*$ and $\eta^*$ of $\tildeqtau$ act on $\overline{\Mtauarc}$. To this end, we choose the zero retraction $\rh:M_{\oout}\rightarrow\ker\ga=0$ and the zero section $\si:\frac{\ker\al}{\image\ga}\rightarrow\ker\al$. A straightforward check yields that the action of $\beta^*$ and $\varepsilon^*$ is encoded by the matrix
$$
[\beta^* \ \varepsilon^*]=\left[\begin{array}{c|c}- & -\\
-1 & 0 \\
- & -\\
- & - \end{array}\right]:\field\oplus0\rightarrow 0\oplus\field\oplus0\oplus 0,
$$
whereas the arrows $\gamma^*$ and $\eta^*$ act according to the matrix
$$
\left[\begin{array}{c}\gamma^*\\
                      \eta^*\end{array}\right]=
\left[\begin{array}{cccc}
- & - & - & - \\
\hline
- & x & - & -
\end{array}\right]
:0\oplus\field\oplus0\oplus 0\rightarrow 0\oplus\field.
$$

This completes the computation of the action of the arrows of $\widetilde{\mu}_{\arcone_1}(Q(\tau))$ on $\overline{\Mtauarc}$.
We have thus computed the premutation $\widetilde{\mu}_{\arcone_1}(\calMtauarc)=(\widetilde{\mu}_{\arcone_1}(Q(\tau)),\tildestau,$ $\overline{\Mtauarc},0)$.

On the other hand, if we flip the arc $\arcone_1$ of $\tau$ we obtain the ideal triangulation $\sigma=f_{\arcone_1}(\tau)$ sketched in Figure \ref{Fig:punctmonogonnodetoursflipped}  (in a clear abuse of notation, we are using the same symbol $\arcone_1$ in both $\tau$ and $\sigma$).
\begin{figure}[!h]
                \caption{}\label{Fig:punctmonogonnodetoursflipped}
                \centering
                \includegraphics[scale=.3]{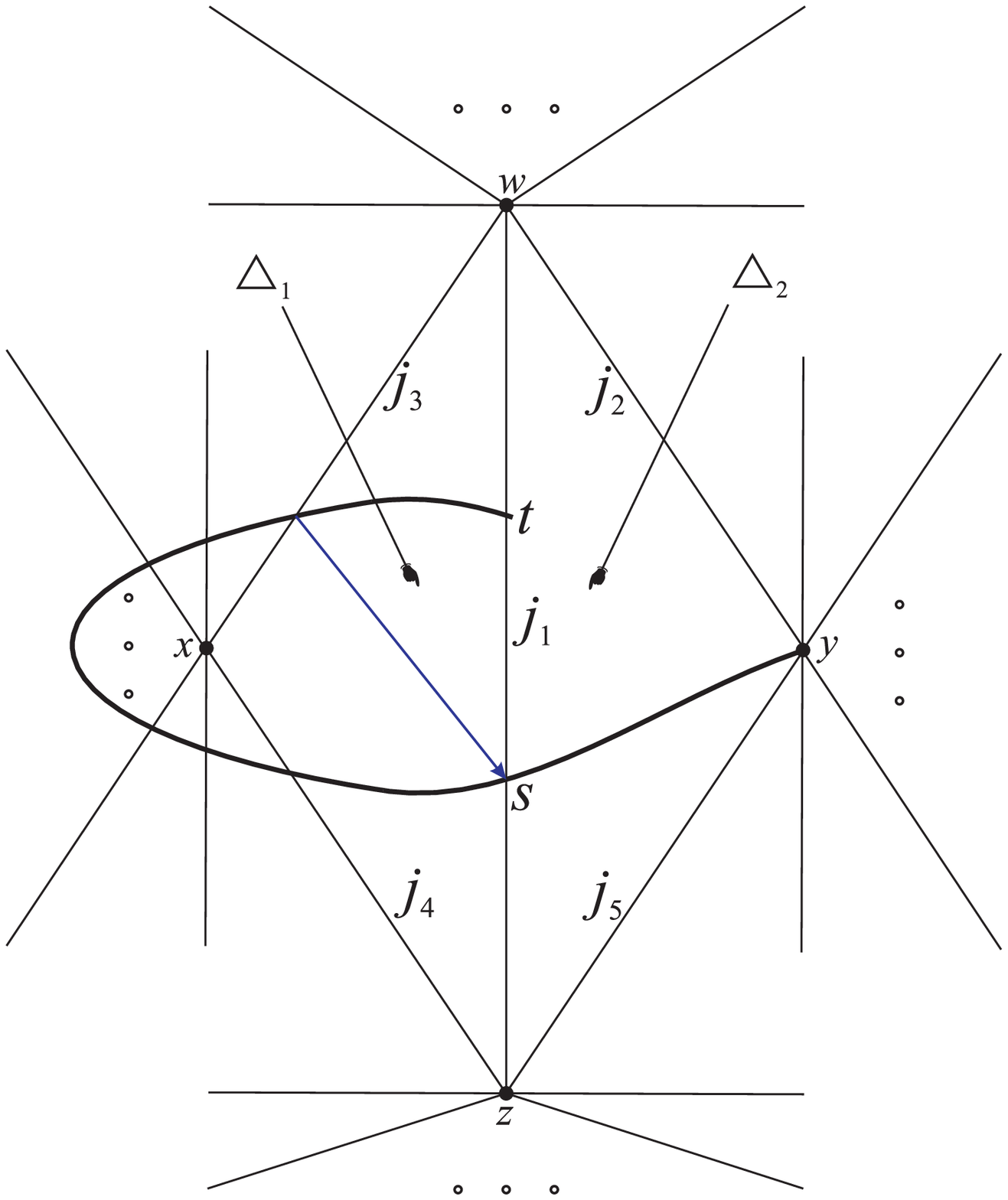}
        \end{figure}\\
The relevant vector spaces attached to the vertices of $Q(\sigma)$ are
$$
N_{\arcone_{1}}=\Msigmaarc_{\arcone_{1}}=\field, \ N_{\arcone_{2}}=\Msigmaarc_{\arcone_2}=0,
$$
$$
N_{\arcone_{3}}=\Msigmaarc_{\arcone_{3}}=\field, \ N_{\arcone_{4}}=\Msigmaarc_{\arcone_4}=\field,
$$
$$
\text{and} \ N_{\arcone_5}=\Msigmaarc_{\arcone_5}=0.
$$
We also have
$$
\badintersection^{\triangle_1,1}_{\arc,\arcone_1}=\{(t,s,b(d^{\triangle_1,1}_{\arc,\arcone_1}),x)\} \ \text{and}
$$
$$
\badintersection^{\triangle_1,r+1}_{\arc,\arcone_1}=\badintersection^{\triangle_2,r}_{\arc,\arcone_2}=
\badintersection^{\triangle_1,r}_{\arc,\arcone_3}=
\badintersection^{\triangle_1,r}_{\arc,\arcone_4}=\badintersection^{\triangle_2,r}_{\arc,\arcone_5}=
\varnothing \ \text{for} \ r\geq 1.
$$

Thus all the relevant detour matrices are hence identities (of the corresponding sizes), except $D^{\triangle_1}_{\arc,\arcone_1}$, which is
$$
D_{\arc,\arcone_1}^{\triangle_1}=
\left[\begin{array}{cc}
1 & 0\\
-x & 1\end{array}\right]
$$
Hence the arrows $\beta^*$, $[\beta\eta]$, $\eta^*$, $\gamma^*$, $\varepsilon^*$ and $[\varepsilon\gamma]$ act on $\Msigmaarc$ according to the following linear maps:
$$
\Msigmaarc_{\beta^*}=(\pi)(D_{\arc,\arcone_1}^{\triangle_1})(\msigmaarc_{\beta^*})=
-x\mathbf{1}:\field\rightarrow\field,
$$
$$
\Msigmaarc_{[\beta\eta]}=
(D_{\arc,\arcone_3}^{\triangle_1})(\msigmaarc_{[\beta\eta]})=
\zero:\field\rightarrow \field,
$$
$$
\Msigmaarc_{\eta^*}=(D_{\arc,\arcone_4}^{\triangle_1})(\msigmaarc_{\eta^*})(\ell)=
\mathbf{1}:
\field\rightarrow \field,
$$
$$
\Msigmaarc_{\gamma^*}=(D_{\arc,\arcone_2}^{\triangle_2})(\msigmaarc_{\gamma^*})(\ell)=
\zero:
\field\rightarrow 0,
$$
$$
\Msigmaarc_{\varepsilon^*}=(\pi)(D_{\arc,\arcone_1}^{\triangle_2})(\msigmaarc_{\varepsilon^*})=
\zero:
0\rightarrow \field,
$$
$$
\Msigmaarc_{[\varepsilon\gamma]}=
(D_{\arc,\arcone_5}^{\triangle_2})(\msigmaarc_{[\varepsilon\gamma]})=
\zero:0\rightarrow 0.
$$

The triple $\Phi=(\varphi,\psi,\eta)$ is a right-equivalence between $\mu_{\arcone_1}(\calMtauarc)$ and $\calMsigmaarc$, where
\begin{itemize}\item $\varphi:(\overline{Q(\tau)},\lambda(\tildestau)-\alpha[\beta\gamma]-\delta[\varepsilon\eta])\rightarrow \qssigma$ is the right-equivalence whose action on the arrows is given by
$$
\beta^*\mapsto-\beta^*, \ \eta^*\mapsto-\eta^*,
$$
and the identity in the rest of the arrows;
\item $\psi:\overline{\Mtauarc}\rightarrow\Msigmaarc$ is the vector space isomorphism given by $\psi_{\arctwo}=\id:\overline{\Mtauarc}_{\arctwo}\rightarrow\Msigmaarc_{\arctwo}$ for $\arctwo\neq\arcone_1$, and
    $\psi_{\arcone_{1}}=-x\id:\overline{\Mtauarc}_{\arcone_1}\rightarrow\Msigmaarc_{\arcone_1}$;
\item $\eta$ is the zero map (of the zero space).
\end{itemize}

This finishes the proof of Theorem \ref{thm:flip<->mut} for configuration 10 of Figure \ref{Fig:component2}.
\end{case}

The rest of the cases are quite similar, and are left to the reader. \label{page:endofproof}
\end{proof}

\begin{coro} The decorated arc representations $\calMtauarc$ are mutation-equivalent to negative simples. More precisely, given any ideal triangulation (without self-folded triangles) $\sigma$ such that $\arc\in\sigma$, then $\calMtauarc$ is mutation-equivalent to the negative simple representation $S^-_\arc(\qsigma,\ssigma)$. Consequently, the Euler-Poincar\'e characteristics of $\Mtauarc$ are the coefficients of the $F$-polynomial that calculates the Laurent expansion of the cluster variable $\arc$ with respect to the cluster $\tau$.
\end{coro}

\begin{proof} One way of proving this corollary is by exhibiting the representation $M(\omega,\arc)$ for each triangulation $\omega$ with self-folded triangles, and showing that Theorem \ref{thm:flip<->mut} remains valid when at least one of the triangulations involved has self-folded triangles. Another way is by showing that
\begin{equation}\label{eq:relatedbyflips}
\text{any two ideal triangulations without self-folded triangles are related by a sequence of flips}
\end{equation}
\begin{center}
in such a way that every triangulation arising in the sequence is an ideal triangulation
\end{center}
\begin{center}
without self-folded triangles.
\end{center}
The fact that any two ideal triangulations are related by a sequence of flips is well-known and has many different proofs (see for example \cite{Mosher}, where an elementary proof is given). Now, it is possible to define a function $f$ from the set of all ideal triangulations to the set of ideal triangulations without self-folded triangles, with the following properties:
\begin{itemize}
\item $f(\tau)=\tau$ whenever the ideal triangulation $\tau$ does not have self-folded triangles;
\item for any two ideal triangulations $\tau$ and $\sigma$, if they are related by a flip, then $f(\tau)$ and $f(\sigma)$ are related by a sequence of flips none of whose intermediate triangulations has self-folded triangles.
\end{itemize}
(Such function can be defined by simultaneously flipping all loops enclosing self-folded triangles.) In this way, if we start with any sequence of flips starting and ending at triangulations without self-folded triangles, by applying $f$ to all the intermediate triangulations of the sequence we get a sequence of flips not involving self-folded triangles at all.
\end{proof}

\subsection{Summary}\label{subsec:summary}

We believe that the reader may profit from this summary of both processes of definition of arc representations and proof of flip $\leftrightarrow$ mutation compatibility.

If it happens that $\arc\in\tau$, then the definition of the decorated arc representation $\calMtauarc$ is trivial: it is the $\arc^{\operatorname{th}}$ negative simple representation of $\qstau$. Otherwise there are two cases, the second one of which is when $\arc$ is a loop cutting out a once-punctured monogon, situation where sometimes we cut off a segment of $\arc$. In any case, there always is a conserved segment $\iota$.

Having made the necessary cuts (if any), we draw the detours of $(\tau,\arc)$ in a recursive fashion: first the 1-detours, then the 2-detours, etc. Each detour $d^n$ detours one puncture $p$, which falls within a disk surrounded by a curve formed by a segment of $\arc$, a segment of an arc in $\tau$ opposite to $p$ and one or two detours (one in the first step of the recursion, two in the later steps), one of which is $d^n$.

For every arc $\arcone$ in the triangulation $\tau$ we define two detour matrices, attaching one to each triangle of $\tau$ containing $\arcone$. These are square matrices (of the same size) whose entries are defined in terms of the detours ending at $\arcone$ and the punctures detoured by these detours.

The segment representation $\mtauarc$ is obtained by applying the simple idea of placing a copy of $\field$ at each intersection point of $\iota$ with $\tau$ and putting the identity map between two such copies if we can go from one to the other using a segment of $\iota$ entirely contained in a triangle of $\tau$.

Using the detour matrices, we ``twist" the action of the arrows of $\qtau$ on $\mtauarc$. The ``twist" is done by first acting with the arrow $a$ and then applying the detour matrix attached to the triangle that contains the head of $h(a)$. The resulting representation is the arc representation $\Mtauarc$, which we decorate with the zero space.

These are the construction ingredients. Now let us turn to the strategy of the proof. There are two simplifying tools: The operation of restriction of a QP or a QP-representation to a vertex subset, which, thanks to the fact that it commutes with right-equivalences, reductions and premutations, narrows the proof of the main theorem to the case of surfaces with empty boundary, since the QP and arc representations of any triangulation of a surface (with boundary) are restrictions of QPs and arc representations of a surface without boundary.

The other simplifying tool is the ``local decomposability" of representations, which, due to the somewhat simple nature of the potential $\stau$, helps us to narrow the configurations which need to be analyzed in the proof of the main theorem. Unfortunately, these tools were not powerful enough to give a uniform such proof.

\section{An application: $\mathbf{g}$-vectors for the positive stratum}\label{Section:gvectors}

This section is devoted to a small application of our arc representations in the cluster algebra context. We assume that the reader is familiar with the \emph{tagged arc} $\leftrightarrow$ \emph{cluster variable} identification established in \cite{FST}.

For each ideal triangulation $\tau$ of a surface $\surf$ let $B(\tau)=(b_{ij}^\tau)_{ij}$ denote its signed-adjacency matrix (cf. \cite{FST} or \cite{Lqps}). Remember the relation between $B(\tau)$ and $\qtau$: the vertices of $\qtau$ are the arcs in $\tau$, with $b_{ij}^\tau$ arrows from $i$ to $j$ whenever $b_{ij}^\tau>0$.

Let $n$ be the \emph{rank} of $\surf$, that is, the number of arcs in any ideal triangulation of $\surf$ (cf. \cite{FST}). Fix an ``initial" ideal triangulation $\tau=\tau^0=\{\arcone_1,\ldots,\arcone_n\}$ of a surface $\surf$. In \cite{FZ4}, S. Fomin and A. Zelevinsky introduce a $\Z^n$-grading for $\Z[\arcone_1^{\pm1},\ldots,\arcone_n^{\pm1},y_1,\ldots,y_n]$ defined by the formulas
$$
\deg(\arcone_l)=\e_l, \ \ \text{and} \ \ \deg(y_l)=-\mathbf{b}_l,
$$
where $\e_1,\ldots,\e_n$ are the standard basis (column) vectors in $\Z^n$, and $\mathbf{b}_l=\sum_kb^{\tau}_{kl}\e_k$ is the $l^{\operatorname{th}}$ column of $B(\tau)=B(\tau^0)$. Under this $\Z^n$-grading, the principal coefficient cluster algebra $\mathcal{A}_{\bullet}(B(\tau))$ is a $\Z^n$-graded subalgebra of $\Z[\arcone_1^{\pm1},\ldots,\arcone_n^{\pm1},y_1,\ldots,y_n]$ and all cluster variables in $\mathcal{A}_{\bullet}(B(\tau))$ are homogeneous elements (cf. \cite{FZ4}, Proposition 6.1 and Corollary 6.2). By definition, the \emph{$\g$-vector} $\g_\arc^{\tau}$ of a cluster variable $\arc\in \mathcal{A}_{\bullet}(B(\tau))$ with respect to the ``initial" triangulation $\tau=\tau^0$ is its multi-degree with respect to the $\Z^n$-grading just defined. Fomin-Zelevinsky have shown in \cite{FZ4} that the mutation dynamics inside cluster algebras are controlled to an amazing extent by $\g$-vectors and \emph{$F$-polynomials}.

In \cite{DWZ2}, H. Derksen, J. Weyman and A. Zelevinsky have given a representation-theoretic interpretation of $\g$-vectors using the mutation theory of quivers with potentials as follows. Let $\arc$ be an (ordinary) arc on $\surf$. As seen in Section \ref{backgroundrepresentations}, for each arc $\arcone\in\tau=\tau^0$ the decorated representation $\mathcal{M}(\tau,\arc)=(\qtau,\stau,\Mtauarc,V(\tau,\arc))$ induces a linear map $\ga_{\arcone}:M(\tau,\arc)_{\arcone,\oout}\rightarrow M(\tau,\arc)_{\arcone,\oin}$ (in Section \ref{backgroundrepresentations} we did not use the subscript $\arcone$). In Theorem 31 of \cite{Lqps} it is proved that if $\partial\surfnoM\neq\varnothing$, then the QP $\qstau$ is non-degenerate. Therefore, combining Theorem \ref{thm:flip<->mut} above with Equations (1.13), (5.2), and Theorem 5.1 of \cite{DWZ2}, we see that the $\arcone^{\operatorname{th}}$ entry of the $\g$-vector $\g_\arc^{\tau}$ is
\begin{equation}
\g_{\arc,\arcone}^{\tau}=\dim\ker\ga_{\arcone}-\dim M(\tau,\arc)_{\arcone}+\dim V(\tau,\arc)_{\arcone}
\end{equation}
provided the underlying surface $\surfnoM$ has non-empty boundary. (According to Conjecture 33 of \cite{Lqps}, the assumption $\partial\surfnoM\neq\varnothing$ is superfluous.)

Assume that $\arc$ is not a loop cutting out a once-punctured monogon. Let $\diamondsuit$ be the quadrilateral of $\tau$ whose diagonal is $\arcone$. The connected components of the intersection $\arc\cap\diamondsuit$ are segments of $\arc$, each of which falls within one of the types described in Figure \ref{Fig:typesofsegmentsforgvectors}.
\begin{figure}[!h]
                \caption{}\label{Fig:typesofsegmentsforgvectors}
                \centering
                \includegraphics[scale=.1]{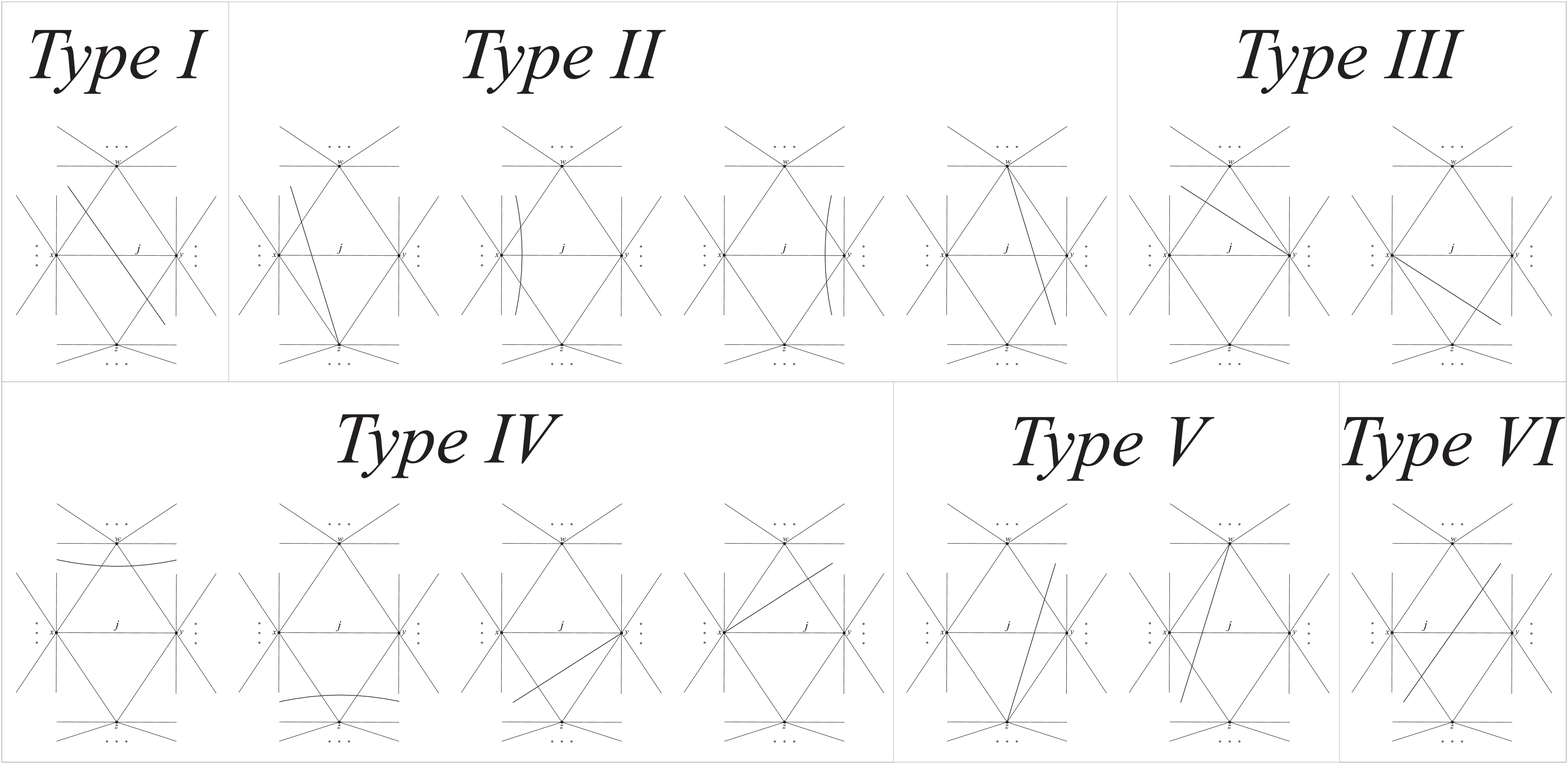}
        \end{figure}
Let $\s(\arc,\arcone)$ (resp. $\rr(\arc,\arcone)$, $\ttt(\arc,\arcone)$, $\vv(\arc,\arcone)$, $\z(\arc,\arcone)$) be the number of components of $\arc\cap\diamondsuit$ that fall within type I (resp. II, III, V, VI) in Figure \ref{Fig:typesofsegmentsforgvectors}.

\begin{thm} Under the assumptions and notation just stated, the $\arcone^{\operatorname{th}}$ entry of the $\g$-vector $\g_\arc^{\tau}$ is
$\g_{\arc,\arcone}^{\tau}=\s(\arc,\arcone)+\ttt(\arc,\arcone)-\vv(\arc,\arcone)-\z(\arc,\arcone)+\delta_{\arc\arcone}$ (the \emph{Kronecker delta}).
\end{thm}

\begin{proof} Clearly, the theorem follows if we establish the equalities
$$
\dim\ker\ga_{\arcone}=2\s(\arc,\arcone)+\rr(\arc,\arcone)+\ttt(\arc,\arcone), \ \
\dim\Mtauarc_\arcone=\s(\arc,\arcone)+\rr(\arc,\arcone)+\vv(\arc,\arcone)+\z(\arc,\arcone).
$$
By the observations made in Subsection \ref{subsec:localdecompositions}, it is enough to check this identities for each of the possible direct summands of the representation $M(\partial)$, which are displayed in the configurations shown in Figures \ref{Fig:component1}, \ref{Fig:component1punctmonogon}, \ref{Fig:component2}, \ref{Fig:component3} and \ref{Fig:component3punctmonogon}. This is straightforward.
\end{proof}

\begin{remark} We warn the reader not to confuse the segments in Figure \ref{Fig:typesofsegmentsforgvectors} with the connected components of the graph $G(\partial)$ introduced in Subsection \ref{subsec:localdecompositions}: In general, many different segments from Figure \ref{Fig:typesofsegmentsforgvectors}, even of different types, can belong to the same connected component $H_l$ of $G(\delta)$.
\end{remark}

\begin{ex} With respect to the triangulations of Figures \ref{curvepuncthexagonmotivating} and \ref{curvepuncthexagonmotivatingwithdetour}, the arc $\arc$ shown there has the following $\g$-vectors:
\begin{displaymath}
\g_{\arc}^\tau=\left[\begin{array}{ccc}
 & 0 & \\
0 &  & 1\\
0 &  & 0\\
 & -1 &
\end{array}\right], \ \ \
\g_{\arc}^\sigma=\left[\begin{array}{cccc}
  & 1 &   &   \\
0 &   &   &   \\
  &   & -1 & 1 \\
0 &   &   &   \\
  & -1 &   &
\end{array}\right]
\end{displaymath}
\end{ex}

\section{Some problems}\label{Section:problems}

There are some problems whose solution the author thinks would help to have a full application of Derksen-Weyman-Zelevinsky's representation-theoretic approach to cluster algebras.

\begin{problem}\label{Problem:findim?} In the case of surfaces with no boundary, determine whether the Jacobian algebras of its triangulations are finite-dimensional or not.
\end{problem}

To ``ilustrate" Problem \ref{Problem:findim?} we include a curious example that came up in discussions with Jerzy Weyman.

\begin{ex}[cf. \cite{Lqps}, Example 35]  Consider the ``canonical" triangulation $\tau$ of the once-punctured torus shown in Figure \ref{torusqp}.
        \begin{figure}[!h]
                \caption{}\label{torusqp}
                \centering
                \includegraphics[scale=.65]{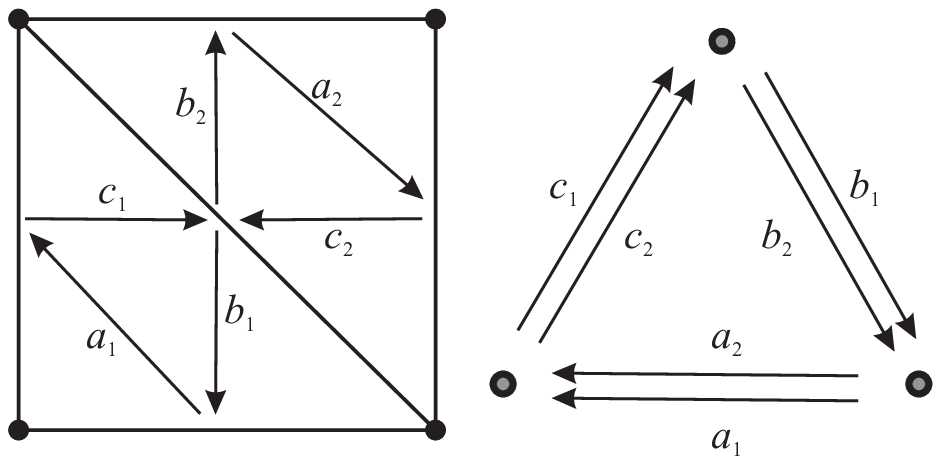}
        \end{figure}
We have $\stau=a_1b_1c_1+a_2b_2c_2+xa_1b_2c_1a_2b_1c_2$. It is readily seen that flips of (ideal) triangulations of this surface are compatible with QP-mutations and that the potential is therefore non-degenerate (it is not possible to obtain a tagged non-ideal triangulation from an ideal one by flips). Also, it is easy to see that the quotient $R\langle Q(\tau)\rangle/I$ is infinite-dimensional, where $I$ is the ideal generated by the cyclic derivatives of $S(\tau)$. However, in contrast to Example 11.3 of \cite{DWZ}, where the non-degenerate potential $a_1b_1c_1+a_2b_2c_2$ has infinite-dimensional Jacobian algebra, the Jacobian algebra $\mathcal{P}(\qtau,\stau)$ is finite-dimensional, as the following basic calculation shows.

First, notice that the paths $a_1b_2c_1a_2b_1c_2$ and $a_2b_1c_2a_1b_2c_1$ represent the same element in $\mathcal{P}(\qtau,\stau)$:
$$
(a_1b_2c_1a_2b_1)c_2\equiv -x^{-1}a_2b_2c_2    \equiv a_2(b_1c_2a_1b_2c_1) \ \operatorname{mod} \ J,
$$
where $J$ is the two-sided ideal of $R\langle\langle\qtau\rangle\rangle$ generated by the cyclic derivatives of $\stau$ (hence the Jacobian ideal $J(\stau)$ is the topological closure of $J$). Notice also that any path in $\qtau$
can be represented in $\mathcal{P}(\qtau,\stau)$ by (a scalar multiple of) an \emph{alternating path}, that is, a path whose conforming arrows have subindices that alternate between the numbers 1 and 2. Here is the calculation for paths of length 2:
$$
a_1b_1\equiv -xa_2b_1c_2a_1b_2 \ \operatorname{mod} \ J, \ \ \text{and similarly for the rest of the paths} \ b_1c_1, \ c_1a_1, \ a_2b_2, \ b_2c_2, \ c_2a_2.
$$

Now the calculation for paths of length 3. For the paths $a_1b_1c_1$, $a_2c_2b_2$ and their rotations it is essentially shown above, whereas for the paths of the form $A_1B_2C_2$ we have
$$
A_1B_2C_2\equiv -xA_1B_1C_2A_1B_2C_1
$$
$$
\equiv x^2 A_2B_1C_2A_1(B_2C_2)A_1B_2C_1
$$
$$
\equiv -x^3A_2B_1C_2(A_1B_1)C_2A_1B_2C_1A_1B_2C_1\cong \ldots \ \operatorname{mod} J,
$$
from what we see that $A_1B_2C_2\in J(\stau)$. Similarly, $A_1B_1C_2,A_2B_1C_1,A_2B_2C_1\in J(\stau)$.

In length 4 it already happens that the only paths that are not zero in $\mathcal{P}(\qtau,\stau)$ are the alternating ones:
$A_1B_1C_1A_1\equiv x(A_2B_1C_2A_1B_2)C_1A_1\in J(\stau)$ and the rest is an easy check.

Now we claim that all the paths of length 7 are zero in $\mathcal{P}(\qtau,\stau)$. After all the above calculations it is clear that we only need to check that the alternating paths of length 7 belong to $J(\stau)$. But
$$
(A_1B_2C_1A_2B_1C_2)A_1 \equiv A_2B_1C_2A_1B_2(C_1A_1)
$$
$$
\equiv -xA_2B_1C_2A_1(B_2C_2)A_1B_2C_1A_2
$$
$$
\equiv x^2 A_2B_1C_2(A_1B_1)C_2A_1B_2(C_1 A_1)B_2C_1A_2\equiv\ldots \ \operatorname{mod} J,
$$
and hence $A_1B_2C_1A_2B_1C_2A_1\in J(\stau)$. Therefore, all paths of length greater than 6 belong to $J(\stau)$, which implies the finite-dimensionality of $\mathcal{P}(\qtau,\stau)$.
\end{ex}

\begin{problem} Prove or disprove that, in the case of surfaces with empty boundary, the potentials defined in \cite{Lqps} for ideal triangulations are non-degenerate.
\end{problem}

\begin{problem}\label{Problem:potsforgeneraltriang} Extend the combinatorial recipe for $\stau$ given in \cite{Lqps} when $\tau$ is an ideal triangulation to the general situation of tagged triangulations.
\end{problem}

Notice that Problem \ref{Problem:potsforgeneraltriang} is still open even when the underlying surface has non-empty boundary (and at least one puncture), despite the fact that in such situation the non-degeneracy of the potentials is already proved.

\begin{problem} Extend the definition of $\Mtauarc$ to the general situation where $\tau$ is a tagged triangulation and $\arc$ is a tagged arc.
\end{problem}

\begin{problem} Give a cell or CW decomposition of the quiver Grassmannians of the arc representations defined above.
\end{problem}

It would be interesting to find some sort of ``dictionary" between Musiker-Schiffler-Williams' framework and ours, and use it, for example, to refine their calculation of Euler-Poincar\'e characteristics of the quiver Grassmannians of arc representations by giving a cell or CW decomposition of these varieties, or to give a combinatorial recipe for the potentials and representations of triangulations in the negative strata.

Finally, let us state a couple of challenging problems that have motivated \cite{Lqps} and the present work:

\begin{problem} Give a combinatorial recipe to calculate non-degenerate potentials for arbitrary quivers. If possible, in such a way that given two mutation-equivalent quivers, the potentials calculated by the recipe are QP-mutation-equivalent.
\end{problem}

\begin{problem} Give a combinatorial recipe that calculates the decorated representations of arbitrary non-degenerate QPs that are mutation-equivalent to negative simples, without performing any mutation.
\end{problem}

\begin{remark} In \cite{DWZ2}, Derksen-Weyman-Zelevinsky have proved several conjectures from \cite{FZ4} without needing to give such combinatorial recipes. In that same paper, a conjectural characterization of representations mutation-equivalent to negative simples is given in terms of the vanishing of the \emph{E-invariant}.
\end{remark}



\section*{Acknowledgments}

This paper is part of my PhD thesis, which I am writing under the supervision of Andrei Zelevinsky. I am sincerely grateful to him for his guidance and permanent availability. To Jerzy Weyman as well, for raising many interesting questions that have helped me to improve my understanding of Jacobian algebras.

I also thank Michael Barot, Jos\'e Antonio de la Pe\~na, Belen Fragoso, Christof Geiss, Martha Takane and Manuel Vargas, for their valuable support during my graduate studies.


\end{document}